\documentclass[12pt]{article}
\usepackage[]{amsmath,amssymb}
\usepackage{amscd}
\usepackage{latexsym}
\usepackage{cite}
\usepackage{amsthm}

\renewcommand{\labelenumi}{(\theenumi)}
\renewcommand{\theenumi}{\roman{enumi}}
\newtheorem{definition}{Definition}[section]
\newtheorem{theorem}[definition]{Theorem}
\newtheorem{lemma}[definition]{Lemma}
\newtheorem{corollary}[definition]{Corollary}
\newtheorem{remark}[definition]{Remark}
\newtheorem{example}[definition]{Example}

\newtheorem{note}[definition]{Note}
\newtheorem{assumption}[definition]{Assumption}
\newtheorem{proposition}[definition]{Proposition}

\begin{document}
\title{\bf  
Notes on the Leonard system classification
}
\author{
Paul Terwilliger}
\date{}

\maketitle
\begin{abstract}
Around 2001 we classified the
Leonard systems up to isomorphism. The proof was lengthy 
and involved considerable computation. In this paper
we give a proof that is shorter and involves minimal
 computation. We also give a comprehensive description
 of the intersection numbers of a Leonard system.

 \end{abstract}

\section{Introduction}
In the area of Algebraic Combinatorics there is an object
called a
commutative association scheme 
\cite{bannaiIto},
\cite{mt}.
This is a combinatorial generalization of
a finite group, that retains enough of the group structure so that
one can still speak of the character table.
During the decade 1970--1980 it was realized by
E. Bannai, P. Delsarte, D. Higman and others
that commutative association schemes help to
unify many aspects of
group theory, coding theory,
and design theory.
An early work in this direction was
the 1973 thesis of  Delsarte \cite{delsarte}.
This thesis helped to motivate the work of Bannai \cite[p.~i]{bannaiIto}, 
who taught a series of graduate courses
on commutative association schemes 
during September 1978–December 1982 
at the Ohio State University. 
The lecture notes from those courses, along with
more recent developments, became
a book coauthored with T. Ito 
\cite{bannaiIto}.
The book had a large impact; it is currently
cited 698 times according to MathSciNet.
\medskip

\noindent In the Introduction to
\cite{bannaiIto}, Bannai and Ito describe the goals of their book.
One goal was to summarize what is known
about commutative association schemes up to that
time.
Another goal  was to focus the reader's attention
on two remarkable types of
 schemes, said to be $P$-polynomial and $Q$-polynomial.
A $P$-polynomial scheme is essentially the same thing
as a distance-regular graph, and can be viewed as a finite
analog of a 2-point homogeneous space 
\cite{wang}.
Similarly, a $Q$-polynomial scheme is a finite analog of
a rank 1 symmetric space 
\cite{wang}.
By a
theorem of H.~C.~Wang
\cite{wang},
a compact Riemannian manifold is
2-point homogeneous if and only if it is rank 1 symmetric.
This
result was extended
to the noncompact case by
J.~Tits \cite{tits}
and S.~Helgason 
\cite{helgason}.
Motivated by all this,
Bannai and Ito conjectured  that
a primitive association scheme
is $P$-polynomial if and only if it is $Q$-polynomial,
provided that the diameter is sufficiently large
\cite[p.~312]{bannaiIto}.
They also proposed the classification of
schemes that are both $P$-polynomial and
$Q$-polynomial
\cite[p.~xiii]{bannaiIto}.
\medskip

\noindent
Progress on the proposed classification
was made while the book was still in preparation.
A $P$-polynomial scheme gets its name from the
fact that there exists a sequence of
orthogonal polynomials $\lbrace u_i\rbrace_{i=0}^d$
such that $u_i(A)=A_i/k_i$ for $0 \leq i \leq d$,
where $d$ is the diameter of the scheme,
$A_i$ is the $i$th associate matrix, $A=A_1$,
and $k_i$ is the $i$th valency \cite[pp.~190,~261]{bannaiIto}.
Similarly, for a $Q$-polynomial scheme
there exists a sequence of orthogonal polynomials
$\lbrace u^*_i\rbrace_{i=0}^d$
such that $u^*_i(A^*)=A^*_i/k^*_i$ for $0 \leq i \leq d$,
where $A^*_i$ is the $i$th dual associate matrix,
$A^*=A^*_1$,
and $k^*_i$ is the $i$th  dual valency
\cite[pp.~193,~261]{bannaiIto},
\cite[p.~384]{tersub1}.
For schemes that are 
$P$-polynomial and $Q$-polynomial,
we have $u_i(\theta_j)= 
u^*_j(\theta^*_i)
$ for $0 \leq i,j\leq d$,
where
$\lbrace \theta_i\rbrace_{i=0}^d$
(resp.
$\lbrace \theta^*_i\rbrace_{i=0}^d$
)
are the eigenvalues of $A$ (resp. $A^*$)
\cite[p.~262]{bannaiIto}.
These equations are known as
Delsarte duality
\cite{leonard}
or Askey-Wilson duality
\cite[p.~261]{ter2004}. This duality can be defined for $d=\infty$,
but throughout this paper we assume that $d$ is finite.
\medskip

\noindent 
Askey-Wilson duality comes up naturally
in the area of special functions and orthogonal polynomials.
In this area the classical orthogonal polynomials
are often described using a partially-ordered set  called the Askey-tableau.
The vertices in the poset represent the various families
of classical orthogonal polynomials, 
and the covering relation describes what
happens when a limit is taken.
See
\cite{labelle} for an early version of the tableau,
and
\cite[pp.~183,~413]{koekoek} for a more recent version.
One branch of the tableau, sometimes called the terminating
branch, contains the polynomials that are
orthogonal with respect to a measure that is
nonzero at finitely many arguments.
At the top of this terminating branch sit the
 $q$-Racah polynomials, introduced in 1979 by R. Askey
 and J. Wilson \cite{aw}.
The rest of the terminating branch consists of
the $q$-Hahn, dual $q$-Hahn,
$q$-Krawtchouk, 
dual $q$-Krawtchouk,
quantum $q$-Krawtchouk,
affine $q$-Krawtchouk,
Racah, Hahn, dual-Hahn, and Krawtchouk polynomials.
The above-named polynomials are defined using
hypergeometric series or basic hypergeometric series,
and it is transparent from the definition that they
satisfy 
Askey-Wilson duality.
\medskip

\noindent Back at Ohio State, there was a graduate student
attending Bannai's classes by the name of Douglas Leonard.
With Askey's encouragement, Leonard 
showed in
\cite{leonard}
that the $q$-Racah polynomials give the most general
 orthogonal polynomial system 
that satisfies  Askey-Wilson duality, under the assumption that $d\geq 9$.
In 
\cite[Theorem~5.1]{bannaiIto}
Bannai and Ito give a version of
Leonard's theorem that removes the assumption on $d$ 
and explicitly describes all the limiting cases
that show up.
This version gives a complete classification
of the orthogonal polynomial systems that satisfy
Askey-Wilson duality.
It shows that the orthogonal polynomial systems that satisfy
Askey-Wilson duality are 
from the terminating branch of the Askey-tableau, except for
one family with
 $q=-1$  now called the Bannai-Ito polynomials 
\cite[p.~271]{bannaiIto},
\cite[Example~5.14]{ter2005}.
In our view, the terminating branch of the Askey-tableau
should include
the Bannai-Ito polynomials. Adopting this view,
for the rest of this paper
we include the Bannai-Ito polynomials in the terminating branch
of the Askey-tableau.
\medskip

\noindent The Leonard theorem
\cite[Theorem~5.1]{bannaiIto} is notoriously complicated; the statement
 alone takes 11 pages. In an effort to
simplify and clarify the theorem, the present author
introduced the notion of a Leonard pair
\cite[Definition~1.1]{2lintrans} and Leonard system
\cite[Definition~1.4]{2lintrans}.
 Roughly speaking, a Leonard pair consists
of two diagonalizable linear transformations
on a finite-dimensional
vector space, each acting on the eigenspaces of the other one
in an irreducible tridiagonal fashion; 
see Definition 2.1 below.
A Leonard system is essentially a Leonard pair, 
together with appropriate orderings of their eigenspaces; see 
Definition 2.3 below. 
In 
\cite[Theorem~1.9]{2lintrans} the Leonard systems are classified up
to isomorphism.
This classification is related to Leonard's theorem as follows.
In  \cite[Appendix~A]{2lintrans}
and \cite[Section~19]{ter2004}, a bijection is given
between the isomorphism classes of Leonard systems over $\mathbb R$,
and the orthogonal polynomial systems that satisfy Askey-Wilson duality.
Given the bijection,
the classification \cite[Theorem~1.9]{2lintrans} 
becomes  a `linear-algebraic version' of Leonard's theorem.
This version is conceptually simple and quite elegant in our view.
In \cite{madrid}   
we start with the Leonard pair axiom
and derive, in a uniform and attractive manner,
the polynomials in the terminating branch of the Askey-tableau,
along with their properties 
such as the 3-term recurrence, difference equation, Askey-Wilson
duality, and orthogonality.
\medskip

\noindent We comment on how the
theory of Leonard pairs and Leonard systems depends on
the choice of ground field.
The classification  \cite[Theorem~1.9]{2lintrans}
shows that 
the ground field does not matter in a substantial way, unless 
it has characteristic 2.
In this case, the theory admits an additional family of polynomials
called the orphans.
The orphans have diameter $d=3$ only;
they are described in
\cite[Example~5.15]{ter2005} and Example 20.13 below.
\medskip

\noindent  The book 
\cite{bannaiIto} appeared in 1984 and 
the paper \cite{2lintrans} appeared 
in 2001. It is natural to ask what happened in 
between. 
The concept of a Leonard pair over $\mathbb R$ appears in
\cite[Definitions~1.1,~1.2]{ter1987}, where it is called a thin Leonard pair. 
Also appearing in
\cite{ter1987} 
is the correspondence between
Leonard pairs over $\mathbb R$ and the orthogonal polynomial systems that
satisfy Askey-Wilson duality. In addition
\cite[Definition~3.1]{ter1987} 
describes an algebra called the Leonard algebra,
now known as the Askey-Wilson
algebra.
The paper
\cite{ter1987}  was submitted but never published.
In \cite{zhedhidden} A. Zhedanov  introduced the Askey-Wilson
algebra. This algebra has a presentation involving two
generators $K_1$, $K_2$ that satisfy a pair of quadratic relations.
In
\cite{zhedmut},   Granovskii, Lutzenko, and Zhedanov
consider a finite-dimensional irreducible module for 
the Askey-Wilson algebra, on which each of 
$K_1$, $K_2$ are diagonalizable. 
Under some minor assumptions, 
they show that each of $K_1$, $K_2$ acts in an irreducible tridiagonal
fashion on an eigenbasis for the other one. 
In hindsight, it is fair to say that they constructed an example
of a Leonard pair, although they did not define a Leonard pair
as an independent concept.
The paper \cite[Theorem~2.1]{tersub1} contains
a version of the Leonard pair concept that  is close to
\cite[Definition 1.1]{2lintrans}. 
\medskip

\noindent Turning to the present paper, we obtain two main results:
(i) an improved proof for the classification
of Leonard systems; (ii) a comprehensive description of
the intersection numbers of a Leonard system.
\medskip

\noindent We now describe our main results in detail.
In 
\cite[Theorem 1.9]{2lintrans} the Leonard systems are
classified up to isomorphism, and the given proof is completely correct
as far as we know. However the proof is longer than necessary.
In the roughly two decades since the
paper was published, we have discovered some `shortcuts' 
that simplify the proof and avoid certain tedious calculations.
The shortcuts are summarized as follows.
\begin{enumerate}
\item[$\bullet$]
In 
\cite[Section~3]{2lintrans} we established the split canonical form
for a Leonard system. In the present paper we make use of the
fact that the split
canonical form still exists under weaker assumptions; these are
described in
 Proposition \ref{lem:tighter} below. 
\item[$\bullet$]
The concept of a normalizing idempotent was introduced by
Edward Hanson in 
\cite[Section~6]{hanson}. In the present paper we use this
concept to simplify numerous arguments; see Sections 6, 7, 17 below.
\item[$\bullet$]
 In 
\cite[Theorem~4.8]{2lintrans}
we explicitly gave the matrix entries for a certain matrix representation
of the primitive idempotents for a Leonard system. The computation
of these entries is tedious and takes up most of 
\cite[Section~4]{2lintrans}. In the present paper we replace all
of this by a single identity
(\ref{eq:view2}) that is established in a few lines.
\item[$\bullet$]
 In the present paper we use an antiautomorphism
$\dagger$ to obtain the result
 Proposition
\ref{prop:threefourths}, which is roughly summarized as follows:
as we construct a Leonard system, if we construct three-fourths
of it then the last fourth comes for free.
\item[$\bullet$]
 In  
\cite[Lemma~7.2]{2lintrans} we used a slightly obscure method to
establish the irreducibility of the underlying module for
a Leonard system. In the present paper this lemma is avoided
using the first bullet point above.
\item[$\bullet$]
We replace the slightly technical results
\cite[Lemmas~10.3--10.5]{2lintrans} 
by a more elementary result,  Proposition
\ref{note:varthetacharacconvS99} below.
\item[$\bullet$]
We replace
most of \cite[Section~11]{2lintrans} by
a single result Proposition
\ref{prop:dg2}
 called the wrap-around result. 
The wrap-around result 
was discovered by T. Ito and the present author
during our effort to classify
the tridiagonal pairs; it is the essential idea behind
the proof of
\cite[Lemma~9.9]{tdqrac}. 
\item[$\bullet$]
 Using the  improvements listed above, we replace 
the arguments in \cite[Sections~13,~14]{2lintrans} with 
more efficient arguments in Section 17 below.
\end{enumerate}
Some parts of the improved proof are unchanged from the original;
we still use
\cite[Sections 8, 9]{2lintrans}  and
\cite[Lemmas~10.2, 12.4]{2lintrans}.
These results are reproduced in the present paper in order to
obtain a complete proof, all in one place. We believe that
this complete proof is suitable for The Book if not this journal.
\medskip

\noindent Concerning our second main result,
we mentioned earlier that the Leonard systems correspond to
the orthogonal polynomial sequences that satisfy Askey-Wilson duality.
Unfortunately, it is a bit difficult to go back and
forth between the two points of view, because from the
polynomial perspective, the main parameters are the intersection
numbers  (or connection coefficients) that describe the 3-term 
recurrence, and from the Leonard system perspective, the main
parameters are the first and second split sequence that make
up part of the parameter array.
There are some equations that relate the two types of parameters; see
\cite[Theorem~17.7]{ter2004} 
and Lemma 19.4 below. However the nonlinear nature of these equations
 makes them difficult to use. In order to mitigate the difficulty,
we display many identities that involve
 the intersection numbers along with the
  first and second split sequence.
Taken together, these identities should
 make it easier to work with the intersection numbers in the future.
These identities can be found in Section 19.
We also explicitly give the intersection numbers for every isomorphism
class of Leonard system; these are contained in the Appendix.
\medskip

\noindent The paper is organized as follows. Sections 2, 3  contain
 preliminary comments and definitions.
In Section 4 we describe the antiautomorphism $\dagger$ and
use it to obtain Proposition 
\ref{prop:threefourths}.
In Section 5 we describe some polynomials that will be used
throughout the paper.
In Section 6 we discuss the concept of a normalizing idempotent.
In Section 7 we use normalizing idempotents to describe
certain kinds of decompositions relevant to Leonard systems.
Section 8 contains the wrap-around result.
In Section 9 we recall the parameter array of a Leonard system.
In Section 10 we state the Leonard system 
classification, which is Theorem
\ref{thm:classification}.
Sections 11--13 are about recurrent sequences.
In Section 14 we define a polynomial in two variables that will be
useful on several occasions later in the paper.
Sections 15, 16 are about the tridiagonal
relations.
In Section 17 we complete the proof of Theorem
\ref{thm:classification}.
Section 18 contains two characterizations related to
Leonard systems and parameter arrays.
In Section 19 we give a comprehensive treatment of the intersection
numbers of a Leonard system. These intersection numbers are
listed in the Appendix.

\section{Preliminaries}

\noindent We now begin our formal argument.
Shortly we will define a Leonard pair and Leonard
system. 
Before we get into detail, we briefly review some
notation and basic concepts.
Let $\mathbb F$ denote a field. Every vector space and algebra 
discussed in this paper is understood to be over $\mathbb F$.
Throughout the paper fix an integer $d\geq 0$.
Let ${\rm Mat}_{d+1}(\mathbb F)$ denote the algebra consisting
of the $d+1$ by $d+1$ matrices that have all entries in $\mathbb F$.
We index the rows and columns by $0,1,\ldots, d$.
Throughout the paper $V$ denotes a vector space with dimension $d+1$.
Let ${\rm End}(V)$ denote the algebra consisting of the $\mathbb F$-linear
maps from $V$ to $V$.
Next we recall how each basis $\lbrace v_i\rbrace_{i=0}^d$ of $V$ gives
an algebra isomorphism
${\rm End}(V) \to
{\rm Mat}_{d+1}(\mathbb F)$.
For $X \in {\rm End}(V)$ and $M \in 
{\rm Mat}_{d+1}(\mathbb F)$, we say that
{\it $M$
represents $X$
with respect to 
$\lbrace v_i\rbrace_{i=0}^d$}
whenever $Xv_j = \sum_{i=0}^d M_{ij}v_i$ for $0 \leq j \leq d$.
The isomorphism 
sends $X$ to the unique matrix
in 
${\rm Mat}_{d+1}(\mathbb F)$
that represents $X$ with respect to 
$\lbrace v_i\rbrace_{i=0}^d$.
A matrix  $M \in {\rm Mat}_{d+1}(\mathbb F)$
is called {\it tridiagonal} whenever
each nonzero entry lies on either the diagonal, the subdiagonal,
or the superdiagonal. 
 Assume that $M$ is tridiagonal. Then $M$
is called {\it irreducible} whenever each entry on the
subdiagonal is nonzero, and each entry on the superdiagonal is
nonzero.

\begin{definition} 
\label{def:lp} \rm
{\rm (See \cite[Definition~1.1]{2lintrans}).}
By a {\it Leonard pair on $V$}
we mean an ordered pair $A, A^*$ of elements in 
${\rm End}(V)$ such that:
\begin{enumerate}
\item[\rm (i)] 
there exists a basis for $V$ with respect to which the
matrix representing $A$ is diagonal and the
matrix representing $A^*$ is irreducible tridiagonal;
\item[\rm (ii)] 
there exists a basis for $V$ with respect to which the
matrix representing $A^*$ is diagonal and the
matrix representing $A$ is irreducible tridiagonal.
\end{enumerate}
The Leonard pair $A,A^*$ is said to be {\it over} $\mathbb F$
and have {\it diameter $d$}.
\end{definition}

\begin{note}\rm According to a common notational convention,
$A^*$ denotes the conjugate-transpose of $A$. We are not
using this convention. In a Leonard pair $A$, $A^*$ the
linear transformations $A$ and $A^*$ are arbitrary subject
to (i), (ii) above.
\end{note}

\noindent When working with a Leonard pair, it is convenient
to consider a closely related object called a Leonard system.
Before defining a Leonard system, we recall a few concepts 
from linear algebra.
An element
$A \in {\rm End}(V)$ is said to be {\it diagonalizable} whenever
$V$ is spanned by the eigenspaces of $A$. The element $A$ is called
{\it multiplicity-free} whenever $A$ is diagonalizable and
each eigenspace of $A$ has dimension one.
Note that $A$ is multiplicity-free if and only if $A$
has $d+1$ mutually distinct eigenvalues in $\mathbb F$.
Assume that $A$ is multiplicity-free,
and let $\lbrace V_i\rbrace_{i=0}^d$ denote an ordering
of the eigenspaces of $A$. For $0 \leq i \leq d$
let $\theta_i$ denote the eigenvalue of $A$ for $V_i$.
For $0 \leq i \leq d$ define
$E_i \in {\rm End}(V)$ such that
$(E_i - I)V_i = 0 $ and
$E_i V_j=0$ if $j\not=i$ $(0 \leq j \leq d)$.
We call $E_i$ the {\it primitive idempotent} of $A$ for
$V_i$ (or 
$\theta_i$).
We have (i) $E_i E_j = \delta_{i,j} E_i$ $(0 \leq i,j\leq d)$;
(ii)  $I = \sum_{i=0}^d E_i$;
(iii)   $AE_i = \theta_i E_i = E_i A$ $(0 \leq i \leq d)$;
(iv) $A = \sum_{i=0}^d \theta_i E_i$;
(v) $V_i = E_iV$ $(0 \leq i \leq d)$;
(vi) ${\rm rank}(E_i) = 1$ $(0 \leq i \leq d)$,
(vii) ${\rm tr}(E_i) = 1$ $(0 \leq i \leq d)$,
where tr means trace.
Moreover
\begin{align}
\label{eq:ei}
  E_i=\prod_{\stackrel{0 \leq j \leq d}{j \neq i}}
            \frac{A-\theta_j I}{\theta_i-\theta_j}
	    \qquad \qquad (0 \leq i \leq d).
\end{align}
Let $\mathcal D$ denote the subalgebra of ${\rm End}(V)$ generated by
$A$. The elements 
$\lbrace A^i\rbrace_{i=0}^d$ form a basis for $\mathcal D$, and
$\prod_{i=0}^d (A-\theta_i I) = 0$. Moreover
$\lbrace E_i\rbrace_{i=0}^d$ 
form a basis for $\mathcal D$.
\medskip

\begin{definition}
\label{def:deflstalkS99}
\rm (See \cite[Definition~1.4]{2lintrans}).
By a {\it Leonard system} on $V$, we mean a 
sequence 
\begin{equation}
\Phi=(A; \lbrace E_i\rbrace_{i=0}^d; A^*; \lbrace E^*_i\rbrace_{i=0}^d)
\label{eq:ourstartingpt}
\end{equation}
of elements in 
${\rm End}(V)$  that satisfy {\rm (i)--(v)} below:
\begin{enumerate}
\item[\rm (i)] each of $A$,  $A^*$ is multiplicity-free;
\item[\rm (ii)]  $\lbrace E_i\rbrace_{i=0}^d$
is an ordering of the primitive 
idempotents of $A$;
\item[\rm (iii)]  $\lbrace E^*_i\rbrace_{i=0}^d$
is an ordering of the primitive 
idempotents of $A^*$;
\item[\rm (iv)] ${\displaystyle{
E^*_iAE^*_j = \begin{cases}
 0, & {\mbox{\rm if $\vert i-j\vert > 1$}}; \\ 
\not=0, &{\mbox{\rm if $\vert i-j \vert = 1$}}
\end{cases}
\qquad (0 \leq i,j\leq d)}}$;
\item[\rm (v)] ${\displaystyle{
E_iA^*E_j = \begin{cases}
 0, & {\mbox{\rm if $\vert i-j\vert > 1$}}; \\ 
\not=0, &{\mbox{\rm if $\vert i-j \vert = 1$}}
\end{cases}
\qquad (0 \leq i,j\leq d)}}$.
\end{enumerate}
The Leonard system $\Phi$ is said to be {\it over} $\mathbb F$
and have {\it diameter} $d$.
\end{definition}

\noindent 
Leonard pairs and Leonard systems are related as follows.
Let $(A; \lbrace E_i\rbrace_{i=0}^d; A^*; \lbrace E^*_i\rbrace_{i=0}^d)$
denote a Leonard system on $V$.
 Then  $A, A^*$ is a Leonard pair on $V$.
Conversely, let $A, A^*$ denote a Leonard pair on $V$. 
Then each of $A$, $A^*$ is multiplicity-free
\cite[Lemma~1.3]{2lintrans}.
Moreover there exists an ordering $\lbrace E_i\rbrace_{i=0}^d$
of the primitive idempotents of
$A$,
and there exists an ordering $\lbrace E^*_i\rbrace_{i=0}^d$
of the primitive idempotents of
$A^*$,
such that 
$(A; \lbrace E_i\rbrace_{i=0}^d; A^*; \lbrace E^*_i\rbrace_{i=0}^d)$
is a Leonard system on $V$.
\medskip

\noindent Next we recall the notion of isomorphism for Leonard pairs and
Leonard systems.
\begin{definition}
\rm Let $A, A^*$ denote a Leonard pair on $V$, and let
$B,B^*$ denote a Leonard pair on a vector space $V'$.
By an {\it isomorphism of Leonard pairs} from $A, A^*$ to
$B, B^*$ we mean a vector space isomorphism
$\sigma: V \to V'$ such that
$B = \sigma A \sigma^{-1}$ and
$B^* = \sigma A^* \sigma^{-1}$.
The Leonard pairs
$A, A^*$ and 
$B,B^*$ are {\it isomorphic} whenever
there exists an isomorphism of Leonard pairs from
$A, A^*$ to
$B, B^*$.
\end{definition}

\noindent Let $\sigma: V \to V'$ denote an isomorphism of
vector spaces. For
  $X \in {\rm End}(V)$ 
abbreviate
  $X^\sigma = \sigma X \sigma^{-1}$ 
and note that $X^\sigma \in {\rm End}(V')$.
The map ${\rm End}(V) \to {\rm End}(V')$, $X \mapsto X^\sigma$
is an isomorphism of algebras. For a Leonard system
$\Phi= (A; \lbrace E_i\rbrace_{i=0}^d; A^*; \lbrace E^*_i\rbrace_{i=0}^d)$
on $V$ the sequence
\begin{align*}
\Phi^\sigma:= (A^\sigma; \lbrace E^\sigma_i\rbrace_{i=0}^d; 
A^{*\sigma}; \lbrace E^{*\sigma}_i\rbrace_{i=0}^d)
\end{align*}
is a Leonard system on $V'$.
\begin{definition}\rm
Let $\Phi$ denote a Leonard system on $V$,
and let $\Phi'$ denote a Leonard system on a vector space $V'$.
By an {\it isomorphism of Leonard systems}
from $\Phi$ to $\Phi'$,
we mean an isomorphism of vector spaces $\sigma: V\to V'$
such that
$\Phi^\sigma = \Phi'$.
The Leonard systems
$\Phi$ and $\Phi'$ are  {\it isomorphic} whenever
there exists an isomorphism of Leonard systems from 
$\Phi$ to $\Phi'$.
\end{definition}

\noindent In
\cite[Theorem~1.9]{2lintrans} we classified the
Leonard systems up to isomorphism.
Our first main goal in the present paper 
is to give an improved proof of this classifiction.
This goal will be accomplished in
Sections 3--17.
The statement of the classification is given in
Theorem
\ref{thm:classification}. The proof of 
Theorem
\ref{thm:classification}
will be completed in Section 17.
\medskip

\noindent Recall the commutator notation
$\lbrack r,s\rbrack = rs - sr$.

\section{Pre Leonard systems}

\noindent As we start our investigation of Leonard systems,
it is helpful to consider a more general object called
a pre Leonard system. This object is defined as follows. 

\begin{definition}
\label{def:prels}
\rm By a {\it pre Leonard system} on $V$, we mean a sequence
\begin{align}
\Phi=(A; \lbrace E_i\rbrace_{i=0}^d; A^*; \lbrace E^*_i\rbrace_{i=0}^d)
\label{eq:sequence}
\end{align}
of elements in ${\rm End}(V)$ that satisfy conditions
(i)--(iii) in Definition
\ref{def:deflstalkS99}.
\end{definition}

\noindent The results in this section
refer to the pre Leonard system 
$\Phi$ from (\ref{eq:sequence}).

\begin{definition}
\label{def:D}
\rm 
For $0 \leq i \leq d$ let $\theta_i$ (resp. $\theta^*_i$) denote
the eigenvalue of $A$ (resp. $A^*$) for $E_i$ (resp. $E^*_i$).
We call $\lbrace \theta_i \rbrace_{i=0}^d$
(resp. $\lbrace \theta^*_i \rbrace_{i=0}^d$) the
{\it eigenvalue sequence}
(resp. {\it dual eigenvalue sequence}) of $\Phi$.
Let $\mathcal D$ (resp. $\mathcal D^*$) denote the subalgebra of
${\rm End}(V)$ generated by $A$ (resp. $A^*$).
\end{definition}

\begin{definition}
\label{def:ai}
\rm
Define
\begin{align*}
a_i = {\rm tr} (A E^*_i), \qquad \qquad
a^*_i = {\rm tr} (A^* E_i) \qquad \qquad (0 \leq i \leq d).
\end{align*}
We call
$\lbrace a_i \rbrace_{i=0}^d$
(resp. $\lbrace a^*_i \rbrace_{i=0}^d$) 
the {\it diagonal sequence}
(resp.  {\it dual diagonal sequence}) of $\Phi$.
\end{definition}

\begin{lemma}
\label{lem:aithi}
 We have
\begin{align}
&\theta_0 + \theta_1 + \cdots + \theta_d = 
a_0 + a_1 + \cdots + a_d, 
\label{eq:aithi}
\\
&\theta^*_0 + \theta^*_1 + \cdots + \theta^*_d = 
a^*_0 + a^*_1 + \cdots + a^*_d.
\label{eq:aithis}
\end{align}
\end{lemma}
\begin{proof} To obtain
(\ref{eq:aithi}), 
observe that
\begin{align*}
\sum_{i=0}^d \theta_i = {\rm tr}(A) = {\rm tr}
\Bigl(A \sum_{i=0}^d E^*_i\Bigr) = 
\sum_{i=0}^d a_i.
\end{align*}
The proof of 
(\ref{eq:aithis}) is similar.
\end{proof}

\begin{lemma}
\label{lem:EAE}
For $0 \leq i \leq d$,
\begin{enumerate}
\item[\rm (i)] $E^*_i A E^*_i = a_i E^*_i$;
\item[\rm (ii)]
$E_i A^* E_i = a^*_i E_i$.
\end{enumerate}
\end{lemma}
\begin{proof} (i) Abbreviate 
${\mathcal A}={\rm End}(V)$.
Since $E^*_i$ has rank 1, 
the vector space $E^*_i\mathcal A E^*_i$ is spanned by $E^*_i$.
Therefore there exists
$\alpha_i \in \mathbb F$ such that
$E^*_i A E^*_i=\alpha_i E^*_i$.
In this equation take the trace of each side and use
Definition \ref{def:ai} along with ${\rm tr}(XY) = {\rm tr}(YX)$
to obtain
 $a_i=\alpha_i$.
\\
\noindent (ii) Similar to the proof of (i) above.
\end{proof}

\noindent We have been discussing the pre Leonard system
\begin{align*}
\Phi= (A; \lbrace E_i\rbrace_{i=0}^d; A^*; \lbrace E^*_i\rbrace_{i=0}^d)
\end{align*}
on $V$.  Each of the following is a pre Leonard system on $V$:
\begin{align*}
&\Phi^\Downarrow:=
(A; \lbrace E_{d-i}\rbrace_{i=0}^d; A^*; \lbrace E^*_i\rbrace_{i=0}^d),
\\
&\Phi^\downarrow:=
(A; \lbrace E_i\rbrace_{i=0}^d; A^*; \lbrace E^*_{d-i}\rbrace_{i=0}^d),
\\
&\Phi^*:=
(A^*; \lbrace E^*_i\rbrace_{i=0}^d; A; \lbrace E_{i}\rbrace_{i=0}^d).
\end{align*}

\begin{proposition}
\label{prop:ttaa}
With the above notation,
\bigskip

\centerline{
\begin{tabular}[t]{c|cccc}
   {\rm pre LS} & {\rm eigenvalue seq.} & {\rm dual eigenvalue seq.} &
   {\rm diagonal seq.} & {\rm dual diagonal seq.}
\\
\hline
$\Phi$ & 
$\lbrace \theta_i \rbrace_{i=0}^d$ &
$\lbrace \theta^*_i \rbrace_{i=0}^d$ &
$\lbrace a_i \rbrace_{i=0}^d$ &
$\lbrace a^*_i \rbrace_{i=0}^d$
\\
$\Phi^\Downarrow $&   
$\lbrace \theta_{d-i} \rbrace_{i=0}^d$ &
$\lbrace \theta^*_i \rbrace_{i=0}^d$ &
$\lbrace a_i \rbrace_{i=0}^d$ &
$\lbrace a^*_{d-i} \rbrace_{i=0}^d$
\\
$\Phi^\downarrow $ &
$\lbrace \theta_i \rbrace_{i=0}^d$ &
$\lbrace \theta^*_{d-i} \rbrace_{i=0}^d$ &
$\lbrace a_{d-i} \rbrace_{i=0}^d$ &
$\lbrace a^*_i \rbrace_{i=0}^d$
\\
$\Phi^* $ & 
$\lbrace \theta^*_i \rbrace_{i=0}^d$ &
$\lbrace \theta_i \rbrace_{i=0}^d$ &
$\lbrace a^*_i \rbrace_{i=0}^d$ &
$\lbrace a_i \rbrace_{i=0}^d$
\\
\end{tabular}}
\bigskip
\end{proposition}
\begin{proof} Use Definitions
\ref{def:D},
\ref{def:ai}.
\end{proof}

\section{The antiautomorphism $\dagger$}

\noindent We continue to discuss the  pre Leonard system
$\Phi=(A; \lbrace E_i\rbrace_{i=0}^d; A^*; \lbrace E^*_i\rbrace_{i=0}^d)$
from 
Definition \ref{def:prels}.

\begin{lemma}
\label{lem:EAEassume} 
Assume that
\begin{align*}
E^*_iAE^*_j = \begin{cases}
 0, & {\mbox{ if $\vert i-j\vert > 1$}}; \\ 
\not=0, &{\mbox{ if $\vert i-j \vert = 1$}}
\end{cases}
\qquad (0 \leq i,j\leq d).
\end{align*}
Then the elements
\begin{align}
A^i E^*_0 A^j \qquad \qquad 0 \leq i,j\leq d
\label{eq:AiEAj}
\end{align}
form a basis for the vector space ${\rm End}(V)$.
\end{lemma}
\begin{proof} For $0 \leq i \leq d$ pick $0 \not=v_i \in E^*_iV$.
So $\lbrace v_i \rbrace_{i=0}^d $ is a basis for $V$.
Without loss of generality, we may identify each  $X\in {\rm End}(V)$
with the matrix in ${\rm Mat}_{d+1}(\mathbb F)$
that represents $X$ with respect to
$\lbrace v_i \rbrace_{i=0}^d$.
From this point of view $A$ is irreducible tridiagonal and
$E^*_0 = {\rm diag}(1,0,\ldots,0)$.
Using these matrices one routinely checks that 
the elements (\ref{eq:AiEAj}) are linearly independent.
There are $(d+1)^2$ elements listed in 
(\ref{eq:AiEAj}), and this is the dimension of
${\rm End}(V)$. Therefore
the elements (\ref{eq:AiEAj}) form a basis for
${\rm End}(V)$.
\end{proof}
\begin{lemma}
\label{lem:genset}
Under the assumption in Lemma
\ref{lem:EAEassume},
each of the following is a generating set for
the algebra ${\rm End}(V)$:
{\rm (i)} $A, E^*_0$; 
{\rm (ii)} $A, A^*$.
\end{lemma}
\begin{proof} (i) By Lemma
\ref{lem:EAEassume}.
\\
\noindent (ii) By (i) above and since $E^*_0$ is a polynomial in
$A^*$.
\end{proof}
\noindent By an
{\it automorphism} of ${\rm End}(V)$ we mean
an algebra isomorphism 
 ${\rm End}(V)\to
 {\rm End}(V)$.
By an {\it antiautomorphism} of 
${\rm End}(V)$ 
we mean a vector space isomorphism $\zeta :
{\rm End}(V) \to 
{\rm End}(V) $
such that $(XY)^\zeta = Y^\zeta X^\zeta$ for
all $X, Y \in
{\rm End}(V)$.
\begin{lemma}
\label{lem:antiaut}
Under the assumption in Lemma
\ref{lem:EAEassume},
\begin{enumerate}
\item[\rm (i)] 
there exists a unique
antiautomorphism $\dagger$ of ${\rm End}(V)$ that 
fixes each of $A, A^*$;
\item[\rm (ii)]
$\dagger$ fixes everything in $\mathcal D$ and everything in
$\mathcal D^*$;
\item[\rm (iii)]
$\dagger$ fixes each of $E_i, E^*_i$ for $0 \leq i \leq d$;
\item[\rm (iv)]
$(X^\dagger)^\dagger = X$ for all $X \in {\rm End}(V)$.
\end{enumerate}
\end{lemma}
\begin{proof} (i) First we show that $\dagger$ exists.
For $0 \leq i \leq d$ pick $0 \not=v_i \in E^*_iV$.
So $\lbrace v_i \rbrace_{i=0}^d$ is a  basis for $V$.
For $X \in {\rm End}(V)$ let $X^\sharp 
\in {\rm Mat}_{d+1}(\mathbb F)$ represent $X$ with respect to
$\lbrace v_i \rbrace_{i=0}^d$. The map
$\sharp: 
{\rm End}(V) \to
{\rm Mat}_{d+1}(\mathbb F)$, $X \mapsto X^\sharp$
is an algebra isomorphism.
Write $B=A^\sharp$ and $B^*=A^{*\sharp}$.
The matrix $B$ is irreducible tridiagonal and 
$B^*={\rm diag}(\theta^*_0, \theta^*_1,\ldots, \theta^*_d)$.
Define a diagonal matrix $K \in {\rm Mat}_{d+1}(\mathbb F)$
with diagonal entries
\begin{align*}
K_{ii} = \frac{B_{01} B_{12} \cdots B_{i-1,i}}{B_{10}B_{21} \cdots B_{i,i-1}}
\qquad \qquad (0 \leq i \leq d).
\end{align*}
The matrix $K$ is invertible and $K^{-1} B^t K = B$.
Define a map 
$\flat:  {\rm Mat}_{d+1}(\mathbb F) \to
{\rm Mat}_{d+1}(\mathbb F),
X \mapsto K^{-1} X^t K$. The map $\flat $
is an antiautomorphism
of ${\rm Mat}_{d+1}(\mathbb F)$ 
that fixes each of $B,B^*$. The composition
\[
\dagger: \quad 
\begin{CD}
 {\rm End}(V) @>>\sharp >  
{\rm Mat}_{d+1}(\mathbb F)
 @>>\flat> 
{\rm Mat}_{d+1}(\mathbb F)
                 @>>\sharp^{-1} > {\rm End}(V)
                  \end{CD} 
              \]
is an antiautomorphism of ${\rm End}(V)$ that
fixes each of $A,A^*$.
We have shown that $\dagger$ exists. Next we show that
$\dagger$ is unique. Let $\zeta$ denote an antiautomorphism
of ${\rm End}(V)$ that fixes each of $A,A^*$. Then the
composition $\dagger \circ \zeta^{-1}$ is an automorphism
of ${\rm End}(V)$ that fixes each of $A, A^*$. Now
$\dagger \circ \zeta^{-1} = 1$ in view of
Lemma \ref{lem:genset}(ii).
 Therefore
 $\dagger = \zeta$. 
We have shown that $\dagger$ is unique.
\\
\noindent (ii)  Since $A$ (resp. $A^*$) generates
$\mathcal D$ (resp. $\mathcal D^*$).
\\
\noindent (iii) Since $E_i \in \mathcal D$ and 
$E^*_i \in \mathcal D^*$  for $0 \leq i \leq d$.
\\
\noindent (iv) 
The composition
$\dagger \circ \dagger$ is an automorphism of ${\rm End}(V)$
that fixes each of $A, A^*$. Now
$\dagger \circ \dagger = 1$ in view of Lemma
\ref{lem:genset}(ii).
\end{proof}

\begin{proposition}
\label{prop:threefourths}
Consider the following four conditions:
\begin{enumerate}
\item[\rm (i)]
$\displaystyle{
E^*_iAE^*_j = \begin{cases}
 0, & {\mbox{ if $ i-j > 1$}}; \\ 
\not=0, &{\mbox{ if $ i-j = 1$}}
\end{cases}
\qquad (0 \leq i,j\leq d);
}$
\item[\rm (ii)]
$\displaystyle{
E^*_iAE^*_j = \begin{cases}
 0, & {\mbox{ if $ j-i > 1$}}; \\ 
\not=0, &{\mbox{ if $ j-i = 1$}}
\end{cases}
\qquad (0 \leq i,j\leq d);
}$
\item[\rm (iii)]
$\displaystyle{
E_iA^*E_j = \begin{cases}
 0, & {\mbox{ if $ i-j > 1$}}; \\ 
\not=0, &{\mbox{ if $ i-j = 1$}}
\end{cases}
\qquad (0 \leq i,j\leq d);
}$
\item[\rm (iv)]
$\displaystyle{
E_iA^*E_j = \begin{cases}
 0, & {\mbox{ if $ j-i > 1$}}; \\ 
\not=0, &{\mbox{ if $ j-i = 1$}}
\end{cases}
\qquad (0 \leq i,j\leq d).
}$
\end{enumerate}
Assume at least three of {\rm (i)--(iv)} hold. Then each of
{\rm (i)--(iv)} holds;
in other words 
the pre Leonard system $\Phi$ is a Leonard system.
\end{proposition}
\begin{proof} Interchanging $A, A^*$ if necessary, we may assume
without loss
of generality that (i), (ii) hold.
Now the assumption of Lemma
\ref{lem:EAEassume} 
holds, so   
Lemma
\ref{lem:antiaut} applies. 
Consider the map $\dagger$ from Lemma
\ref{lem:antiaut}.
For $0 \leq i,j\leq d$ we have
\begin{align*}
(E_i A^*E_j)^\dagger = E_jA^*E_i.
\end{align*}
Therefore
$E_i A^*E_j= 0$ if and only if $E_jA^*E_i=0$.
Consequently (iii) holds if and only if (iv) holds.
The result follows.
\end{proof}

\section{The polynomials 
 $\tau_i, \eta_i, \tau^*_i, \eta^*_i$} 

\noindent We continue to discuss the  pre Leonard system
$\Phi=(A; \lbrace E_i\rbrace_{i=0}^d; A^*; \lbrace E^*_i\rbrace_{i=0}^d)$
from 
Definition \ref{def:prels}.
\medskip

\noindent 
Let $\lambda $ denote an indeterminate. Let $\mathbb F\lbrack \lambda
\rbrack$ denote the algebra consisting of the polynomials in 
$\lambda$ that have all coefficients in $\mathbb F$.

\begin{definition}
\label{def:taui}
\rm
For $0 \leq i \leq d$ define $\tau_i, \eta_i, \tau^*_i, \eta^*_i 
\in \mathbb F\lbrack \lambda \rbrack$ by
\begin{align*}
&\tau_i = (\lambda-\theta_0) (\lambda-\theta_1) \cdots 
(\lambda-\theta_{i-1}),
\qquad \qquad
\eta_i = (\lambda-\theta_d) (\lambda-\theta_{d-1}) \cdots 
(\lambda-\theta_{d-i+1}),
\\
&\tau^*_i = (\lambda-\theta^*_0) (\lambda-\theta^*_1) \cdots 
(\lambda-\theta^*_{i-1}),
\qquad \qquad
\eta^*_i = (\lambda-\theta^*_d) (\lambda-\theta^*_{d-1}) \cdots 
(\lambda-\theta^*_{d-i+1}).
\end{align*}
Each of 
$\tau_i, \eta_i, \tau^*_i, \eta^*_i $ is monic with degree $i$.
\end{definition}

\noindent We mention some results about 
 $\lbrace \tau_i\rbrace_{i=0}^d$ and 
 $\lbrace \eta_i\rbrace_{i=0}^d$;
similar results apply to
 $\lbrace \tau^*_i\rbrace_{i=0}^d$ and 
 $\lbrace \eta^*_i\rbrace_{i=0}^d$.

\begin{lemma} 
\label{lem:tau1}
 The vectors $\lbrace \tau_i(A)\rbrace_{i=0}^d$ form a basis
 for
  $\mathcal D$. Moreover the vectors
 $\lbrace \eta_i(A)\rbrace_{i=0}^d$ form a basis for
 $\mathcal D$.
 \end{lemma}
\begin{proof} Since $\lbrace A^i\rbrace_{i=0}^d$ is a basis for
 $\mathcal D$, and $\tau_i$, $\eta_i$ have degree $i$ for
 $0 \leq i \leq d$.
\end{proof}

\begin{lemma}
\label{lem:split2}
For $0 \leq i \leq d$,
\begin{enumerate}
\item[\rm (i)]
$\tau_i(A) = \sum_{h=i}^d \tau_i(\theta_h) E_h$;
\item[\rm (ii)]
$\eta_i(A) = \sum_{h=0}^{d-i} \eta_i(\theta_h) E_h$.
\end{enumerate}
\end{lemma}
\begin{proof} (i)
We have $A=\sum_{h=0}^d \theta_h E_h$, so
$\tau_i(A)=\sum_{h=0}^d \tau_i(\theta_h) E_h$. 
However $\tau_i(\theta_h)=0$ for $0 \leq h \leq i-1$,
so 
$\tau_i(A)=\sum_{h=i}^d \tau_i(\theta_h) E_h$. 
\\
\noindent (ii) Apply (i) above to $\Phi^\Downarrow$.
\end{proof}

\begin{lemma}
\label{lem:split3} For
$0 \leq i \leq d$,
\begin{enumerate}
\item[\rm (i)]  the elements 
 $\lbrace \tau_j(A)\rbrace_{j=0}^i$ and
 $\lbrace A^j\rbrace_{j=0}^i$ have the same span;
\item[\rm (ii)] 
the elements 
 $\lbrace \tau_j(A)\rbrace_{j=i}^d$ and
 $\lbrace E_j\rbrace_{j=i}^d$ have the same span.
\end{enumerate}
\end{lemma}
\begin{proof}(i) The polynomial $\tau_j$ has degree $j$ for $0 \leq j \leq d$.
\\
\noindent (ii)
By Lemma
\ref{lem:tau1} it suffices to show that 
the span of 
 $\lbrace \tau_j(A)\rbrace_{j=i}^d$ is contained in the
 span of
 $\lbrace E_j\rbrace_{j=i}^d$.
But this follows from Lemma
\ref{lem:split2}(i).
\end{proof}

\begin{lemma}
\label{lem:split4} For
$0 \leq i \leq d$,
\begin{enumerate}
\item[\rm (i)]  the elements 
 $\lbrace \eta_j(A)\rbrace_{j=0}^i$ and
 $\lbrace A^j\rbrace_{j=0}^i$ have the same span;
\item[\rm (ii)] 
the elements 
 $\lbrace \eta_j(A)\rbrace_{j=i}^d$ and
 $\lbrace E_j\rbrace_{j=0}^{d-i}$ have the same span.
\end{enumerate}
\end{lemma}
\begin{proof} 
Apply Lemma
\ref{lem:split3}
to $\Phi^\Downarrow$.
\end{proof}

\section{Normalizing idempotents}

\noindent We continue to discuss the  pre Leonard system
$\Phi=(A; \lbrace E_i\rbrace_{i=0}^d; A^*; \lbrace E^*_i\rbrace_{i=0}^d)$
from Definition
\ref{def:prels}.
\medskip

\noindent Next we explain what it means for
$E^*_0$ to be 
normalizing. This concept was introduced in 
\cite[Section~6]{hanson}, although our point of view is
different.

\begin{definition}
\label{def:normalizing}
\rm
The primitive idempotent $E^*_0$ is called {\it normalizing}
whenever 
\begin{align*}
E_i E^*_0 \not=0   \qquad \qquad (0 \leq i \leq d).
\end{align*}
\end{definition}

\noindent In the next two lemmas we give some
necessary and sufficient conditions for $E^*_0$ to
be normalizing. The proofs are routine, and omitted.

\begin{lemma}
\label{lem:normalChar}
The following {\rm (i)--(iv)} are equivalent:
\begin{enumerate}
\item[\rm (i)] $E^*_0$ is normalizing;
\item[\rm (ii)] ${\mathcal D}E^*_0$ has dimension $d+1$;
\item[\rm (iii)] the elements $\lbrace A^i E^*_0\rbrace_{i=0}^d$ are linearly
independent;
\item[\rm (iv)] for $X \in \mathcal D$,
 $X E^*_0=0$ implies $X= 0$.
\end{enumerate}
\end{lemma}

\noindent 
\begin{lemma}
\label{lem:normalChar2}
The following {\rm (i)--(iv)} are equivalent:
\begin{enumerate}
\item[\rm (i)] $E^*_0$ is normalizing;
\item[\rm (ii)] $E_iV = E_iE^*_0V$ for $0 \leq i \leq d$;
\item[\rm (iii)] ${\mathcal D}E^*_0V  = V$; 
\item[\rm (iv)] for $0 \not=\xi \in E^*_0V$ the map
       $\mathcal D \to V$,
       $X \mapsto X\xi$
       is a bijection.
\end{enumerate}
\end{lemma}

\begin{proposition}
\label{prop:normTrick}
Assume that $E^*_0$ is normalizing.
Then for $X \in {\rm End}(V)$ and $0 \leq i \leq d$
the following
are equivalent:
\begin{enumerate}
\item[\rm (i)] $XE_i=0$;
\item[\rm (ii)] $XE_iE^*_0=0$.
\end{enumerate}
\end{proposition}
\begin{proof} Using Lemma
\ref{lem:normalChar2}(ii),
\begin{align*}
XE_i=0 
\quad \Leftrightarrow \quad
XE_iV=0
\quad \Leftrightarrow \quad
XE_iE^*_0V=0
\quad \Leftrightarrow \quad
XE_iE^*_0=0.
\end{align*}
\end{proof}

\section{Normalizing idempotents and decompositions}

\noindent We continue to discuss the  pre Leonard system
$\Phi=(A; \lbrace E_i\rbrace_{i=0}^d; A^*; \lbrace E^*_i\rbrace_{i=0}^d)$
from Definition
\ref{def:prels}.

\begin{definition}
\label{def:decomp}
\rm
By a {\it decomposition of $V$} we mean a sequence
$\lbrace V_i \rbrace_{i=0}^d$
of one-dimensional subspaces of $V$ such that the sum
$V= \sum_{i=0}^d V_i$ is direct.
\end{definition}

\begin{example} 
\label{ex:decomp}
\rm Each of the sequences
\begin{align*}
\lbrace E_iV\rbrace_{i=0}^d,
\qquad \qquad
\lbrace E^*_iV\rbrace_{i=0}^d
\end{align*}
is a decomposition of $V$.
\end{example}

\begin{example}
\label{ex:decbas} \rm Let $\lbrace v_i\rbrace_{i=0}^d$
denote a basis for $V$. For $0 \leq i \leq d$ let $V_i$
denote the span of $v_i$. Then $\lbrace V_i\rbrace_{i=0}^d$
is a decomposition of $V$, said to be {\it induced} by
$\lbrace v_i\rbrace_{i=0}^d$.
\end{example}

\begin{lemma}
\label{lem:AssumeNorm}
Assume that $E^*_0$ is normalizing, and define
\begin{align*} 
U_i = \tau_i(A) E^*_0V \qquad \qquad (0 \leq i \leq d).
\end{align*}
Then the following {\rm (i)--(v)} hold:
\begin{enumerate}
\item[\rm (i)] $\lbrace U_i\rbrace_{i=0}^d $ is a decomposition of $V$;
\item[\rm (ii)] $(A-\theta_i I)U_i = U_{i+1}$ $\qquad$ $(0 \leq i \leq d-1)$;
\item[\rm (iii)] $(A-\theta_d I)U_d = 0$;
\item[\rm (iv)] $U_0 + U_1 + \cdots + U_i = E^*_0V + A E^*_0V+ \cdots
+ A^i E^*_0V$ $\qquad $ $(0 \leq i \leq d)$;
\item[\rm (v)]
$U_i + U_{i+1} + \cdots + U_d = E_iV+ E_{i+1}V+ \cdots + E_dV$
$\qquad$ $(0 \leq i\leq d)$.
\end{enumerate}
\end{lemma}
\begin{proof} (i) By Lemma
\ref{lem:normalChar2}(iv) and since
$\lbrace \tau_i(A)\rbrace_{i=0}^d$ is a basis for $\mathcal D$.
\\
\noindent (ii) By Definition \ref{def:taui}.
\\
\noindent (iii) Since
$0=\prod_{i=0}^d (A-\theta_i I)=(A-\theta_d I)\tau_d(A)$.
\\
\noindent (iv) By
Lemma
\ref{lem:split3}(i)
and
Lemma \ref{lem:normalChar2}(iv).
\\
\noindent (iv) 
By Lemma
\ref{lem:split3}(ii)
and
Lemma \ref{lem:normalChar2}(iv).
\end{proof}

\begin{lemma} 
\label{lem:normal}
The following {\rm (i)--(iii)} are equivalent:
\begin{enumerate}
\item[\rm (i)] 
$\displaystyle{
E^*_iAE^*_j = \begin{cases}
 0, & {\mbox{ if $ i-j > 1$}}; \\ 
\not=0, &{\mbox{ if $ i-j = 1$}}
\end{cases}
\qquad (0 \leq i,j\leq d);
}$
\item[\rm (ii)] for $0 \leq i \leq d$ there exists
$f_i \in {\mathbb F}\lbrack \lambda \rbrack$ such that
${\rm deg}(f_i)=i$ and 
$E^*_iV = f_i(A) E^*_0V$;
\item[\rm (iii)] for $0 \leq i \leq d$,
\begin{align}
E^*_0V + E^*_1V+ \cdots + E^*_iV = E^*_0V + A E^*_0V + \cdots + A^i E^*_0V.
\label{eq:Aisum}
\end{align}
\end{enumerate}
\noindent Assume that {\rm (i)--(iii)} hold. Then $E^*_0$ is
normalizing.
\end{lemma}
\begin{proof} ${\rm (i)} \Rightarrow {\rm (ii)}$ For $0 \leq i \leq d$
pick
$0 \not=v_i \in E^*_iV$.
So $\lbrace v_i\rbrace_{i=0}^d$ is a basis for $V$.
Let $ B \in {\rm Mat}_{d+1}(\mathbb F)$ represent $A$
with respect to $\lbrace v_i \rbrace_{i=0}^d$. The
entries of $B$ satisfy
\begin{align*}
B_{ij} = \begin{cases}
 0, & {\mbox{\rm if $ i-j > 1$}}; \\ 
\not=0, &{\mbox{\rm if $ i-j = 1$}}
\end{cases}
\qquad (0 \leq i,j\leq d).
\end{align*}
Define polynomials $\lbrace f_i \rbrace_{i=0}^d$ in
$\mathbb F\lbrack \lambda \rbrack$ by $f_0=1$ and
\begin{align*}
\lambda f_j = \sum_{i=0}^{j+1} B_{ij} f_i \qquad \qquad (0 \leq j\leq d-1).
\end{align*}
For $0 \leq i \leq d$ the polynomial
$f_i$ 
has degree $i$. Also
$v_i = f_i(A)v_0$, so
$E^*_iV = f_i(A)E^*_0V$.
\\
\noindent ${\rm (ii)} \Rightarrow {\rm (iii)}$
The polynomial $f_j$ has degree at most $i$ for
$0 \leq j\leq i$, so
\begin{align*}
E^*_0V + E^*_1V+ \cdots + E^*_iV \subseteq
E^*_0V + A E^*_0V + \cdots + A^i E^*_0V.
\end{align*}
In this inclusion, the left-hand side has dimension $i+1$
and the right-hand side has dimension at most $i+1$.
Therefore the inclusion holds with equality.
\\
\noindent ${\rm (iii)} \Rightarrow {\rm (i)}$
For $0 \leq i \leq d$ let $V_i$ denote the common value in
(\ref{eq:Aisum}). Observe that
\begin{align*}
E^*_i V_j = 
 \begin{cases}
 0, & {\mbox{\rm if $ i > j$}}; \\ 
\not=0, &{\mbox{\rm if $ i\leq j$}}
\end{cases}
\qquad (0 \leq i,j\leq d).
\end{align*}
Also observe  that
\begin{align*}
V_{j+1} = V_j + AV_j \qquad \qquad (0 \leq j \leq d-1).
\end{align*}
Now for $0 \leq i,j\leq d$ we check the conditions in (i).
First assume that $i-j>1$. Then
\begin{align*}
E^*_iAE^*_jV \subseteq E^*_iAV_j \subseteq E^*_i V_{j+1} = 0,
\end{align*}
so $E^*_iAE^*_j = 0$.
Next assume that $i-j=1$. To show 
that $E^*_iAE^*_j\not=0$,
we suppose $E^*_iAE^*_j=0$ and 
get a contradiction. For 
 $0 \leq h \leq i-1$ we have
$E^*_iA E^*_h = 0$, so
$E^*_iA E^*_hV = 0$.
Therefore $E^*_i A V_{i-1} = 0$. We also have $E^*_iV_{i-1}=0$, so
$E^*_iV_i = E^*_i(V_{i-1} + A V_{i-1}) = 0$,
for a contradiction. Therefore 
 $E^*_iAE^*_j\not=0$.
\medskip

\noindent Assume that (i)--(iii) hold. Setting $i=d$ in (iii) we obtain
$V={\mathcal D} E^*_0V$. Consequently $E^*_0$ is normalizing by Lemma
\ref{lem:normalChar2}(i),(iii).
\end{proof}

\begin{proposition} 
\label{lem:tighter}
\label{prop:splitchar} 
The following {\rm (i)--(iii)}  are equivalent:
\begin{enumerate}
\item[\rm (i)]
Both
\begin{align}
E^*_iAE^*_j &= \begin{cases}
 0, & {\mbox{ if $ i-j > 1$}}; \\ 
\not=0, &{\mbox{ if $i-j = 1$}}
\end{cases}
\qquad (0 \leq i,j\leq d),
\label{eq:part1}
\\
E_i A^*E_j &= 0 \quad \mbox{ if $j-i> 1$} \qquad (0 \leq i,j \leq d).
\label{eq:part2}
\end{align}
\item[\rm (ii)] There exists a decomposition $\lbrace U_i \rbrace_{i=0}^d$
of $V$ such that
\begin{align}
&(A-\theta_i I) U_i = U_{i+1} \qquad (0 \leq i \leq d-1),
\qquad \quad (A-\theta_d I)U_d=0,
\label{eq:Aaction}
\\
& (A^*-\theta^*_i I)U_i \subseteq U_{i-1}
\qquad  (1 \leq i \leq d),
\qquad \qquad (A^* - \theta^*_0 I) U_0=0.
\label{eq:Asaction}
\end{align}
\item[\rm (iii)] There exist scalars $\lbrace \varphi_i\rbrace_{i=1}^d$
in $\mathbb F$ and a basis for $V$ with respect to which

\begin{align}
\label{eq:split}
       A: \quad 
	 \left(
	    \begin{array}{cccccc}
	      \theta_0 &  & & & & {\bf 0}  \\
	       1 & \theta_1 & &    & &   \\
		 & 1 & \theta_2  & &  &
		  \\
		  && \cdot & \cdot &&
		    \\
		    & & &  \cdot & \cdot & \\
		     {\bf 0}  & &  & & 1 & \theta_d
		      \end{array}
		       \right),
\qquad \quad 
	A^*: \quad \left(
	    \begin{array}{cccccc}
	      \theta^*_0 & \varphi_1  & & & & {\bf 0}  \\
	        & \theta^*_1 & \varphi_2 &    & &   \\
		 &  & \theta^*_2  & \cdot &  &
		  \\
		  && & \cdot & \cdot &
		    \\
		    & & &   &  \cdot  & \varphi_d \\
		     {\bf 0}  & &  & &  & \theta^*_d
		      \end{array}
		       \right).
\end{align}
\end{enumerate}
\noindent Assume that {\rm (i)--(iii)} hold.
Then
$E^*_0$ is normalizing and $U_i= \tau_i(A)E^*_0V$ for
$0 \leq i \leq d$. 
The basis for $V$ from
{\rm (iii)}
 is $\lbrace \tau_i(A)\xi\rbrace_{i=0}^d$,
where $0 \not=\xi \in E^*_0V$. 
This basis induces the 
 decomposition
$\lbrace U_i \rbrace_{i=0}^d$.
The sequence $\lbrace \varphi_i\rbrace_{i=1}^d$ is unique.
\end{proposition}
\begin{proof}
${\rm (i)}\Rightarrow {\rm (ii)} $  
The element $E^*_0$ is normalizing by
Lemma
\ref{lem:normal}, so Lemma
\ref{lem:AssumeNorm}
applies. Consider the decomposition
$\lbrace U_i \rbrace_{i=0}^d$
of $V$ from Lemma
\ref{lem:AssumeNorm}.
This decomposition satisfies
(\ref{eq:Aaction}) 
by 
Lemma 
\ref{lem:AssumeNorm}(ii),(iii).
We now show (\ref{eq:Asaction}).
By Lemma
\ref{lem:AssumeNorm}(iv) and Lemma
\ref{lem:normal}(iii),
\begin{align*}
U_0 + \cdots + U_i = 
E^*_0V+\cdots + E^*_iV  \qquad \qquad (0 \leq i \leq d).
\end{align*}
For $0 \leq i \leq d$,
\begin{align*}
(A^*-\theta^*_i I) U_i 
&\subseteq (A^*-\theta^*_iI)(U_0 + \cdots + U_i) \\
&= (A^*-\theta^*_iI)(E^*_0V+\cdots + E^*_iV ) \\
&= E^*_0V+\cdots + E^*_{i-1}V  \\
&= U_0+\cdots + U_{i-1}.
\end{align*}
\noindent Also, using Lemma
\ref{lem:AssumeNorm}(v) and 
(\ref{eq:part2}),
\begin{align*}
(A^*-\theta^*_i I)U_i 
&\subseteq (A^*-\theta^*_iI)(U_i + \cdots + U_d) \\
&= (A^*-\theta^*_iI)(E_iV + \cdots + E_dV) \\
&\subseteq  E_{i-1}V + \cdots + E_dV \\
&=  U_{i-1} + \cdots + U_d.
\end{align*}
The above comments imply
(\ref{eq:Asaction}).
\\
${\rm (ii)}\Rightarrow {\rm (i)} $ 
From 
(\ref{eq:Asaction}) we obtain
\begin{align}
U_0 + \cdots + U_i = 
E^*_0V+\cdots + E^*_iV  \qquad \qquad (0 \leq i \leq d).
\label{eq:needthis}
\end{align}
In particular $U_0=E^*_0V$. By
this and (\ref{eq:Aaction}) we obtain
$U_i=\tau_i(A)E^*_0V$ for $0 \leq i \leq d$.
Consequently $V={\mathcal D}E^*_0V$, so $E^*_0$
is normalizing by Lemma
\ref{lem:normalChar2}(i),(iii).
We show
(\ref{eq:part1}).
Combining 
Lemma 
\ref{lem:AssumeNorm}(iv) and
(\ref{eq:needthis}) we obtain
Lemma \ref{lem:normal}(iii). This gives
Lemma \ref{lem:normal}(i), which is 
(\ref{eq:part1}).
Next we show 
(\ref{eq:part2}).
Let $i,j$ be given with
$j-i>1$.  Using Lemma
\ref{lem:AssumeNorm}(v) and
(\ref{eq:Asaction}),
\begin{align*}
E_iA^*E_jV
&\subseteq  E_iA^*(E_jV+ \cdots + E_dV)
\\
&= E_iA^*(U_j+ \cdots + U_d)
\\
&\subseteq E_i(U_{j-1}+ \cdots + U_d)
\\
&= E_i(E_{j-1}V+ \cdots + E_dV)
\\
&=0.
\end{align*}
\noindent Therefore $E_i A^*E_j=0$. We have shown 
(\ref{eq:part2}).
\\
\noindent 
${\rm (ii)}\Leftrightarrow {\rm (iii)} $ Assertion (iii) is a reformulation
of (ii) in terms of matrices.

Now assume that (i)--(iii) hold.
We mentioned in the proof of 
${\rm (ii)}\Rightarrow {\rm (i)} $ 
that $E^*_0$ is normalizing and
$U_i = \tau_i(A)E^*_0V$ for $0 \leq i \leq d$.
Let $\lbrace u_i\rbrace_{i=0}^d$ denote the basis for $V$ 
from (iii), and define $\xi = u_0$.
From the matrix representing $A^*$ in
(\ref{eq:split}),
we see that $\xi$ is an eigenvector for $A^*$ with eigenvalue
$\theta^*_0$. So $\xi \in E^*_0V$.
From the matrix representing $A$ in
(\ref{eq:split}),
we obtain 
$(A-\theta_i I)u_i =u_{i+1}$  for $0 \leq i \leq d-1$.
Consequently $u_i = \tau_i(A)\xi$ for $0 \leq i \leq d$.
For $0 \leq i \leq d$ we have 
$u_i = \tau_i(A)\xi \in \tau_i(A)E^*_0V=U_i$.
So the basis $\lbrace u_i \rbrace_{i=0}^d$ induces  
the decomposition $\lbrace U_i \rbrace_{i=0}^d$.
The sequence $\lbrace \varphi_i\rbrace_{i=1}^d$ is unique
since
the vector $\xi$ is unique up to multiplication by a nonzero scalar.
\end{proof}

\begin{lemma} 
\label{lem:varphinz}
Assume that the equivalent conditions {\rm (i)--(iii)} hold
in Proposition 
\ref{prop:splitchar}.
Then  for $1 \leq i \leq d$ the following {\rm (i)--(iii)}
are equivalent:
\begin{enumerate}
\item[\rm (i)] $E_{i-1}A^*E_i \not=0$;
\item[\rm (ii)] $(A^*-\theta^*_i I) U_i =U_{i-1}$;
\item[\rm (iii)] $\varphi_i \not=0$.
\end{enumerate}
\end{lemma}
\begin{proof} 
${\rm (i)} \Rightarrow {\rm (ii)}$ 
We assume that 
$(A^*-\theta^*_i I) U_i \not =U_{i-1}$ and get a contradiction.
We have
$(A^*-\theta^*_i I) U_i =0$ since 
$(A^*-\theta^*_i I) U_i \subseteq U_{i-1}$ 
and $U_{i-1}$ has dimension one.
Using Lemma
\ref{lem:AssumeNorm}(v),
\begin{align*}
E_{i-1} A^*E_i V 
&\subseteq  E_{i-1}A^*(E_iV+\cdots + E_dV)\\
&= E_{i-1}A^*(U_i+\cdots + U_d)\\
&\subseteq  E_{i-1}(U_i+\cdots + U_d)\\
& =  E_{i-1}(E_iV+\cdots + E_dV)\\
&=0.
\end{align*}
Therefore $E_{i-1}A^*E_i=0$ for a contradiction.
We have shown $(A^*-\theta^*_i I) U_i =U_{i-1}$.
\\
\noindent 
${\rm (ii)} \Rightarrow {\rm (i)}$ 
We assume that $E_{i-1} A^*E_i=0$ and get a contradiction.
Using Lemma
\ref{lem:AssumeNorm}(v),
\begin{align*}
U_{i-1}
&= (A^*-\theta^*_i I ) U_i 
\\
&\subseteq    (A^*-\theta^*_i I ) (U_i+ \cdots + U_d) \\
& =    (A^*-\theta^*_i I ) (E_iV+ \cdots + E_dV) \\
& \subseteq    E_iV+ \cdots + E_dV \\
& = U_i+ \cdots + U_d.
\end{align*}
This contradicts the fact that
$\lbrace U_i\rbrace_{i=0}^d$ is a decomposition.
We have shown $E_{i-1} A^*E_i\not=0$.
\\
\noindent
${\rm (ii)} \Leftrightarrow {\rm (iii)}$ 
Use the matrix representation of $A^*$ from
(\ref{eq:split}).
\end{proof}

\section{A result about wrap-around}

\noindent We continue to discuss the pre Leonard system
$\Phi=(A; \lbrace E_i\rbrace_{i=0}^d; A^*; \lbrace E^*_i\rbrace_{i=0}^d)$
from Definition
\ref{def:prels}.
\medskip

\noindent 
Throughout this section we assume
that $\Phi$ satisfies the equivalent conditions
(i)--(iii) in Proposition
\ref{prop:splitchar}.
We will obtain a useful result involving the
scalars $\lbrace \varphi_i \rbrace_{i=1}^d$ from
Proposition
\ref{prop:splitchar}(iii); this result is sometimes called
the wrap-around result.
\medskip

\noindent Recall the parameters
$\lbrace a_i\rbrace_{i=0}^d$, $\lbrace a^*_i\rbrace_{i=0}^d $ from Definition
\ref{def:ai}. We next compute these parameters in terms of
$\lbrace \theta_i\rbrace_{i=0}^d$,
$\lbrace \theta^*_i\rbrace_{i=0}^d$,
$\lbrace \varphi_i\rbrace_{i=1}^d$.
First assume that $d=0$. Then $A=\theta_0 I$
and $A^*=\theta^*_0 I$, so $a_0=\theta_0$ and $a^*_0 = \theta^*_0$.
\begin{lemma} 
{\rm (See \cite[Lemma~5.1]{2lintrans}).}
\label{lem:dg1}
For $d\geq 1$ we have
\begin{align*}
a_0 &= \theta_0 + \frac{\varphi_1}{\theta^*_0-\theta^*_1},
\\
a_i &= \theta_i + \frac{\varphi_i}{\theta^*_i-\theta^*_{i-1}}
+\frac{\varphi_{i+1}}{\theta^*_i-\theta^*_{i+1}}
\qquad \qquad (1 \leq i \leq d-1),
\\
a_d &= \theta_d + \frac{\varphi_d}{\theta^*_d-\theta^*_{d-1}}
\end{align*}
\noindent and
\begin{align*}
a^*_0 &= \theta^*_0 + \frac{\varphi_1}{\theta_0-\theta_1},
\\
a^*_i &= \theta^*_{i} + \frac{\varphi_{i}}{\theta_i-\theta_{i-1}}
+\frac{\varphi_{i+1}}{\theta_i-\theta_{i+1}}
\qquad \qquad (1 \leq i \leq d-1),
\\
a^*_d &= \theta^*_d + \frac{\varphi_d}{\theta_d-\theta_{d-1}}.
\end{align*}
\end{lemma}
\begin{proof} Concerning
$\lbrace a_i\rbrace_{i=0}^d$, define
\begin{align*}
x_0 &= \theta_0 + \frac{\varphi_1}{\theta^*_0-\theta^*_1},
\\
x_i &= \theta_i + \frac{\varphi_i}{\theta^*_i-\theta^*_{i-1}}
+\frac{\varphi_{i+1}}{\theta^*_i-\theta^*_{i+1}}
\qquad \qquad (1 \leq i \leq d-1),
\\
x_d &= \theta_d + \frac{\varphi_d}{\theta^*_d-\theta^*_{d-1}}.
\end{align*}
We show that $a_i=x_i$ for $0 \leq i \leq d$. Since
$\lbrace \theta^*_i\rbrace_{i=0}^d$ are mutually distinct, 
it suffices to show that
$0 = \sum_{i=0}^d (x_i-a_i)\theta^{*r}_i$
for $0 \leq r \leq d$. Let $r$ be given. We compute the trace of
$A A^{*r}$ in two ways. On one hand, 
$A^*= \sum_{i=0}^d \theta^*_i E^*_i$
so $A^{*r}= \sum_{i=0}^d \theta^{*r}_i E^*_i$. By this and
Definition \ref{def:ai},
\begin{align}
\label{eq:view1}
{\rm tr} (A A^{*r}) = \sum_{i=0}^d a_i \theta^{*r}_i.
\end{align}
On the other hand, consider the matrix representations of
$A$ and $A^*$ 
from
(\ref{eq:split}).
 Using these matrices
we compute the
trace of $A A^{*r}$ as the sum of the diagonal entries.
A brief calculation yields 
\begin{align}
\label{eq:view2}
{\rm tr} (A A^{*r}) = \sum_{i=0}^d \theta_i \theta^{*r}_i
+ \sum_{i=1}^d \varphi_i \frac{\theta^{*r}_{i-1} - \theta^{*r}_i}{
\theta^*_{i-1}-\theta^*_i}.
\end{align}
Comparing 
(\ref{eq:view1}),
(\ref{eq:view2}) we get an equation that  becomes
$0 = \sum_{i=0}^d (x_i-a_i)\theta^{*r}_i$ after rearranging the terms.
We have shown that $a_i=x_i$ for $0 \leq i \leq d$.
Our assertions concerning $\lbrace a^*_i\rbrace_{i=0}^d$ 
are similarly obtained, by computing in two ways the 
trace of $A^* A^r$ for $0 \leq r \leq d$.
\end{proof}

\begin{lemma}
{\rm (See \cite[Lemma~5.2]{2lintrans}).}
\label{lem:sumaieiequalsvarphiS99}
For $1 \leq i \leq d$ 
the scalar $\varphi_i$ is equal to each of the following four expressions:
\begin{align*}
&(\theta^*_i-\theta^*_{i-1})\sum_{h=0}^{i-1} (\theta_h-a_h),
\qquad \qquad    
(\theta^*_{i-1}-\theta^*_i)\sum_{h=i}^{d} (\theta_h-a_h),
\\
&(\theta_i-\theta_{i-1})\sum_{h=0}^{i-1} (\theta^*_h-a^*_h),
\qquad \qquad  
(\theta_{i-1}-\theta_i)\sum_{h=i}^{d} (\theta^*_h-a^*_h).
\end{align*}
\end{lemma}
\begin{proof} 
Use Lemmas
\ref{lem:aithi},
\ref{lem:dg1}.
\end{proof}

\begin{definition}\rm
\label{def:vartheta}
Define
\begin{align*}
\vartheta_i = \varphi_i - (\theta^*_i-\theta^*_0)(\theta_{i-1}-\theta_d)
\qquad \qquad (1 \leq i \leq d)
\end{align*}
and 
$\vartheta_0=0$,
$\vartheta_{d+1}=0$.
\end{definition}

\begin{proposition} {\rm (wrap-around)}
\label{prop:dg2} Assume $d\geq 2$. Then
\begin{align*}
\sum_{i=0}^{d-2} E_d A^* E_i E^*_0 (\theta_i - \theta_{d-1})
= E_d E^*_0 (\vartheta_1 -\vartheta_d).
\end{align*}
\end{proposition}
\begin{proof}
In the equation $I = \sum_{i=0}^d E^*_i$,
multiply each term on the right by $AE^*_0$. Simplify
the result using $E^*_0AE^*_0= a_0 E^*_0$ and
$E^*_i A E^*_0 = 0 $ $(2 \leq i \leq d)$ to obtain
\begin{align}
A E^*_0 = a_0 E^*_0 + E^*_1 A E^*_0.
\label{eq:step1}
\end{align}
In 
(\ref{eq:step1}), multiply each term on the left by $A^*$ and simplify
to get
\begin{align}
A^*A E^*_0 = a_0 \theta^*_0 E^*_0 + \theta^*_1 E^*_1 A E^*_0.
\label{eq:step2}
\end{align}
\noindent In the equation
$I = \sum_{i=0}^d E_i$,
multiply each term on the left by $E_d A^*$. Simplify
the result using $E_d A^* E_d = a^*_d E_d$ to obtain
\begin{align}
E_d A^* = a^*_d E_d + 
\sum_{i=0}^{d-1} E_d A^* E_i.
\label{eq:step3}
\end{align}
\noindent In 
(\ref{eq:step3}), multiply each term on the right by $A$ and simplify
to obtain
\begin{align}
E_d A^*A = a^*_d \theta_d E_d + 
\sum_{i=0}^{d-1} \theta_i E_d A^* E_i.
\label{eq:step4}
\end{align}
\noindent We now compute
$\theta^*_1 E_d$ times 
(\ref{eq:step1})
minus
$E_d$ times 
(\ref{eq:step2})
minus
(\ref{eq:step3}) times $\theta_{d-1} E^*_0$
plus
(\ref{eq:step4}) times $ E^*_0$.
The result is
\begin{align*}
&E_d E^*_0 \biggl(
(\theta^*_0 - \theta^*_1) a_0+
(\theta_{d-1} - \theta_d) a^*_d
+ \theta_d \theta^*_1 - \theta_{d-1} \theta^*_0
\biggr)
 = \sum_{i=0}^{d-2} E_d A^* E_i E^*_0 (\theta_i - \theta_{d-1}).
\end{align*}
In the above equation,
consider the coefficient of $E_d E^*_0$.
Evaluate this coefficient using
\begin{align*}
&a_0 = \theta_0 + \frac{\varphi_1}{\theta^*_0-\theta^*_1},
\qquad \qquad 
\varphi_1 = \vartheta_1 + (\theta^*_1 - \theta^*_0)(\theta_0 - \theta_d),
\\
&a^*_d = \theta^*_d + \frac{\varphi_d}{\theta_d-\theta_{d-1}},
\qquad \qquad 
\varphi_d = \vartheta_d + (\theta^*_d - \theta^*_0)(\theta_{d-1} - \theta_d)
\end{align*}
to find that this coefficient is $\vartheta_1 -\vartheta_d$. 
\end{proof}
\noindent For the sake of completeness, we mention a second version
of Proposition
\ref{prop:dg2}. We do not use this second version, so we will
not dwell on the proof.

\begin{lemma}
Assume $d\geq 2$. Then
\begin{align*}
\sum_{i=2}^{d} E^*_0 A E^*_i E_d (\theta^*_1 - \theta^*_i)
= E^*_0 E_d (\vartheta_1 -\vartheta_d).
\end{align*}
\end{lemma}
\begin{proof} Similar to the proof of
Proposition
\ref{prop:dg2}.
\end{proof}

\section{The parameter array of a Leonard system}

\noindent In this section we consider a Leonard system
$\Phi= (A; \lbrace E_i\rbrace_{i=0}^d; A^*; \lbrace E^*_i\rbrace_{i=0}^d)$
on $V$. Note that $\Phi$ satisfies the equivalent conditions
(i)--(iii) in Proposition
\ref{prop:splitchar}.

\begin{definition}\rm
\label{def:splitSeq}
By the {\it first split sequence} for $\Phi$
we mean the sequence $\lbrace \varphi_i\rbrace_{i=1}^d$ from
 Proposition
\ref{prop:splitchar}(iii).
Let $\lbrace \phi_i \rbrace_{i=1}^d$ denote the first split
sequence for $\Phi^\Downarrow$. We call
$\lbrace \phi_i \rbrace_{i=1}^d$
the {\it second split sequence} for $\Phi$.
\end{definition}

\noindent By Lemma
\ref{lem:varphinz} and the construction,
$\varphi_i$ and 
$\phi_i$ are nonzero for
 $1 \leq i \leq d$.
\medskip

\begin{lemma} 
\label{lem:2basis} There exists a basis for $V$
with respect to which
\begin{align}
       A: \quad 
	 \left(
	    \begin{array}{cccccc}
	      \theta_d &  & & & & {\bf 0}  \\
	       1 & \theta_{d-1} & &    & &   \\
		 & 1 & \theta_{d-2}  & &  &
		  \\
		  && \cdot & \cdot &&
		    \\
		    & & &  \cdot & \cdot & \\
		     {\bf 0}  & &  & & 1 & \theta_0
		      \end{array}
		       \right),
\qquad \qquad
	A^*:\quad
	  \left(
	    \begin{array}{cccccc}
	      \theta^*_0 & \phi_1  & & & & {\bf 0}  \\
	        & \theta^*_1 & \phi_2 &    & &   \\
		 &  & \theta^*_2  & \cdot &  &
		  \\
		  && & \cdot & \cdot &
		    \\
		    & & &   &  \cdot  & \phi_d \\
		     {\bf 0}  & &  & &  & \theta^*_d
		      \end{array}
		       \right).
\label{eq:phibasis}
\end{align}
\end{lemma}
\begin{proof} Apply Proposition
\ref{prop:splitchar}(iii) 
to $\Phi^\Downarrow$.
\end{proof}

\noindent

\begin{lemma} 
\label{lem:isoLS}
For a Leonard system $\Phi'$ over $\mathbb F$, the following are
equivalent:
\begin{enumerate}
\item[\rm (i)] $\Phi, \Phi'$ are isomorphic;
\item[\rm (ii)] $\Phi, \Phi'$ have the same eigenvalue sequence, dual eigenvalue
sequence, and first split sequence;
\item[\rm (iii)] $\Phi, \Phi'$ have the same eigenvalue sequence, dual eigenvalue
sequence, and second split sequence.
\end{enumerate}
\end{lemma}
\begin{proof}  
${\rm (i)} \Leftrightarrow {\rm (ii)}$ By
Proposition
\ref{prop:splitchar}(iii).
\\
${\rm (i)} \Leftrightarrow {\rm (iii)}$ By Lemma
\ref{lem:2basis}.
\end{proof}

\noindent In Lemma
\ref{lem:dg1}
we gave some formulas for 
$\lbrace a_i\rbrace_{i=0}^d$,
$\lbrace a^*_i\rbrace_{i=0}^d$ that involved
$\lbrace \varphi_i\rbrace_{i=1}^d$.
Next we give some similar formulas that involve
$\lbrace \phi_i\rbrace_{i=1}^d$.

\begin{lemma} 
\label{lem:dg1Down}
For $d\geq 1$ we have
\begin{align*}
a_0 &= \theta_d + \frac{\phi_1}{\theta^*_0-\theta^*_1},
\\
a_i &= \theta_{d-i} + \frac{\phi_i}{\theta^*_i-\theta^*_{i-1}}
+\frac{\phi_{i+1}}{\theta^*_i-\theta^*_{i+1}}
\qquad \qquad (1 \leq i \leq d-1),
\\
a_d &= \theta_0 + \frac{\phi_d}{\theta^*_d-\theta^*_{d-1}}
\end{align*}
\noindent and
\begin{align*}
a^*_0 &= \theta^*_d + \frac{\phi_d}{\theta_0-\theta_{1}},
\\
a^*_{i} &= \theta^*_{d-i} + 
\frac{\phi_{d-i+1}}{\theta_{i}-\theta_{i-1}}
+
\frac{\phi_{d-i}}{\theta_{i}-\theta_{i+1}}
\qquad \qquad (1 \leq i \leq d-1),
\\
a^*_d &= \theta^*_0 + \frac{\phi_1}{\theta_d-\theta_{d-1}},
\end{align*}
\end{lemma}
\begin{proof}
Recall that $\lbrace \phi_i \rbrace_{i=1}^d$
is the first split sequence for $\Phi^\Downarrow$.
Apply Lemma
\ref{lem:dg1}
to  $\Phi^\Downarrow$ and use the data for
  $\Phi^\Downarrow$ in Proposition
\ref{prop:ttaa}.
\end{proof}

\begin{lemma}
\label{lem:sumaieiequalsphiS99}
{\rm (See \cite[Lemma~6.4]{2lintrans}).}
For $1 \leq i \leq d$ 
the scalar $\phi_i$ is equal to each of the following four expressions:
\begin{align*}
&(\theta^*_i-\theta^*_{i-1})\sum_{h=0}^{i-1} (\theta_{d-h}-a_h),
\qquad \qquad    
(\theta^*_{i-1}-\theta^*_i)\sum_{h=i}^{d} (\theta_{d-h}-a_h),
\\
&(\theta_{d-i}-\theta_{d-i+1})\sum_{h=0}^{i-1} (\theta^*_h-a^*_{d-h}),
\qquad \qquad  
(\theta_{d-i+1}-\theta_{d-i})\sum_{h=i}^{d} (\theta^*_h-a^*_{d-h}).
\end{align*}
\end{lemma}
\begin{proof} 
Apply Lemma
\ref{lem:sumaieiequalsvarphiS99} to $\Phi^\Downarrow$
and use the data for $\Phi^\Downarrow$ in Proposition
\ref{prop:ttaa}.
\end{proof}

\begin{definition}\rm
{\rm (See \cite[Definition~10.1]{terCanForm}).}
\label{def:pa}
By the {\it parameter array} of $\Phi$ we mean the sequence
\begin{align*}
\bigl(
\lbrace \theta_i \rbrace_{i=0}^d;
\lbrace \theta^*_i \rbrace_{i=0}^d;
\lbrace \varphi_i \rbrace_{i=1}^d;
\lbrace \phi_i \rbrace_{i=1}^d
\bigr)
\end{align*}
where we recall that 
$\lbrace \theta_i \rbrace_{i=0}^d$ is the eigenvalue sequence of $\Phi$,
$\lbrace \theta^*_i \rbrace_{i=0}^d$ is the dual eigenvalue sequence of $\Phi$,
$\lbrace \varphi_i \rbrace_{i=1}^d$ is the first split sequence of $\Phi$,
and $\lbrace \phi_i \rbrace_{i=1}^d$ is the second split sequence of $\Phi$.
\end{definition}

\begin{lemma} 
\label{prop:pacomp} 
{\rm (See \cite[Theorem~1.11]{2lintrans}).}
The parameter arrays of 
\begin{align*}
\Phi, \qquad 
\Phi^\Downarrow, \qquad 
\Phi^\downarrow, \qquad 
\Phi^*
\end{align*}
are related as follows.
\bigskip

\centerline{
\begin{tabular}[t]{c|c}
   {\rm LS} & {\rm parameter array}  
\\
\hline
$\Phi$ & 
$\bigl(
\lbrace \theta_i \rbrace_{i=0}^d;
\lbrace \theta^*_i \rbrace_{i=0}^d;
\lbrace \varphi_i \rbrace_{i=1}^d;
\lbrace \phi_i \rbrace_{i=1}^d
\bigr)
$
\\
$\Phi^\Downarrow $& $ 
\bigl(
\lbrace \theta_{d-i} \rbrace_{i=0}^d;
\lbrace \theta^*_i \rbrace_{i=0}^d;
\lbrace \phi_i \rbrace_{i=1}^d;
\lbrace \varphi_i \rbrace_{i=1}^d
\bigr) $
\\
$\Phi^\downarrow $ & $ 
\bigl(
\lbrace \theta_{i} \rbrace_{i=0}^d;
\lbrace \theta^*_{d-i} \rbrace_{i=0}^d;
\lbrace \phi_{d-i+1} \rbrace_{i=1}^d;
\lbrace \varphi_{d-i+1} \rbrace_{i=1}^d
\bigr)$
\\
$\Phi^*$ & $
\bigl(
\lbrace \theta^*_{i} \rbrace_{i=0}^d;
\lbrace \theta_{i} \rbrace_{i=0}^d;
\lbrace \varphi_{i} \rbrace_{i=1}^d;
\lbrace \phi_{d-i+1} \rbrace_{i=1}^d
\bigr) $
\\
\end{tabular}}
\bigskip

\end{lemma}
\begin{proof} Use 
 Proposition
\ref{prop:ttaa} and
Lemmas
\ref{lem:sumaieiequalsvarphiS99},
\ref{lem:sumaieiequalsphiS99}.
\end{proof}

\noindent We mention a variation on Lemma
\ref{lem:isoLS}.

\begin{proposition}
\label{cor:paIso}
Two Leonard systems over $\mathbb F$ are isomorphic if and only if
they have the same parameter array.
\end{proposition}
\begin{proof} By Lemma
\ref{lem:isoLS} and Definition
\ref{def:pa}.
\end{proof}

\section{Statement of the Leonard system classification}

\noindent In the following theorem we classify up to isomorphism
the Leonard systems over $\mathbb F$.

\begin{theorem}
\label{thm:classification}
{\rm(See \cite[Theorem~1.9]{2lintrans}).}
Consider a sequence 
\begin{align}
\bigl(
\lbrace \theta_i \rbrace_{i=0}^d;
\lbrace \theta^*_i \rbrace_{i=0}^d;
\lbrace \varphi_i \rbrace_{i=1}^d;
\lbrace \phi_i \rbrace_{i=1}^d
\bigr)
\label{eq:parameterArray}
\end{align}
of scalars in $\mathbb F$.
Then there exists  a Leonard system $\Phi$ over $\mathbb F$  with
parameter array
{\rm (\ref{eq:parameterArray})} if and only if
the following conditions {\rm (PA1)--(PA5)} hold:
\begin{description}
\item[\rm (PA1)] $ \theta_i\not=\theta_j,\quad  \theta^*_i\not=\theta^*_j
\quad $
if $\;\;i\not=j,\qquad \qquad (0 \leq i,j\leq d)$;
\item[\rm (PA2)] $ \varphi_i \not=0, \quad \phi_i\not=0 \qquad\qquad 
(1 \leq i \leq d)$;
\item[\rm (PA3)] $ {\displaystyle{ \varphi_i = \phi_1 \sum_{h=0}^{i-1}
\frac{\theta_h-\theta_{d-h}}{\theta_0-\theta_d} 
+(\theta^*_i-\theta^*_0)(\theta_{i-1}-\theta_d) \qquad \;\;(1 \leq i \leq d)}}$;
\item[\rm (PA4)] $ {\displaystyle{ \phi_i = \varphi_1 \sum_{h=0}^{i-1}
\frac{\theta_h-\theta_{d-h}}{\theta_0-\theta_d} 
+(\theta^*_i-\theta^*_0)(\theta_{d-i+1}-\theta_0) \qquad (1 \leq i \leq d)}}$;
\item[\rm (PA5)] the scalars
\begin{equation}
\frac{\theta_{i-2}-\theta_{i+1}}{\theta_{i-1}-\theta_i},\qquad \qquad  
 \frac{\theta^*_{i-2}-\theta^*_{i+1}}{\theta^*_{i-1}-\theta^*_i} 
 \qquad  \qquad 
\label{eq:defbetaplusoneS99int}
\end{equation}
 are equal and independent of $i$ for $2\leq i \leq d-1$.
\end{description}
Moreover, if $\Phi$ exists then
 $\Phi$ is unique up to isomorphism of Leonard systems.
\end{theorem}

\noindent The proof of
Theorem
\ref{thm:classification}
will be completed in Section
17.
\medskip

\begin{definition}\rm 
By a {\it parameter array of diameter $d$} over $\mathbb F$, we mean
a sequence 
(\ref{eq:parameterArray}) of scalars in $\mathbb F$
that satisfy
{\rm (PA1)--(PA5)}. 
\end{definition}

\noindent Theorem
\ref{thm:classification} gives a bijection between
the following two sets:
\begin{enumerate}
\item[\rm (i)] the parameter arrays over $\mathbb F$ that have diameter $d$;
\item[\rm (ii)] the isomorphism classes of Leonard systems over
$\mathbb F$ that have diameter $d$.
\end{enumerate}

\noindent We have a comment.
\begin{lemma}
\label{cor:iso1}
For $d\geq 1$, a parameter array 
$\bigl(
\lbrace \theta_i \rbrace_{i=0}^d;
\lbrace \theta^*_i \rbrace_{i=0}^d;
\lbrace \varphi_i \rbrace_{i=1}^d;
\lbrace \phi_i \rbrace_{i=1}^d
\bigr)$
is uniquely determined by
$\varphi_1$, $\lbrace \theta_i \rbrace_{i=0}^d$,
$\lbrace \theta^*_i \rbrace_{i=0}^d$.
\end{lemma}
\begin{proof} By the nature of the equations
(PA3), (PA4).
\end{proof}

\section{Recurrent sequences}

 \noindent  Throughout this section
let $\lbrace \theta_i\rbrace_{i=0}^d$ denote
 scalars in $\mathbb F$.

\begin{definition}
\label{lem:beginthreetermS99}
{\rm (See \cite[Definition~8.2]{2lintrans}).}
\rm
Let $\beta, \gamma, \varrho$ denote scalars
in $\mathbb F$.
\begin{enumerate} 
\item[\rm (i)] The sequence 
  $\lbrace \theta_i\rbrace_{i=0}^d$ 
is said to be {\it recurrent} whenever $\theta_{i-1}\not=\theta_i$ for
$2 \leq i \leq d-1$, and 
\begin{equation}
\frac{\theta_{i-2}-\theta_{i+1}}{\theta_{i-1}-\theta_i} 
\label{eq:thethingwhichisbetaS99}
\end{equation}
is independent of
$i$ for $2 \leq i \leq  d-1$.
\item[\rm (ii)] The sequence 
  $\lbrace \theta_i\rbrace_{i=0}^d$ 
is said to be {\it $\beta$-recurrent} whenever 
\begin{equation}
\theta_{i-2}\,-\,(\beta+1)\theta_{i-1}\,+\,(\beta +1)\theta_i \,-\,\theta_{i+1}
\label{eq:betarecS99}
\end{equation}
is zero for 
$2 \leq i \leq d-1$.
\item[\rm (iii)] The sequence 
  $\lbrace \theta_i\rbrace_{i=0}^d$ 
is said to be {\it $(\beta,\gamma)$-recurrent} whenever 
\begin{equation}
\theta_{i-1}\,-\,\beta \theta_i\,+\,\theta_{i+1}=\gamma 
\label{eq:gammathreetermS99}
\end{equation}
 for 
$1 \leq i \leq d-1$.
\item[\rm (iv)] The sequence 
  $\lbrace \theta_i\rbrace_{i=0}^d$ 
is said to be {\it $(\beta,\gamma,\varrho)$-recurrent} whenever 
\begin{equation}
\theta^2_{i-1}-\beta \theta_{i-1}\theta_i+\theta^2_i 
-\gamma (\theta_{i-1} +\theta_i)=\varrho
\label{eq:varrhothreetermS99}
\end{equation}
 for 
$1 \leq i \leq d$.
\end{enumerate}
\end{definition}

\begin{lemma} 
\label{lem:recvsbrecS99}
The following are equivalent:
\begin{enumerate}
\item[\rm (i)] the sequence 
  $\lbrace \theta_i\rbrace_{i=0}^d$
is  recurrent;
\item[\rm (ii)]
the scalars $\theta_{i-1}\not=\theta_i$ for
$2 \leq i \leq d-1$, and 
there exists $\beta \in \mathbb F$ such that
  $\lbrace \theta_i\rbrace_{i=0}^d$ 
  is $\beta$-recurrent.
\end{enumerate}
\noindent Suppose {\rm (i), (ii)} hold, and that $d\geq 3$. Then
the common value of
{\rm (\ref{eq:thethingwhichisbetaS99})}
is equal to $\beta +1$.

\end{lemma}
\begin{proof} Routine.

\end{proof}

\begin{lemma}
\label{lem:brecvsbgrecS99}
For $\beta \in \mathbb F$ the following are equivalent:
\begin{enumerate}
\item[\rm (i)] the sequence 
  $\lbrace \theta_i\rbrace_{i=0}^d$
is  $\beta$-recurrent;
\item[\rm (ii)] there exists $\gamma \in \mathbb F$ such that
  $\lbrace \theta_i\rbrace_{i=0}^d$
  is $(\beta,\gamma)$-recurrent.
\end{enumerate}
\end{lemma}
\begin{proof} 
${\rm (i)}\Rightarrow {\rm (ii)} $ 
For $2\leq i \leq d-1$, the expression
(\ref{eq:betarecS99}) is zero by assumption,
so
\begin{align*}
\theta_{i-2}-\beta \theta_{i-1}+\theta_i = 
\theta_{i-1}-\beta \theta_i+\theta_{i+1}.
\end{align*}
The left-hand side of 
(\ref{eq:gammathreetermS99}) is independent of $i$, and
the result follows.

\noindent 
${\rm (ii)}\Rightarrow {\rm (i)} $ 
For $2\leq i \leq d-1$,
subtract the equation 
(\ref{eq:gammathreetermS99}) at $i$ from the corresponding equation
obtained by replacing $i$ by $i-1$, to find
(\ref{eq:betarecS99}) is zero. 
\end{proof}

\begin{lemma}
\label{lem:bgrecvsbgdrecS99}
The following {\rm (i), (ii)} hold
for all $\beta, \gamma \in \mathbb F$.
\begin{enumerate}
\item[\rm (i)]  Suppose 
  $\lbrace \theta_i\rbrace_{i=0}^d$
is  $(\beta,\gamma)$-recurrent. Then 
there exists $\varrho \in \mathbb F$ such that
  $\lbrace \theta_i\rbrace_{i=0}^d$
 is $(\beta,\gamma,\varrho)$-recurrent.
\item[\rm (ii)]  Suppose 
  $\lbrace \theta_i\rbrace_{i=0}^d$
is  $(\beta,\gamma,\varrho)$-recurrent, and that $\theta_{i-1}\not=\theta_{i+1}$
for $1 \leq i\leq d-1$. Then
  $\lbrace \theta_i\rbrace_{i=0}^d$
  is $(\beta,\gamma)$-recurrent.
\end{enumerate}
\end{lemma}

\begin{proof} 
Let  $p_i$ denote the expression on the left in
(\ref{eq:varrhothreetermS99}),
and observe
 \begin{align*}
p_i-p_{i+1} &= 
(\theta_{i-1}-\theta_{i+1})(\theta_{i-1}-\beta \theta_i +\theta_{i+1} - \gamma)
\end{align*}
for $1 \leq i \leq d-1$. 
Assertions (i), (ii) are both routine consequences of this.
\end{proof}

\section{Recurrent sequences in closed form } 

\noindent In this section,  we obtain some formula
involving recurrent sequences. Let 
$\overline{\mathbb F}$ denote the algebraic closure of
$\mathbb F$.
For $q \in \overline{\mathbb F}$ let
$\mathbb F\lbrack q \rbrack$ denote the field extension of
$\mathbb F$ generated by $q$.

\medskip 

\noindent 
 Throughout this section let
 $\beta$ and 
  $\lbrace \theta_i\rbrace_{i=0}^d$
 denote 
 scalars in $\mathbb F$.

\begin{lemma}
\label{lem:closedformthreetermS99}
Assume that
  $\lbrace \theta_i\rbrace_{i=0}^d$ is $\beta$-recurrent. 
Then the following  {\rm (i)--(iii)} hold.
\begin{enumerate}
\item[\rm (i)]  Suppose $\beta \not=2$, $\beta \not=-2$, and pick
$0 \not=q \in 
\overline{\mathbb F}$ such that 
 $q+q^{-1}=\beta $. Then there exist  scalars 
 $\alpha_1, \alpha_2, \alpha_3$  in  
$\mathbb F\lbrack q \rbrack $ such that
\begin{equation}
\theta_i = \alpha_1 + \alpha_2 q^i + \alpha_3 q^{-i}
\qquad \qquad (0 \leq i \leq d). 
\label{eq:closedformthreetermIS99}
\end{equation}
\item[\rm (ii)] Suppose $\beta = 2$. Then there
exist 
 $\alpha_1, \alpha_2, \alpha_3 $  in $\mathbb F $ such that
\begin{equation}
\theta_i = \alpha_1 + \alpha_2 i + \alpha_3 i(i-1)/2 
\qquad \qquad (0 \leq i \leq d).  
\label{eq:closedformthreetermIIS99}
\end{equation}
\item[\rm (iii)]
Suppose $\beta = -2$ and  ${\rm char}(\mathbb F) \not=2$. Then there
exist 
 $\alpha_1, \alpha_2, \alpha_3 $  in $\mathbb F $ such that
\begin{equation}
\theta_i = \alpha_1 + \alpha_2 (-1)^i + \alpha_3 i(-1)^i 
\qquad \qquad (0 \leq i \leq d).  
\label{eq:closedformthreetermIIIS99}
\end{equation}
\end{enumerate}
\noindent Referring to case {\rm (ii)} above,
if
${\rm char}(\mathbb F) =2$ then we interpret
the expression $i(i-1)/2$ as
                   $0$ if $i=0$ or $i=1$ {\rm (mod $4$)}, 
				  and as
				 $1$ if $i=2$ or $i=3$ {\rm (mod $4$)}.
\end{lemma}
\begin{proof} (i)
We assume $d\geq 2$; otherwise the result is trivial.
Let $q$ be given, and consider the equations 
(\ref{eq:closedformthreetermIS99}) for $i=0,1,2$. These
equations are linear in $\alpha_1, \alpha_2, \alpha_3$.
We routinely find the coefficient matrix is nonsingular,
so there exist
$\alpha_1, \alpha_2, \alpha_3$ in $\mathbb F \lbrack q \rbrack $
such that
(\ref{eq:closedformthreetermIS99})  holds for $i=0,1,2$.
Using these scalars, let $\varepsilon_i $ denote the
left-hand  side of 
(\ref{eq:closedformthreetermIS99}) minus the 
right-hand  side of 
(\ref{eq:closedformthreetermIS99}), for $0 \leq i \leq d$.
On one hand 
$\varepsilon_0$,
$\varepsilon_1$,
$\varepsilon_2$ are zero from the construction.
On the other hand,
one readily checks
\begin{align*}
\varepsilon_{i-2}\,-\,(\beta+1)\varepsilon_{i-1}\,+\,(\beta +1)\varepsilon_i \,-\,\varepsilon_{i+1}=0 
\end{align*}
for 
$2 \leq i \leq d-1$.
By these comments $\varepsilon_i=0$
for $0 \leq i \leq d$, and the result follows.

\noindent (ii), (iii) Similar to the proof of (i) above.
\end{proof}

\begin{lemma}
\label{lem:closedformcommentthreetermS99}
Assume 
that  $\lbrace \theta_i\rbrace_{i=0}^d$ are mutually distinct
and $\beta$-recurrent.
Then {\rm (i)--(iv)} hold below.
\begin{enumerate}
\item[\rm (i)]  Suppose $\beta \not=2$, $\beta \not=-2$,
and pick
$0 \not= q \in 
\overline{\mathbb F}$ such that 
 $q+q^{-1}=\beta $. Then  
$q^i \not=1 $ for 
$1 \leq i \leq d$.   
\item[\rm (ii)] Suppose $\beta = 2$ and ${\rm char}(\mathbb F)=p$, $p\geq 3$.
Then $d<p$. 
\item[\rm (iii)] 
Suppose $\beta = -2$ and  ${\rm char}(\mathbb F) =p$, $p\geq 3$. Then 
 $d<2p$. 
\item[\rm (iv)] 
Suppose $\beta = 0$ and  ${\rm char}(\mathbb F) =2$. Then $d\leq 3$. 
\end{enumerate}
\end{lemma}
\begin{proof} (i) Using
(\ref{eq:closedformthreetermIS99}), we find 
$q^i=1 $ implies $\theta_i=\theta_0$ for $1 \leq i \leq d$.

\noindent (ii) 
Suppose $d\geq p$. Setting $i=p$
in (\ref{eq:closedformthreetermIIS99}) 
and recalling that $p$ is congruent to $0$ modulo $p$,
 we obtain $\theta_p=\theta_0$, a contradiction.
Hence $d<p$.

\noindent (iii)
Suppose $d\geq 2p$. Setting $i=2p$
in 
(\ref{eq:closedformthreetermIIIS99})
and recalling that $p$ is congruent to $0$ modulo $p$,
 we obtain $\theta_{2p}=\theta_0$, a contradiction.
Hence $d<2p$.

\noindent (iv) Suppose $d\geq 4$. Setting $i=4$ in
(\ref{eq:closedformthreetermIIS99}),
we find $\theta_4=\theta_0$
in view of the comment at the end of Lemma
\ref{lem:closedformthreetermS99}.
This is a contradiction, so $d\leq 3$.
\end{proof}

\begin{lemma}
\label{lem:symeigformulaS99}
{\rm (See \cite[Lemma~9.4]{2lintrans}).}
Assume that
  $\lbrace \theta_i\rbrace_{i=0}^d$ are mutually distinct
  and $\beta$-recurrent.
Pick any integers $i,j,r,s$ $(0 \leq i,j,r,s \leq d)$
such that $i+j=r+s$, $r\not=s$.
Then  {\rm (i)--(iv)} hold below.
\begin{enumerate}
\item[\rm (i)] Suppose $\beta \not=2$, $\beta\not=-2$. Then
\begin{equation}
\frac{\theta_i-\theta_{j}}{\theta_r-\theta_s}
= \frac{q^{i}-q^j}{q^r-q^s},
\label{eq:symeiggencasenewthreetermS99}
\end{equation}
where  $q+q^{-1}=\beta $.
\item[\rm (ii)] Suppose $\beta = 2$ and
 ${\rm char}(\mathbb F) \not=2$. 
 Then
\begin{equation}
\frac{\theta_i-\theta_{j}}{\theta_r-\theta_s}
= \frac{i-j}{r-s}.
\label{eq:symeigbeta2newthreetermS99}
\end{equation}
\item[\rm (iii)] Suppose $\beta = -2$ and  ${\rm char}(\mathbb F)\not=2$.
  Then
\begin{equation}
\frac{\theta_i-\theta_{j}}{\theta_r-\theta_s}
=
 \left\{ \begin{array}{ll}
            (-1)^{i+r} \frac{i-j}{r-s},  & \mbox{if $\;i+j\;$ is even}; \\
	(-1)^{i+r},  & \mbox{if $\;i+j\;$ is odd.}
				   \end{array}
				\right. 
\end{equation}
\item[\rm (iv)] Suppose $\beta = 0$ and
 ${\rm char}(\mathbb F) =2$. 
 Then
\begin{equation}
\frac{\theta_i-\theta_{j}}{\theta_r-\theta_s}
= 
 \left\{ \begin{array}{ll}
            0,  & \mbox{if $\;i=j$}; \\
	1,  & \mbox{if $\;i\not=j$.}
				   \end{array}
				\right. 
\end{equation}

\end{enumerate}
\end{lemma}

\begin{proof} To get (i), evaluate the left-hand side of  
(\ref{eq:symeiggencasenewthreetermS99}) using
(\ref{eq:closedformthreetermIS99}), and simplify the
result. The cases (ii)--(iv) are  similar.
\end{proof}

\section{A sum}

\noindent Throughout this section  assume $d\geq 1$. 
 Let $\beta $ and $\lbrace \theta_i\rbrace_{i=0}^d$
 denote scalars in $\mathbb F$ with
 $\lbrace \theta_i\rbrace_{i=0}^d$ mutually distinct.

\medskip

\noindent  
 We consider the sums
\begin{equation}
 \sum_{h=0}^{i-1} \frac{\theta_h-\theta_{d-h}}{\theta_0-\theta_d},
\label{eq:omegaabovediagS99}
\end{equation}
where $0 \leq i \leq d+1$.
Denoting the sum in 
(\ref{eq:omegaabovediagS99}) by $\vartheta_i$, we have
\begin{equation}
 \vartheta_0=0,\qquad \vartheta_1=1,\qquad  \vartheta_d=1,\qquad
\vartheta_{d+1}=0.
\label{eq:varthetaprelimsS99}           
\end{equation}
Moreover 
\begin{equation}
\vartheta_i= \vartheta_{d-i+1} \qquad \qquad (0 \leq i \leq d+1).
\label{eq:varthetaprelims2S99}
\end{equation}

\noindent 
The sums
(\ref{eq:omegaabovediagS99})
play an important role a bit later, so we will 
examine them carefully.
We  begin by giving explicit formulas for
the sums 
(\ref{eq:omegaabovediagS99}) under the assumption
that
 $\lbrace \theta_i\rbrace_{i=0}^d$ 
is $\beta$-recurrent. To avoid trivialities we assume that
$d\geq 3$.
\begin{lemma} 
\label{lem:symeigvalsformulaS99}
{\rm (See \cite[Lemma~10.2]{2lintrans}).}
Assume that 
 $\lbrace \theta_i\rbrace_{i=0}^d$ are mutually distinct and
$\beta$-recurrent. Further assume that $d\geq 3$. 
Then for $0 \leq i \leq d+1$ we have the following.
\begin{enumerate}
\item[\rm (i)] Suppose $\beta \not=2$, $\beta\not=-2$. Then
\begin{equation}
 \sum_{h=0}^{i-1} 
\frac{\theta_h-\theta_{d-h}}{\theta_0-\theta_d}
= \frac{q^i-1}{q-1}\,\frac{q^{d-i+1}-1}{q^d-1},
\label{eq:alphagencaseS99}
\end{equation}
where  $q+q^{-1}=\beta $.
\item[\rm (ii)] Suppose $\beta = 2$ and 
${\rm char}(\mathbb F) \not=2$. Then
\begin{equation}
 \sum_{h=0}^{i-1} 
\frac{\theta_h-\theta_{d-h}}{\theta_0-\theta_d}
= \frac{i(d-i+1)}{d}.
\label{eq:alphacasebeta2S99}
\end{equation}
\item[\rm (iii)] Suppose $\beta = -2$,  
${\rm char}(\mathbb F) \not=2$, and 
$d$ odd. Then
\begin{equation}
 \sum_{h=0}^{i-1} 
\frac{\theta_h-\theta_{d-h}}{\theta_0-\theta_d}
=  \left\{ \begin{array}{ll}
                   0,  & \mbox{if $i$ is even}; \\
				  1, & \mbox{if $i$ is odd.}
				   \end{array}
				\right. 
\label{eq:alphacasebetamin2S99}
\end{equation}                     
\item[\rm (iv)] Suppose $\beta = -2$, 
${\rm char}(\mathbb F) \not=2$, and 
  $d$ even. Then 
\begin{equation}
 \sum_{h=0}^{i-1} 
\frac{\theta_h-\theta_{d-h}}{\theta_0-\theta_d}
 = 
  \left\{ \begin{array}{ll}
                   i/d,  & \mbox{if $i $ is even; } \\
	(d-i+1)/d, \quad 	   & \mbox{if $i$ is odd. }
				   \end{array}
				\right.  
\label{eq:alphacasebetamin2deveS99}
\end{equation}
\item[\rm (v)] Suppose $\beta = 0$, 
${\rm char}(\mathbb F) =2$, and 
  $d=3$. Then 
\begin{equation}
 \sum_{h=0}^{i-1} 
\frac{\theta_h-\theta_{d-h}}{\theta_0-\theta_d}
 = 
  \left\{ \begin{array}{ll}
                   0,  & \mbox{if $i $ is even; } \\
	1, \quad 	   & \mbox{if $i$ is odd. }
				   \end{array}
				\right.  
\label{eq:alphacasebeta0char2S99}
\end{equation}

\end{enumerate}
\end{lemma}

\begin{proof} 
The  above sums
can  be  computed directly from 
Lemma \ref{lem:symeigformulaS99}.

\end{proof}

\begin{note}\rm Referring to Lemma
\ref{lem:symeigvalsformulaS99}, the cases (iii), (iv)
can be handled in
the following uniform way.
Suppose $\beta=-2$ and 
${\rm char}(\mathbb F) \not=2$. Then for $0 \leq i \leq d+1$,
\begin{align*}
 \sum_{h=0}^{i-1} 
\frac{\theta_h-\theta_{d-h}}{\theta_0-\theta_d}
 = \frac{2d+1+(2i-2d-1)(-1)^i +(-1)^d + (2i-1)(-1)^{i+d}}{4d}.
 \end{align*}
 \end{note}

\noindent We make an observation.

\begin{lemma}
\label{lem:varthetacharacS99}
Assume 
that $\lbrace \theta_i\rbrace_{i=0}^d$ 
are mutually distinct and 
 $\beta$-recurrent. 
Define 
\begin{equation}
\vartheta_i = \sum_{h=0}^{i-1} 
\frac{\theta_h-\theta_{d-h}} {\theta_0-\theta_d}
\qquad \qquad (0 \leq i \leq d+1).
\label{eq:remembervarthS99}
\end{equation}
Then the sequence
 $\lbrace \vartheta_i\rbrace_{i=0}^{d+1}$ 
is $\beta$-recurrent.
\end{lemma}

\begin{proof} For $d=1$ there is nothing
to prove. For $d=2$ we have
\begin{align*}
\vartheta_0 -
(\beta+1) \vartheta_1 
+
(\beta+1) \vartheta_2 
-
\vartheta_3 =0
\end{align*}
since
\begin{align*}
 \vartheta_0=0,\qquad \vartheta_1=1,\qquad  \vartheta_2=1,\qquad
\vartheta_{3}=0.
\end{align*}
\noindent For $d\geq 3$ the result is obtained by 
examining the  cases in Lemma
\ref{lem:symeigvalsformulaS99}.
\end{proof}

\begin{proposition}
\label{note:varthetacharacconvS99}
Assume that 
$\lbrace \theta_i\rbrace_{i=0}^d $ are mutually distinct
and $\beta$-recurrent.
Then for scalars  $\lbrace \vartheta_i\rbrace_{i=0}^{d+1}$ 
 in  $\mathbb F$ the following are equivalent:
\begin{enumerate}
\item[\rm (i)] $\displaystyle
\vartheta_i = \vartheta_1 \sum_{h=0}^{i-1} 
\frac{\theta_h-\theta_{d-h}}{\theta_0 - \theta_d}
\qquad \qquad (0 \leq i \leq d+1)$;
\item[\rm (ii)] 
the sequence  $\lbrace \vartheta_i\rbrace_{i=0}^{d+1}$ 
is $\beta$-recurrent and 
\begin{align}
\label{ex:normalize}
\vartheta_0=0, \qquad \quad \vartheta_1 =\vartheta_d, \qquad \quad 
\vartheta_{d+1} = 0.
\end{align}
\end{enumerate}
\end{proposition}
\begin{proof}
${\rm (i)} \Rightarrow {\rm (ii)}$ The sequence
 $\lbrace \vartheta_i\rbrace_{i=0}^{d+1}$ 
is $\beta$-recurrent by Lemma
\ref{lem:varthetacharacS99}. Condition 
(\ref{ex:normalize})
follows  from
(\ref{eq:varthetaprelimsS99}).

\noindent 
${\rm (ii)} \Rightarrow {\rm (i)}$ Define
\begin{align*}
\Delta_i = \vartheta_i - \vartheta_1 
 \sum_{h=0}^{i-1} 
\frac{\theta_h-\theta_{d-h}}{\theta_0 - \theta_d}
\qquad \qquad (0 \leq i \leq d+1).
\end{align*}
We show that $\Delta_i=0$ for $0 \leq i \leq d+1$.
By construction
\begin{align}
\label{eq:start}
&\Delta_0 = 0, \qquad
\Delta_1 = 0, \qquad
\Delta_d = 0, \qquad
\Delta_{d+1} = 0.
\end{align}
For the rest of the proof we assume that $d\geq 3$; otherwise we are done.
By construction and Lemma
\ref{lem:varthetacharacS99}, the sequence
$\lbrace \Delta_i \rbrace_{i=0}^{d+1}$ is $\beta$-recurrent.
We break the argument into cases.
\\
\noindent {\bf Case $\beta \not=2$, $\beta\not=-2$}. Pick 
$0 \not= q \in 
\overline{\mathbb F}$ such that 
 $q+q^{-1}=\beta $. 
There exist
$\alpha_1$,
$\alpha_2$,
$\alpha_3$ in ${\mathbb F}\lbrack q \rbrack$ such that
\begin{align*}
\Delta_i = \alpha_1 + \alpha_2 q^i + \alpha_3 q^{-i} \qquad 
\qquad (0 \leq i \leq d+1).
\end{align*}
Since  $\lbrace \theta_i \rbrace_{i=0}^d$ are mutually distinct
and $\beta$-recurrent, 
we have $q^i \not=1$ for $1 \leq i \leq d$.
The first three equations in (\ref{eq:start}) give
\begin{align*}
	 \left(
	    \begin{array}{c}
	      0 \\
	       0 \\
	       0 
		      \end{array}
		       \right)
		       =
	 \left(
	    \begin{array}{ccc}
	      1 & 1  & 1 \\
	       1 & q & q^{-1} \\
	       1 & q^d  & q^{-d} 
		      \end{array}
		       \right)
	 \left(
	    \begin{array}{c}
	      \alpha_1  \\
	       \alpha_2 \\
	       \alpha_3 
		      \end{array}
		       \right).
\end{align*}
For the above equation the coefficient matrix has determinant
\begin{align*}
(q-1)(q^d-1)(q^{d-1}-1)q^{-d},
\end{align*}
which is nonzero.
Therefore the coefficient matrix is invertible,
so each of 
$\alpha_1$,
$\alpha_2$,
$\alpha_3$ is zero. Consequently $\Delta_i = 0$ for $0 \leq i \leq d+1$.
\\
\noindent {\bf Case $\beta =2$ and ${\rm char}(\mathbb F)\not=2$}.
There exist
$\alpha_1$,
$\alpha_2$,
$\alpha_3$ in ${\mathbb F}$ such that
\begin{align*}
\Delta_i = \alpha_1 + \alpha_2 i + \alpha_3 i(i-1)/2 \qquad 
\qquad (0 \leq i \leq d+1).
\end{align*}
Since $\lbrace \theta_i \rbrace_{i=0}^d$ are mutully distinct
and $\beta$-recurrent, 
we have 
${\rm char}(\mathbb F)=0$ or
${\rm char}(\mathbb F)=p$ with $d<p$.   
The first three equations in (\ref{eq:start}) give
\begin{align*}
	 \left(
	    \begin{array}{c}
	      0 \\
	       0 \\
	       0 
		      \end{array}
		       \right)
		       =
	 \left(
	    \begin{array}{ccc}
	      1 & 0  & 0 \\
	       1 & 1 & 0  \\
	       1 & d  & d(d-1)/2 
		      \end{array}
		       \right)
	 \left(
	    \begin{array}{c}
	      \alpha_1  \\
	       \alpha_2 \\
	       \alpha_3 
		      \end{array}
		       \right).
\end{align*}
For the above equation the coefficient matrix has determinant
$d(d-1)/2$,  which is nonzero.
Therefore the coefficient matrix is invertible,
so each of 
$\alpha_1$,
$\alpha_2$,
$\alpha_3$ is zero. Consequently $\Delta_i = 0$ for $0 \leq i \leq d+1$.
\\
\noindent {\bf Case $\beta =-2$ and ${\rm char}(\mathbb F)\not=2$}.
There exist
$\alpha_1$,
$\alpha_2$,
$\alpha_3$ in ${\mathbb F}$ such that
\begin{align*}
\Delta_i = \alpha_1 + \alpha_2 (-1)^i + \alpha_3 i (-1)^i \qquad 
\qquad (0 \leq i \leq d+1).
\end{align*}
Since $\lbrace \theta_i \rbrace_{i=0}^d$ are mutully distinct
and $\beta$-recurrent, 
we have 
\begin{align}
{\rm char}(\mathbb F)=0 \quad {\rm or} \quad
{\rm char}(\mathbb F)=p, \quad d< 2p.
\label{eq:charDetail}
\end{align}
The first three equations in (\ref{eq:start}) give
\begin{align*}
	 \left(
	    \begin{array}{c}
	      0 \\
	       0 \\
	       0 
		      \end{array}
		       \right)
		       =
	 \left(
	    \begin{array}{ccc}
	      1 & 1  & 0 \\
	       1 & -1 & -1  \\
	       1 & (-1)^d  & d(-1)^d 
		      \end{array}
		       \right)
	 \left(
	    \begin{array}{c}
	      \alpha_1  \\
	       \alpha_2 \\
	       \alpha_3 
		      \end{array}
		       \right).
\end{align*}
For the above equation, consider the determinant of the coefficient matrix.
For even $d=2n$ 
this determinant is $-2^2n$, and 
for odd $d=2n+1$
this determinant is $2^2 n$.
Note that $2 \not=0$ in $\mathbb F$  since
 ${\rm char}(\mathbb F)\not=2$.
For either parity of $d$ we have  $n\not=0$ in $\mathbb F$ by
(\ref{eq:charDetail}).
So for either parity of $d$
the determinant is nonzero.
Therefore the coefficient matrix is invertible,
so each of 
$\alpha_1$,
$\alpha_2$,
$\alpha_3$ is zero. Consequently $\Delta_i = 0$ for $0 \leq i \leq d+1$.
\\
\noindent {\bf Case $\beta =0$ and ${\rm char}(\mathbb F)=2$}.
We have $d=3$ by Lemma
\ref{lem:closedformcommentthreetermS99}(iv).
There exist
$\alpha_1$,
$\alpha_2$,
$\alpha_3$ in ${\mathbb F}$ such that
\begin{align*}
\Delta_i = \alpha_1 + \alpha_2 i + \alpha_3 i (i-1)/2  \qquad 
\qquad (0 \leq i \leq 4),
\end{align*}
where $i(i-1)/2$ is interpreted at the end of Lemma
\ref{lem:closedformthreetermS99}.
The first three equations in (\ref{eq:start}) give
\begin{align*}
	 \left(
	    \begin{array}{c}
	      0 \\
	       0 \\
	       0 
		      \end{array}
		       \right)
		       =
	 \left(
	    \begin{array}{ccc}
	      1 & 0  & 0 \\
	       1 & 1 & 0  \\
	       1 & 1  & 1 
		      \end{array}
		       \right)
	 \left(
	    \begin{array}{c}
	      \alpha_1  \\
	       \alpha_2 \\
	       \alpha_3 
		      \end{array}
		       \right).
\end{align*}
In the above equation the coefficient matrix is invertible,
so each of 
$\alpha_1$,
$\alpha_2$,
$\alpha_3$ is zero. Consequently $\Delta_i = 0$ for $0 \leq i \leq d+1$.

\end{proof}

\section{The polynomial $P(x,y)$}

\noindent Let $\beta, \gamma, \varrho$ denote scalars in $\mathbb F$,
and consider a polynomial in two variables
\begin{align}
P(x,y)= x^2-\beta xy+y^2-\gamma(x+y)-\varrho.
\label{eq:P}
\end{align}
\noindent Note that $P(x,y)= P(y,x)$.
Let $\lbrace \theta_i \rbrace_{i=0}^d$ denote scalars in $\mathbb F$.

\begin{lemma} 
\label{lem:P}
The following are equivalent:
\begin{enumerate}
\item[\rm (i)]
$P(\theta_{i-1},\theta_i)=0 $ for $1 \leq i \leq d$;
\item[\rm (ii)] the sequence $\lbrace \theta_i \rbrace_{i=0}^d$ is
$(\beta, \gamma, \varrho)$-recurrent.
\end{enumerate}
\end{lemma}
\begin{proof} By Definition
\ref{lem:beginthreetermS99}(iv).
\end{proof}

\begin{proposition} 
\label{prop:P}
Assume that 
$\lbrace \theta_i\rbrace_{i=0}^d$ are mutually distinct
and $(\beta, \gamma,\varrho)$-recurrent. Then
the following hold:
\begin{enumerate}
\item[\rm (i)]
$P(x,\theta_j) = (x-\theta_{j-1})(x-\theta_{j+1})
\qquad  (1 \leq j \leq d-1)$;
\item[\rm (ii)]
for $0 \leq i,j\leq d$,
$P(\theta_i, \theta_j) = 0$ implies
$|i-j|=1$ or $i,j\in \lbrace 0,d\rbrace$.
\end{enumerate}
\end{proposition}
\begin{proof}
(i) The polynomial $P(x,\theta_j)$ is monic in $x$,
and has roots $\theta_{j-1}$, $\theta_{j+1}$
by Lemma
\ref{lem:P}.
\\
\noindent (ii) Assume that $P(\theta_i,\theta_j)=0$.
Also assume that $1 \leq i \leq d-1$ or
 $1 \leq j \leq d-1$; otherwise
 $i,j\in \lbrace 0,d\rbrace$ and we are done.
Interchanging $i,j$ if necessary, we may assume
that $1 \leq j \leq d-1$. Using (i) we have
$0= P(\theta_i,\theta_j)= (\theta_i-\theta_{j-1})(\theta_i-\theta_{j+1})$.
Therefore $i=j-1$ or $i=j+1$, so $|i-j|=1$.
\end{proof}

\section{The tridiagonal relations}

In this section, we consider a Leonard system
\begin{align*}
\Phi=(A; \lbrace E_i\rbrace_{i=0}^d; A^*; \lbrace E^*_i\rbrace_{i=0}^d)
\end{align*}
on $V$, with eigenvalue sequence
$\lbrace \theta_i \rbrace_{i=0}^d$ and dual eigenvalue
sequence
$\lbrace \theta^*_i \rbrace_{i=0}^d$.
 Our goal is to prove the following result.

\begin{theorem}
 {\rm (See \cite[Theorem 1.12]{2lintrans}).}
\label{tdptheorem}
There exists a sequence of scalars
$\beta,\gamma,\gamma^*,\varrho,\varrho^*$ taken from $\mathbb F$ such
that both
\begin{description}
\item[\rm (TD1)]
$\qquad 0 = \lbrack A, A^2 A^*-\beta A A^* A+A^* A^2-\gamma\left(
A A^*+A^* A\right)-\varrho A^* \rbrack$,
\item[\rm (TD2)]
$\qquad 0=\lbrack A^*, A^{*2} A-\beta A^*AA^*
+AA^{*2}-\gamma^* \left(A^*A+AA^*\right)
-\varrho^* A\rbrack$.
\end{description}
The sequence is unique if
$d\geq 3$.
\end{theorem}
\noindent The relations (TD1), (TD2) are called the
{\it tridiagonal relations}. They are displayed in
\cite[Lemma~5.4]{tersub3} and  examined carefully in
  \cite{qSerre}.

\begin{lemma} \label{paul12p2}
For  $\beta, \gamma, \varrho \in  \mathbb F$ the
following are equivalent:
\begin{enumerate}
\item[\rm (i)] the scalars 
 $\beta, \gamma, \varrho $ satisfy
{\rm (TD1)}; 
\item[\rm (ii)] the sequence $\lbrace \theta_i\rbrace_{i=0}^d$
is $(\beta,\gamma, \varrho)$-recurrent.
\end{enumerate}
\end{lemma}
\begin{proof}  Let $C$ denote the expression on
the right in
(TD1).
We have
\begin{align*}
C = \sum_{i=0}^d \sum_{j=0}^d E_iCE_j.
\end{align*}
For $0 \leq i,j\leq d$,
\begin{align}
\label{eq:cpiece}
E_iCE_j =  (\theta_i-\theta_j)P(\theta_i, \theta_j)E_iA^*E_j
\end{align}
\noindent where $P$ is from
(\ref{eq:P}).
\\
${\rm (i)} \Rightarrow {\rm (ii)}$ 
We have $C=0$.
So for $1 \leq j \leq d$,
\begin{align*}
0=  E_{j-1}CE_j=
(\theta_{j-1}-\theta_j)P(\theta_{j-1},\theta_j)E_{j-1}A^*E_j.
\end{align*}
By construction  $\theta_{j-1}\not=\theta_j$
 and
$E_{j-1}A^*E_j\not=0$. Therefore
$P(\theta_{j-1},\theta_j)=0$. Consequently the
sequence 
$\lbrace \theta_i\rbrace_{i=0}^d$
is $(\beta,\gamma, \varrho)$-recurrent.
\\
${\rm (ii)} \Rightarrow {\rm (i)}$ 
For $0 \leq i,j\leq d$ the right-hand side of
(\ref{eq:cpiece}) has at least one zero factor, so 
$E_iCE_j=0$. Consequently
$C=0$.
\end{proof}

\begin{lemma} \label{eiarej}
 The following {\rm (i)--(iii)} hold for $0 \leq i,j\leq d$:
\begin{enumerate}
\item[\rm (i)] $E^*_i A^r E^*_j=0$ for $0 \leq r<\vert i-j\vert $;
\item[\rm (ii)] $E^*_iA^rE^*_j \not=0$  for $r=\vert i-j\vert $;
\item[\rm (iii)] for $0 \leq r,s\leq d$,
\begin{align*}
&E^*_i A^r A^* A^s E^*_j=
\begin{cases}
 \theta^*_{j+s} E^*_i A^{r+s} E^*_j,& {\mbox {if $i-j=r+s$}}; \\
 \theta^*_{j-s}E^*_i A^{r+s} E^*_j,&  {\mbox{if $j-i=r+s$}}; \\ 
 0,& {\mbox{ if  $\vert i-j\vert >r+s$}}.
\end{cases}
\end{align*}
\end{enumerate}
\end{lemma}
\begin{proof} For $0 \leq i \leq d$ pick $0 \not=v_i \in E^*_iV$.
So $\lbrace v_i \rbrace_{i=0}^d$ is a basis for $V$.
Without loss of generality, we may identify each  $X\in {\rm End}(V)$
with the matrix in ${\rm Mat}_{d+1}(\mathbb F)$
that represents $X$ with respect to
$\lbrace v_i \rbrace_{i=0}^d$.
From this point of view, $A$ is irreducible tridiagonal
and
$A^* = {\rm diag}(\theta^*_0,\theta^*_1,\ldots,\theta^*_d)$.
Moreover for $0 \leq i \leq d$, 
the matrix
$E^*_i$ is diagonal with $(i,i)$-entry 1 and all other entries 0.
Using these matrix representations, one routinely verifies the 
assertions (i)--(iii) in the lemma statement.
\end{proof}

\noindent Recall that $\mathcal D$ is the subalgebra of
${\rm End}(V)$ generated by $A$.
\begin{lemma}
\label{lem:Li}
Define 
\begin{align*}
L_i = E_0 + E_1 + \cdots + E_i \qquad \qquad (0 \leq i \leq d).
\end{align*}
Then 
\begin{enumerate}
\item[\rm (i)] $\lbrace L_i \rbrace_{i=0}^d$ is a basis for the vector space
$\mathcal D$;
\item[\rm (ii)] $L_d=I$;
\item[\rm (iii)] for $0 \leq i \leq d-1$,
\begin{align*}
L_i A^*-A^*L_i = E_i A^*E_{i+1}-E_{i+1}A^*E_i.
\end{align*}
\end{enumerate}
\end{lemma}
\begin{proof}
(i) Since $\lbrace E_i \rbrace_{i=0}^d$ is a basis for $\mathcal D$.
\\
\noindent (ii) 
Since $I=\sum_{i=0}^d E_i$.
\\
\noindent (iii)  
For $0\leq 
j\leq d-1$ we have
\begin{align}
E_jA^* & = E_j A^* (E_0 + \cdots + E_d)
\nonumber
\\
&=E_jA^*E_{j-1}+E_jA^*E_j+E_jA^*E_{j+1} \label{ejas12p1},
\end{align}
where $E_{-1}=0$. 
Similarly for $0 \leq j \leq d-1$,
\begin{align}
A^*E_j & =  E_{j-1}A^*E_j+E_jA^*E_j+E_{j+1}A^*E_j.
\label{asej12p1}
\end{align}
Sum both (\ref{ejas12p1}) and (\ref{asej12p1}) over
$j=0,1,\ldots,i$ and  take the difference  between these two sums.
\end{proof}

\begin{lemma}
\label{lem:spanSpan}
We have
\begin{align*}
& {\rm Span}\lbrace  XA^*Y-YA^*X \,\vert\, X,Y\in{\mathcal D} \rbrace
= 
\lbrace ZA^*-A^*Z\, \vert\,  Z\in{\mathcal D} \rbrace.
\end{align*}
\end{lemma}
\begin{proof}
Using Lemma
\ref{lem:Li} we obtain
\begin{align*}
{\rm Span}\lbrace XA^*Y&-YA^*X \,|\, X,Y\in{\mathcal D} \rbrace \\
&={\rm Span}\lbrace E_iA^*E_j-E_jA^*E_i \,\vert \, 0\leq i,j\leq d \rbrace \\
&= {\rm Span}\lbrace E_iA^*E_{i+1}-E_{i+1}A^*E_i\,\vert\,0\leq i\leq d-1 \rbrace
\\
&= {\rm Span}\lbrace L_iA^*-A^*L_i \,\vert \,0\leq i\leq d-1 \rbrace \\
&=\{ ZA^*-A^*Z \,\vert \, Z\in{\mathcal D} \rbrace. 
\end{align*}
\end{proof}

\noindent {\it Proof of Theorem \ref{tdptheorem}}.
 First assume that $d\geq 3$. By Lemma
\ref{lem:spanSpan} (with $X=A^2$ and $Y=A$) there exists $Z\in\mathcal D$
such that
\begin{equation} \label{asliebpa}
A^2A^*A-A A^*A^2 = ZA^*-A^*Z.
\end{equation}
Since $\lbrace A^i \rbrace_{i=0}^d$ is a basis for $\mathcal D$,
there exists a
 polynomial $f\in\mathbb F[\lambda]$
such that ${\rm deg}(f)\leq d$ and 
$Z=f(A)$. Let $k$ denote the degree of $f$.

We show that $k=3$. 
We first assume that $k<3$ and get a contradiction.
We multiply each term in
(\ref{asliebpa})
on the left by $E^*_3$ and on the right by
$E^*_0$. We evaluate the result using
Lemma
\ref{eiarej} to  find
$(\theta^*_1-\theta^*_2)E^*_3A^3E^*_0=0$.
The scalar
$ \theta^*_1-\theta^*_2$ is nonzero,
and $E^*_3A^3E^*_0\not=0$ by
Lemma \ref{eiarej}(ii).
Therefore
$ (\theta^*_1-\theta^*_2 )E^*_3A^3E^*_0\not=0$
for a contradiction.
We have shown $k\geq 3$. Let $c$ denote the coefficient of
$\lambda^k$ in $f$. By construction $c\not=0$.

Next we assume that $k>3$ and get a
contradiction. We multiply each term in (\ref{asliebpa}) on the
left by $E^*_k$ and on the right by $E^*_0$. We evaluate the
result using
 Lemma
\ref{eiarej}  to find
$0= c(\theta^*_0-\theta^*_k) E^*_kA^kE^*_0$.
 The scalars
$c$ and
$ \theta^*_0-\theta^*_k$ are nonzero,
and $E^*_kA^kE^*_0\not=0$ by
Lemma \ref{eiarej}(ii).
Therefore $ 0\not=c\left(\theta^*_0-\theta^*_k\right)E^*_kA^kE^*_0$
for
a contradiction.
We have shown $k=3$.

Define $\beta = c^{-1}-1$, so $\beta+1=c^{-1}$.
Multiply each term in  (\ref{asliebpa}) by $c^{-1}$. The
result is
\begin{equation} \label{prebrack}
(\beta+1)(A^2 A^* A-A A^*A^2) =
A^3 A^* -A^*A^3 -\gamma
(A^2 A^*-A^*A^2)-\varrho (AA^*-A^*A),
\end{equation}
where $\gamma, \varrho \in
\mathbb F$.
The equation (\ref{prebrack}) is (TD1) in disguise; it
is (TD1) with the commutator  expanded. 
Therefore $\beta, \gamma,\varrho$ satisfy (TD1). 
Concerning (TD2), 
pick any integer $i$ $(2 \leq i\leq d-1)$.
We multiply each term in
(\ref{prebrack}) 
 on the left by
$E^*_{i-2}$ and on the right by $E^*_{i+1}$. We evaluate
the result using
 Lemma \ref{eiarej} to  find that 
 $E^*_{i-2}A^3E^*_{i+1}$ times
\begin{align}
\label{eq:brecur}
\theta^*_{i-2}-(\beta +1)\theta^*_{i-1}+
(\beta +1)\theta^*_i -\theta^*_{i+1}
\end{align}
is zero.
We have $E^*_{i-2}A^3E^*_{i+1}\not=0$ by
Lemma
\ref{eiarej}(ii), 
so
(\ref{eq:brecur}) is zero.
Thus the sequence
$\lbrace \theta^*_i\rbrace_{i=0}^d$
is $\beta$-recurrent. 
By Lemma \ref{lem:brecvsbgrecS99}
 there exists
$\gamma^* \in \mathbb F$ such that
$\lbrace \theta^*_i\rbrace_{i=0}^d$
is $(\beta, \gamma^*)$-recurrent.
By Lemma \ref{lem:bgrecvsbgdrecS99}(i) 
there exists
$\varrho^* \in \mathbb F$ such that
$\lbrace \theta^*_i\rbrace_{i=0}^d$
is $(\beta, \gamma^*, \varrho^*)$-recurrent.
By this and Lemma \ref{paul12p2} (applied to $\Phi^*$)
we see that $\beta, \gamma^*,
\varrho^*$ satisfy (TD2).

We have obtained scalars
$\beta, \gamma, \gamma^*, \varrho, \varrho^*$ in $\mathbb F$
  that satisfy (TD1), (TD2).
Next we show that these scalars are unique.
Let
 $\beta, \gamma, \gamma^*, \varrho, \varrho^*$
denote any  scalars in $\mathbb F$
  that satisfy (TD1), (TD2).
By Lemma
\ref{paul12p2} the sequence
$\lbrace \theta_i\rbrace_{i=0}^d$
is $(\beta, \gamma, \varrho)$-recurrent.
By Lemma \ref{lem:bgrecvsbgdrecS99}(ii) 
the sequence $\lbrace \theta_i\rbrace_{i=0}^d$
is $(\beta, \gamma)$-recurrent.
By Lemma
\ref{lem:brecvsbgrecS99} 
the sequence $\lbrace \theta_i\rbrace_{i=0}^d$
is $\beta$-recurrent.
Also by 
Lemma \ref{paul12p2} the sequence
$\lbrace \theta^*_i\rbrace_{i=0}^d$
is $(\beta, \gamma^*, \varrho^*)$-recurrent.
By Lemma \ref{lem:bgrecvsbgdrecS99}(ii) 
the sequence $\lbrace \theta^*_i\rbrace_{i=0}^d$
is $(\beta, \gamma^*)$-recurrent.
By Lemma
\ref{lem:brecvsbgrecS99} 
the sequence $\lbrace \theta^*_i\rbrace_{i=0}^d$
 is $\beta$-recurrent.
By these comments and Definition
\ref{lem:beginthreetermS99},
\begin{enumerate}
\item[$\bullet$]
the scalars  
\begin{align*}
\frac{\theta_{i-2}-\theta_{i+1}}{\theta_{i-1}-\theta_i},\qquad
\frac{\theta^*_{i-2}-\theta^*_{i+1}}{\theta^*_{i-1}-\theta^*_i}
\end{align*}
are both equal to $\beta+1$ for $2\le i\leq d-1$;
\item[$\bullet$]  $\gamma=\theta_{i-1}-\beta\theta_i+\theta_{i+1} \qquad
(1\leq  i\le d-1)$,
\item[$\bullet$]
$\gamma^*=\theta^*_{i-1}-\beta\theta^*_i+\theta^*_{i+1}
\qquad 
(1\leq i\leq d-1)$,
\item[$\bullet$]
$\varrho=
\theta^2_{i-1}-\beta\,\theta_{i-1}\,\theta_i
+\theta_{i}^2-\gamma\,(\theta_{i-1}+\theta_i)
\qquad (1 \leq i\leq d)$,
\item[$\bullet$]
$\varrho^*=
\theta_{i-1}^{*2}-\beta\,\theta^*_{i-1}\,\theta^*_i
+\theta_{i}^{*2}-\gamma^*\,(\theta^*_{i-1}+\theta^*_i)
\qquad ( 1 \leq i \leq d)$.
\end{enumerate}
\noindent  The above equations show that
$\beta, \gamma, \gamma^*, \varrho, \varrho^*$ are unique.
We have proved the theorem under the assumption
$d\geq 3$. 

Next assume that $d\leq 2$.
 Pick any  $\beta  \in \mathbb F$. For 
$d=2$ define $\gamma=\theta_0-\beta \theta_1+\theta_2$ and for 
$d\leq 1$ pick any $\gamma \in \mathbb F$. For $d\geq 1$
define
\begin{align*}
\varrho =
 \theta_0^2-\beta\theta_0\theta_1
+\theta_1^2-\gamma (\theta_0+\theta_1)
\end{align*}
and for $d=0$ pick any $\varrho \in \mathbb F$. 
One checks that $\lbrace \theta_i \rbrace_{i=0}^d$ 
is $(\beta,\gamma,\varrho)$-recurrent.
By Lemma
\ref{paul12p2} the scalars $\beta, \gamma, \varrho$ satisfy
(TD1). Replacing $\Phi$ by $\Phi^*$ in the above argument, we
obtain $\gamma^*, \varrho^* \in \mathbb F$ 
such that 
 $\beta, \gamma^*, \varrho^*$ satisfy
 (TD2).
\hfill {$\Box$}
\medskip

\noindent
We emphasize one aspect of the above proof.

\begin{corollary}
\label{cor:equal}
  {\rm (See \cite[Lemma~12.7]{2lintrans}).}
For the Leonard system $\Phi$ the scalars  
\begin{equation}
\frac{\theta_{i-2}-\theta_{i+1}}{\theta_{i-1}-\theta_i},\qquad
\frac{\theta^*_{i-2}-\theta^*_{i+1}}{\theta^*_{i-1}-\theta^*_i}
\end{equation}
are equal and independent of $i$ for 
$2\leq i\leq d-1$.
\end{corollary}

\section{The tridiagonal relations, cont.}

Throughout this section let
$\lbrace \theta_i\rbrace_{i=0}^d$,
$\lbrace \theta^*_i\rbrace_{i=0}^d$,
$\lbrace \varphi_i\rbrace_{i=1}^d$ denote
scalars in $\mathbb F$.
Define matrices $A, A^* \in {\rm Mat}_{d+1} (\mathbb F)$ by
\begin{align*}
       A  =
	 \left(
	    \begin{array}{cccccc}
	      \theta_0 &  & & & & {\bf 0}  \\
	       1 & \theta_1 & &    & &   \\
		 & 1 & \theta_2  & &  &
		  \\
		  && \cdot & \cdot &&
		    \\
		    & & &  \cdot & \cdot & \\
		     {\bf 0}  & &  & & 1 & \theta_d
		      \end{array}
		       \right),
\qquad \qquad
	A^* = \left(
	    \begin{array}{cccccc}
	      \theta^*_0 & \varphi_1  & & & & {\bf 0}  \\
	        & \theta^*_1 & \varphi_2 &    & &   \\
		 &  & \theta^*_2  & \cdot &  &
		  \\
		  && & \cdot & \cdot &
		    \\
		    & & &   &  \cdot  & \varphi_d \\
		     {\bf 0}  & &  & &  & \theta^*_d
		      \end{array}
		       \right).
\end{align*}
\begin{definition}\rm
\label{def:vth} Define the scalars
\begin{align*}
\vartheta_i = \varphi_i - (\theta^*_i-\theta^*_0)(\theta_{i-1}-\theta_d)
\qquad \qquad (1 \leq i \leq d)
\end{align*}
and $\vartheta_0=0$, $\vartheta_{d+1}=0$.
\end{definition}

\begin{lemma}
\label{lem:theentriesofdolangradyS99n}
{\rm (See \cite[Lemma~12.4]{2lintrans}).}
Let $\beta, \gamma,  \varrho $ denote 
scalars in $\mathbb F$, and consider the commutator 
\begin{equation}
\lbrack A,A^2A^*-\beta AA^*A + A^*A^2 -\gamma (AA^*+A^*A)-\varrho A^*\rbrack. 
\label{eq:entriesdolangradyS99n}
\end{equation}
Then the entries of  
{\rm (\ref{eq:entriesdolangradyS99n})}
are as  follows.
\begin{enumerate}
\item[\rm (i)] The $(i+1,i-2)$-entry is 
\begin{align*}
\theta^*_{i-2}\;-\;(\beta+1) \theta^*_{i-1} 
\;+\;(\beta+1) \theta^*_i \;-\;\theta^*_{i+1}
\end{align*}
for $2 \leq i \leq d-1$.
\item[\rm (ii)] The $(i,i-2)$-entry is 
\begin{align*}
&
\vartheta_{i-2}\;-\;(\beta+1) \vartheta_{i-1} 
\;+\;(\beta+1) \vartheta_i \;-\;\vartheta_{i+1}
\\
&+\;(\theta^*_{i-2}-\theta^*_0)
(\theta_{i-3}\;-\;(\beta+1) \theta_{i-2} 
\;+\;(\beta+1) \theta_{i-1} \;-\;\theta_i)
\\
&+\;(\theta_i-\theta_d)
(\theta^*_{i-2}\;-\;(\beta+1) \theta^*_{i-1} 
\;+\;(\beta+1) \theta^*_i \;-\;\theta^*_{i+1})
\\
&+\;(\theta^*_{i-2}-\theta^*_i)
(\theta_{i-2}\;-\;\beta \theta_{i-1} 
\;+\;\theta_i\;-\;\gamma)
\end{align*}
for $2 \leq i \leq d$,  where  
$\lbrace \vartheta_i\rbrace_{i=0}^{d+1}$  are from
Definition 
\ref{def:vth}.
\item[\rm (iii)] The $(i,i-1)$-entry is
\begin{align*}
&\varphi_{i-1}(\theta_{i-2}-\beta \theta_{i-1}+\theta_i-\gamma)\;-\;
\varphi_{i+1}(\theta_{i-1}-\beta \theta_i+\theta_{i+1}-\gamma)
\\
&+\;(\theta^*_{i-1}-\theta^*_i)
(\theta^2_{i-1}-\beta \theta_{i-1}\theta_i + \theta_i^2
-\gamma (\theta_{i-1} +\theta_i) -\varrho)
\end{align*}
for $1 \leq i \leq d$.
\item[\rm (iv)] The $(i,i)$-entry is 
\begin{align*}
&\varphi_i(\theta^2_{i-1}-\beta \theta_{i-1}\theta_i + \theta_i^2
-\gamma (\theta_{i-1} +\theta_i) -\varrho)
\\
&-\;\varphi_{i+1}(\theta^2_i-\beta \theta_i\theta_{i+1} + \theta_{i+1}^2
-\gamma (\theta_{i} +\theta_{i+1}) -\varrho)
\end{align*}
for $0 \leq i \leq d$.
\item[\rm (v)] The $(i-1,i)$-entry is 
\begin{align*}
&&\varphi_i(\theta_{i-1}-\theta_i)(\theta^2_{i-1}-\beta \theta_{i-1}\theta_i + \theta_i^2
-\gamma (\theta_{i-1} +\theta_i) -\varrho)
\end{align*}
for $1 \leq i \leq d$.
\end{enumerate}
All other entries  in 
{\rm (\ref{eq:entriesdolangradyS99n})}
are zero. In the above 
formulas, we assume $\varphi_0=0$, $\varphi_{d+1}=0$,
and that $\theta_{-1}$, $\theta_{d+1}$, 
 $\theta^*_{d+1}$ 
are indeterminates.
\end{lemma}

\begin{proof} Routine matrix multiplication.

\end{proof}

\begin{lemma}
\label{lem:specialCase}
Let $\beta,\gamma, \varrho$ denote scalars in $\mathbb F$.
Assume that 
$\lbrace \theta_i \rbrace_{i=0}^d$ are mutually distinct and
 $(\beta, \gamma,\varrho)$-recurrent.
Assume that $\lbrace \theta^*_i \rbrace_{i=0}^d$ is $\beta$-recurrent.
Then for the commutator {\rm (\ref{eq:entriesdolangradyS99n})}
the $(i,i-2)$-entry is 
\begin{align*}
\vartheta_{i-2} - 
(\beta+1) \vartheta_{i-1}
+(\beta+1) \vartheta_{i}
-
\vartheta_{i+1}
\end{align*}
for $2 \leq i \leq d$, where
$\lbrace \vartheta_i\rbrace_{i=0}^{d+1}$  are from
Definition  \ref{def:vth}.
All other entries in
{\rm (\ref{eq:entriesdolangradyS99n})} are zero.
\end{lemma}
\begin{proof} Examine the entries given in
Lemma
\ref{lem:theentriesofdolangradyS99n}.
\end{proof}

\begin{proposition}
\label{cor:th1d}
With the notation and assumptions of Lemma \ref{lem:specialCase},
\begin{align*}
0 =
\lbrack A,A^2A^*-\beta AA^*A + A^*A^2 -\gamma (AA^*+A^*A)-\varrho A^*\rbrack 
\end{align*}
if and only if $\lbrace \vartheta_i \rbrace_{i=0}^{d+1}$ is $\beta$-recurrent.
\end{proposition}
\begin{proof} By Lemma
\ref{lem:specialCase}.
\end{proof}

\section{The proof of Theorem
\ref{thm:classification}}

\noindent In this section we prove Theorem
\ref{thm:classification}. 
\medskip

\noindent {\it Proof of Theorem 
\ref{thm:classification}}. We may assume that $d\geq 1$;
otherwise the result is vacuous. Assume that there exists
a Leonard system
$\Phi=(A; \lbrace E_i\rbrace_{i=0}^d; A^*; \lbrace E^*_i\rbrace_{i=0}^d)$
over $\mathbb F$ with parameter array
(\ref{eq:parameterArray}).
 We show that this parameter array satisfies
(PA1)--(PA5).
Condition (PA1) holds since $A$ and $A^*$ are
multiplicity-free.
Condition (PA2) holds by the comment below Definition
\ref{def:splitSeq}.
Condition (PA5) holds by Corollary
\ref{cor:equal}.  
By (PA5) and Lemma
\ref{lem:recvsbrecS99},
 there exists
$\beta \in \mathbb F$
such that $\lbrace \theta_i \rbrace_{i=0}^d$ is $\beta$-recurrent 
and $\lbrace \theta^*_i \rbrace_{i=0}^d$ is $\beta$-recurrent.
Concerning (PA3),  Define
\begin{align*}
\vartheta_i = \varphi_i - (\theta^*_i-\theta^*_0)(\theta_{i-1}-\theta_d)
\qquad \qquad (1 \leq i \leq d).
\end{align*}
We show that
\begin{align}
\vartheta_i = \phi_1 \sum_{h=0}^{i-1} \frac{\theta_h - \theta_{d-h}}{\theta_0 - \theta_d}
\qquad \qquad (1 \leq i \leq d).
\label{eq:goalPA3}
\end{align}
To this end we invoke Proposition
\ref{note:varthetacharacconvS99}.
We will show
(i) $\vartheta_1=\phi_1$;
(ii) $\vartheta_d = \phi_1$;
(iii) $\lbrace \vartheta_i\rbrace_{i=0}^{d+1}$
is $\beta$-recurrent, where  
$\vartheta_0 = 0 $ and $\vartheta_{d+1}=0$. 
To show (i), we compare the formulas for $a_0$
in Lemmas
\ref{lem:dg1},
\ref{lem:dg1Down}
to obtain $\varphi_1 - \phi_1 = 
(\theta^*_1-\theta^*_0)(\theta_0 - \theta_d)$.
We have
\begin{align*}
\vartheta_1 = \varphi_1 - (\theta^*_1-\theta^*_0)(\theta_0 - \theta_d) = \phi_1.
\end{align*}
To show (ii),   we compare the formulas for $a^*_d$
in Lemmas
\ref{lem:dg1},
\ref{lem:dg1Down}
to obtain $\varphi_d - \phi_1 = 
(\theta^*_d-\theta^*_0)(\theta_{d-1} - \theta_d)$.
We have
\begin{align*}
\vartheta_d = \varphi_d - (\theta^*_d-\theta^*_0)(\theta_{d-1} - \theta_d) = \phi_1.
\end{align*}
To show (iii) we apply
Proposition
\ref{cor:th1d} 
to the 
 matrix representations of 
$A$, $A^*$ from
Proposition
\ref{prop:splitchar}(iii).
Recall that $\lbrace \theta_i \rbrace_{i=0}^d$
is $\beta$-recurrent. By Lemma
\ref{lem:brecvsbgrecS99} there exists
$\gamma \in \mathbb F$
such that 
$\lbrace \theta_i \rbrace_{i=0}^d$ is
$(\beta, \gamma)$-recurrent.
By Lemma
\ref{lem:bgrecvsbgdrecS99}(i) there exists
$\varrho \in \mathbb F$ such that
$\lbrace \theta_i \rbrace_{i=0}^d$ is
$(\beta, \gamma,\varrho)$-recurrent.
The scalars $\beta, \gamma, \varrho$ satisfy (TD1) by
Lemma
\ref{paul12p2}. 
The assumptions of
Lemma
\ref{lem:specialCase}
are satisfied, so by
Proposition
\ref{cor:th1d} 
the sequence $\lbrace \vartheta_i\rbrace_{i=0}^{d+1}$
is $\beta$-recurrent.
We have shown (i)--(iii). Now 
(\ref{eq:goalPA3}) holds by Proposition
\ref{note:varthetacharacconvS99},
so (PA3) holds.
To obtain (PA4), apply (PA3) to the Leonard system
$\Phi^\Downarrow$, and use
Lemma \ref{prop:pacomp}.
We have obtained (PA1)--(PA5), and we are done in one  direction.

We now reverse the direction. Assume that the scalars
(\ref{eq:parameterArray}) satisfy (PA1)--(PA5). We display
a Leonard system $\Phi$ over $\mathbb F$ that has parameter
array 
(\ref{eq:parameterArray}).
Recall the vector space $V$ with dimension $d+1$.
Pick a basis $\lbrace u_i \rbrace_{i=0}^d$ for $V$.
Define $A, A^* \in {\rm End}(V)$ with the following
matrix representations 
with respect to $\lbrace u_i \rbrace_{i=0}^d$:
\begin{align*}
       A: \quad
	 \left(
	    \begin{array}{cccccc}
	      \theta_0 &  & & & & {\bf 0}  \\
	       1 & \theta_1 & &    & &   \\
		 & 1 & \theta_2  & &  &
		  \\
		  && \cdot & \cdot &&
		    \\
		    & & &  \cdot & \cdot & \\
		     {\bf 0}  & &  & & 1 & \theta_d
		      \end{array}
		       \right),
\qquad \qquad
	A^*:\quad  \left(
	    \begin{array}{cccccc}
	      \theta^*_0 & \varphi_1  & & & & {\bf 0}  \\
	        & \theta^*_1 & \varphi_2 &    & &   \\
		 &  & \theta^*_2  & \cdot &  &
		  \\
		  && & \cdot & \cdot &
		    \\
		    & & &   &  \cdot  & \varphi_d \\
		     {\bf 0}  & &  & &  & \theta^*_d
		      \end{array}
		       \right).
\end{align*}
Observe that $A$ (resp. $A^*$) is multiplicity-free with eigenvalues
$\lbrace \theta_i \rbrace_{i=0}^d$
(resp. $\lbrace \theta^*_i \rbrace_{i=0}^d$);
 for $0 \leq i \leq d$
let $E_i$ (resp. $E^*_i$) denote
the primitive idempotent of $A$ (resp. $A^*$) for
$\theta_i$
(resp. $\theta^*_i$).
The sequence 
$\Phi:=(A; \lbrace E_i\rbrace_{i=0}^d; A^*; \lbrace E^*_i\rbrace_{i=0}^d)
$
is a pre Leonard system on $V$. We show that $\Phi$ is a 
Leonard system on $V$.
To this end we 
show the following (\ref{eq:s1})--(\ref{eq:s9})
for $0 \leq i,j\leq d$:
\begin{align}
&E^*_i A E^*_j = 0 \quad {\rm if} \quad  i-j > 1;
\label{eq:s1}
\\
&E^*_i A E^*_j = 0 \quad {\rm if} \quad  j-i> 1;
\label{eq:s2}
\\
&E^*_i A E^*_j \not= 0 \quad {\rm if} \quad i-j=1;
\label{eq:s6}
\\
&E^*_i A E^*_j \not= 0 \quad {\rm if} \quad j-i=1
\label{eq:s7}
\end{align}
\noindent and
\begin{align}
&E_i A^* E_j = 0 \quad {\rm if} \quad d> i-j> 1;
\label{eq:s4}
\\
&E_i A^* E_j = 0 \quad {\rm if} \quad  d= i-j>1;
\label{eq:s5}
\\
&E_i A^* E_j = 0 \quad {\rm if} \quad j-i > 1;
\label{eq:s3}
\\
&E_i A^* E_j \not= 0 \quad {\rm if} \quad i-j=1;
\label{eq:s8}
\\
&E_i A^* E_j \not= 0 \quad {\rm if} \quad j-i=1.
\label{eq:s9}
\end{align}
Proposition
\ref{prop:splitchar} implies
(\ref{eq:s1}),
(\ref{eq:s6}),
(\ref{eq:s3}).
Lemma
\ref{lem:varphinz} implies
(\ref{eq:s9}).
Before proceeding we make some comments.
The element $E^*_0$ is normalizing
by Proposition
\ref{prop:splitchar}.
By (PA5) and Lemma
\ref{lem:recvsbrecS99},
there exists $\beta \in \mathbb F$ such that
$\lbrace \theta_i\rbrace_{i=0}^d$ is $\beta$-recurrent and
$\lbrace \theta^*_i\rbrace_{i=0}^d$ is $\beta$-recurrent.
By Lemma
\ref{lem:brecvsbgrecS99}
there exists $\gamma \in \mathbb F$ such that
$\lbrace \theta_i\rbrace_{i=0}^d$ is $(\beta,\gamma)$-recurrent.
By 
Lemma \ref{lem:bgrecvsbgdrecS99}(i) there exists
$\varrho \in \mathbb F$ such that
$\lbrace \theta_i\rbrace_{i=0}^d$ is $(\beta,\gamma,\varrho)$-recurrent.
By (PA3), for $1 \leq i \leq d$  we have
\begin{align*}
\varphi_i - (\theta^*_i - \theta^*_0)(\theta_{i-1}-\theta_d)
= 
\phi_1 \sum_{h=0}^{i-1} \frac{\theta_h - \theta_{d-h}}{\theta_0-\theta_d};
\end{align*}
let $\vartheta_i$ denote this common value.
 Note that
$\vartheta_1=\phi_1=
\vartheta_d$. For notational convenience define
$\vartheta_0 = 0$ and
$\vartheta_{d+1} = 0$.

We show (\ref{eq:s4}).
The sequence $\lbrace \vartheta_i\rbrace_{i=0}^{d+1}$
satisfies
Proposition \ref{note:varthetacharacconvS99}(i).
By 
Proposition \ref{note:varthetacharacconvS99}
the sequence $\lbrace \vartheta_i\rbrace_{i=0}^{d+1}$
is $\beta$-recurrent.
Next we apply
Proposition
\ref{cor:th1d}.
We mentioned earlier that $\lbrace \theta_i\rbrace_{i=0}^d$
 is $(\beta, \gamma, \varrho)$-recurrent and
 $\lbrace \theta^*_i\rbrace_{i=0}^d$ is
 $\beta$-recurrent. By these comments
and Proposition
\ref{cor:th1d}, the scalars $\beta, \gamma, \varrho$
satisfy (TD1).
For $0 \leq i,j\leq  d$ we multiply each term in (TD1)
on the left by $E_i$ and on the right by $E_j$.
This yields
\begin{align}
0 = E_i A^* E_j (\theta_i-\theta_j)P(\theta_i,\theta_j),
\label{eq:ttt}
\end{align}
where $P$ is from
(\ref{eq:P}). 
By  Proposition
\ref{prop:P}(ii) we obtain $P(\theta_i,\theta_j) \not=0$ if $d>  i-j>1$.
By this and 
(\ref{eq:ttt}) 
we have $E_i A^*E_j =0$ if $d>  i-j>1$.
We have shown (\ref{eq:s4}). 

We show  
(\ref{eq:s5}). We may assume $d\geq 2$; otherwise there is nothing to prove.
We show $E_d A^*E_0=0$. 
Since $E^*_0$ is normalizing, it suffices to show
that $E_d A^* E_0 E^*_0=0$ in view of Proposition
\ref{prop:normTrick}.
To show that $E_d A^* E_0 E^*_0=0$,
we invoke Proposition
\ref{prop:dg2}.
By (\ref{eq:s4}) we have $E_d A^* E_i=0$ for $1 \leq i \leq d-2$.
We mentioned earlier that $\vartheta_1 =\vartheta_d$.
These comments and Proposition
\ref{prop:dg2} imply
$E_d A^*E_0 E^*_0(\theta_0-\theta_{d-1})=0$. 
The scalar
$\theta_0-\theta_{d-1}$ is nonzero since $d\geq 2$, so 
$E_d A^*E_0 E^*_0=0$.
 We have shown (\ref{eq:s5}).

We show (\ref{eq:s8}).
We must show that
 $E_i A^*E_{i-1}\not=0$ for $1 \leq i \leq d$.
To this end we first apply 
Proposition
\ref{prop:splitchar}
to $\Phi^\Downarrow$.
Proposition
\ref{prop:splitchar}(i) 
 holds for $\Phi^\Downarrow$
by
(\ref{eq:s1}),
(\ref{eq:s6}),
(\ref{eq:s4}),
(\ref{eq:s5}).
So Proposition
\ref{prop:splitchar}(iii) holds for $\Phi^\Downarrow$. 
Consequently there exist
scalars 
$\lbrace \varphi^\Downarrow_i \rbrace_{i=1}^d$ in $\mathbb F$
and a basis for $V$ with respect to which
\begin{align}
       A: \quad 
	 \left(
	    \begin{array}{cccccc}
	      \theta_d &  & & & & {\bf 0}  \\
	       1 & \theta_{d-1} & &    & &   \\
		 & 1 & \theta_{d-2}  & &  &
		  \\
		  && \cdot & \cdot &&
		    \\
		    & & &  \cdot & \cdot & \\
		     {\bf 0}  & &  & & 1 & \theta_0
		      \end{array}
		       \right),
\qquad 
	A^*:\quad
	  \left(
	    \begin{array}{cccccc}
	      \theta^*_0 & \varphi^\Downarrow_1  & & & & {\bf 0}  \\
	        & \theta^*_1 & \varphi^\Downarrow_2 &    & &   \\
		 &  & \theta^*_2  & \cdot &  &
		  \\
		  && & \cdot & \cdot &
		    \\
		    & & &   &  \cdot  & \varphi^\Downarrow_d \\
		     {\bf 0}  & &  & &  & \theta^*_d
		      \end{array}
		       \right).
\label{eq:DD}
\end{align}
We are trying to show that $E_i A^*E_{i-1}\not=0$ for
$1 \leq i \leq d$. By Lemma
\ref{lem:varphinz}
(applied to $\Phi^\Downarrow$), it suffices to show
that $\varphi^\Downarrow_i \not=0$ for $1 \leq i \leq d$.
To this end we show that $\varphi^\Downarrow_i = \phi_i$ for $1 \leq i \leq d$.
By (PA4),
\begin{align*}
\phi_i - (\theta^*_i-\theta^*_0)(\theta_{d-i+1}-\theta_0) =
\varphi_1 \sum_{h=0}^{i-1} \frac{\theta_h - \theta_{d-h}}{\theta_0 - \theta_d}
\qquad (1 \leq i \leq d).
\end{align*}
\noindent Define
\begin{align*}
\vartheta^{\Downarrow}_i = \varphi^\Downarrow_i
- (\theta^*_i - \theta^*_0)(\theta_{d-i+1}-\theta_0) \qquad \qquad (1 \leq i \leq d).
\end{align*}
We show
\begin{align}
\vartheta^\Downarrow_i = \varphi_1 \sum_{h=0}^{i-1} 
\frac{\theta_h - \theta_{d-h}}{\theta_0 - \theta_d}
\qquad \qquad (1 \leq i \leq d).
\label{eq:DDneed}
\end{align}
To show 
(\ref{eq:DDneed}), 
by Proposition
\ref{note:varthetacharacconvS99} it suffices to show
(i)  $\vartheta^\Downarrow_1=\varphi_1$;
(ii)  $\vartheta^\Downarrow_1=
  \vartheta^\Downarrow_d$;
(iii) 
 $\lbrace \vartheta^\Downarrow_i\rbrace_{i=0}^{d+1}$ is $\beta$-recurrent,
 where
$\vartheta^\Downarrow_0=0$ and
$\vartheta^\Downarrow_{d+1}=0$.
To show (i), 
for $\Phi$ and $\Phi^\Downarrow$ we compute
$a_0$ using
Lemma
\ref{lem:dg1};
this yields
$a_0 = \theta_0 + \varphi_1(\theta^*_0 - \theta^*_1)^{-1}$ and
$a_0 = \theta_d + \varphi_1^\Downarrow(\theta^*_0 - \theta^*_1)^{-1}$.
Combining these equations we obtain
$\varphi_1-\varphi^\Downarrow_1 = 
 (\theta^*_1 - \theta^*_0)(\theta_0 - \theta_d)$, so
 \begin{align*}
 \vartheta^\Downarrow_1 = \varphi^\Downarrow_1 - 
 (\theta^*_1 - \theta^*_0)(\theta_d - \theta_0) = \varphi_1.
 \end{align*}
To show (ii), we apply Proposition 
\ref{prop:dg2} 
to $\Phi^\Downarrow$. This gives 
\begin{align*}
\sum_{i=0}^{d-2} E_0 A^* E_{d-i} E^*_0 (\theta_{d-i}-\theta_1)=
E_0 E^*_0 (\vartheta^\Downarrow_1 - \vartheta^\Downarrow_d).
\end{align*}
For this equation the left-hand side is zero, since each summand is
zero by
(\ref{eq:s3}).
We mentioned earlier that $E^*_0$ is 
normalizing, so $E_0 E^*_0\not=0$ by Definition
\ref{def:normalizing}.
By these comments
$\vartheta^\Downarrow_1 =\vartheta^\Downarrow_d$.
To show (iii), we apply
Proposition \ref{cor:th1d} to
the matrices (\ref{eq:DD}).
Recall that $\lbrace \theta_i \rbrace_{i=0}^d$ is
$(\beta, \gamma, \varrho)$-recurrent and 
$\lbrace \theta^*_i \rbrace_{i=0}^d$ is
$\beta$-recurrent. We mentioned earlier that $\beta, \gamma, \varrho$
satisfy (TD1). Applying
Proposition \ref{cor:th1d} to
the matrices (\ref{eq:DD}),
we see that
$\lbrace \vartheta^\Downarrow_i \rbrace_{i=0}^{d+1}$ is
$\beta$-recurrent. We have shown {\rm (i)--(iii)}, 
so (\ref{eq:DDneed}) holds by Proposition
\ref{note:varthetacharacconvS99}.
Now by the construction
$\varphi^\Downarrow_i = \phi_i$ for $1 \leq i \leq d$.
Consequently
$\varphi^\Downarrow_i \not= 0$ for $1 \leq i \leq d$, so
$E_i A^*E_{i-1}\not=0$ for $1 \leq i \leq d$.
We have shown 
 (\ref{eq:s8}).

We show (\ref{eq:s2}) and 
 (\ref{eq:s7}) by invoking
Proposition  
\ref{prop:threefourths}. Consider the conditions
(i)--(iv) in that proposition.
Condition (i) holds by
(\ref{eq:s1}),
(\ref{eq:s6}).
Condition (iii) holds by
(\ref{eq:s4}),
(\ref{eq:s5}),
(\ref{eq:s8}).
Condition (iv) holds by
(\ref{eq:s3}),
(\ref{eq:s9}).
Condition (ii) holds by 
Proposition 
\ref{prop:threefourths}, and this implies
(\ref{eq:s2}) and 
 (\ref{eq:s7}).

We have shown (\ref{eq:s1})--(\ref{eq:s9}), so
$\Phi$ is a Leonard system on $V$.
By construction $\Phi$ has eigenvalue sequence
$\lbrace \theta_i \rbrace_{i=0}^d$,
dual eigenvalue sequence
$\lbrace \theta^*_i \rbrace_{i=0}^d$,
and first split sequence
$\lbrace \varphi_i \rbrace_{i=1}^d$.
By construction 
$\lbrace \varphi^\Downarrow_i \rbrace_{i=1}^d$ is the first split
sequence of $\Phi^\Downarrow$ and hence the
 second split sequence of $\Phi$.
We showed $\varphi^\Downarrow_i = \phi_i $ for $1 \leq i \leq d$, so
$\lbrace \phi_i \rbrace_{i=1}^d$ is the second
split sequence for $\Phi$.
By these comments and Definition
\ref{def:pa}, the sequence
(\ref{eq:parameterArray}) is the parameter
array of $\Phi$.
By Proposition
\ref{cor:paIso}
the Leonard system $\Phi$ is unique up to isomorphism.
\hfill {$\Box$}
\medskip

\section{Characterizations of Leonard systems and parameter arrays}

\noindent We are done discussing Theorem
\ref{thm:classification}. In this section we discuss some 
related results
concerning Leonard systems and parameter arrays.
\medskip

\noindent We comment on notation. Let
$\lbrace u_i\rbrace_{i=0}^d$
and 
$\lbrace v_i\rbrace_{i=0}^d$
denote bases for $V$. By the
{\it transition matrix from 
$\lbrace u_i\rbrace_{i=0}^d$
to $\lbrace v_i\rbrace_{i=0}^d$} we mean 
the matrix 
$M \in {\rm Mat}_{d+1}(\mathbb F)$ such that 
$v_j = \sum_{i=0}^d M_{ij} u_i$ for $0 \leq j \leq d$.
\medskip

\noindent The following result is a variation on
\cite[Theorem~5.1]{ter2005b}.

\begin{theorem} 
\label{thm:thmChar}
Let
$\Phi=(A; \lbrace E_i\rbrace_{i=0}^d; A^*; \lbrace E^*_i\rbrace_{i=0}^d)$
denote a pre Leonard system on $V$,
with eigenvalue sequence 
$\lbrace \theta_i \rbrace_{i=0}^d$
and dual eigenvalue sequence
$\lbrace \theta^*_i \rbrace_{i=0}^d$.
Then $\Phi$ is a Leonard system on $V$ if and only 
if the following {\rm (i)}, {\rm (ii)} hold:
\begin{enumerate}
\item[\rm (i)] there exist nonzero scalars
$\lbrace \varphi_i \rbrace_{i=1}^d$ in $\mathbb F$ 
and a basis for $V$ with respect to which

\begin{align*}
       A: \quad 
	 \left(
	    \begin{array}{cccccc}
	      \theta_0 &  & & & & {\bf 0}  \\
	       1 & \theta_1 & &    & &   \\
		 & 1 & \theta_2  & &  &
		  \\
		  && \cdot & \cdot &&
		    \\
		    & & &  \cdot & \cdot & \\
		     {\bf 0}  & &  & & 1 & \theta_d
		      \end{array}
		       \right),
\qquad \qquad
	A^*: \quad \left(
	    \begin{array}{cccccc}
	      \theta^*_0 & \varphi_1  & & & & {\bf 0}  \\
	        & \theta^*_1 & \varphi_2 &    & &   \\
		 &  & \theta^*_2  & \cdot &  &
		  \\
		  && & \cdot & \cdot &
		    \\
		    & & &   &  \cdot  & \varphi_d \\
		     {\bf 0}  & &  & &  & \theta^*_d
		      \end{array}
		       \right);
\end{align*}
\item[\rm (ii)] there exist nonzero scalars
$\lbrace \phi_i \rbrace_{i=1}^d$ in $\mathbb F$ 
and a basis for $V$ with respect to which

\begin{align*}
       A: \quad 
	 \left(
	    \begin{array}{cccccc}
	      \theta_d &  & & & & {\bf 0}  \\
	       1 & \theta_{d-1} & &    & &   \\
		 & 1 & \theta_{d-2}  & &  &
		  \\
		  && \cdot & \cdot &&
		    \\
		    & & &  \cdot & \cdot & \\
		     {\bf 0}  & &  & & 1 & \theta_0
		      \end{array}
		       \right),
\qquad \qquad
	A^*: \quad \left(
	    \begin{array}{cccccc}
	      \theta^*_0 & \phi_1  & & & & {\bf 0}  \\
	        & \theta^*_1 & \phi_2 &    & &   \\
		 &  & \theta^*_2  & \cdot &  &
		  \\
		  && & \cdot & \cdot &
		    \\
		    & & &   &  \cdot  & \phi_d \\
		     {\bf 0}  & &  & &  & \theta^*_d
		      \end{array}
		       \right).
\end{align*}
\end{enumerate}
\noindent In this case
$\bigl(
\lbrace \theta_i \rbrace_{i=0}^d;
\lbrace \theta^*_i \rbrace_{i=0}^d;
\lbrace \varphi_i \rbrace_{i=1}^d;
\lbrace \phi_i \rbrace_{i=1}^d
\bigr)$
is the parameter array of $\Phi$.
\end{theorem}
\begin{proof} Apply
Proposition 
\ref{lem:tighter} and  Lemma
\ref{lem:varphinz} to both $\Phi$ and
$\Phi^\Downarrow$.
\end{proof}

\noindent Using Theorem
\ref{thm:thmChar} we can easily recover the following 
result.

\begin{theorem} 
{\rm (See \cite[Theorem~3.2]{ter2005})}.
Consider a sequence 
\begin{align}
\bigl(
\lbrace \theta_i \rbrace_{i=0}^d;
\lbrace \theta^*_i \rbrace_{i=0}^d;
\lbrace \varphi_i \rbrace_{i=1}^d;
\lbrace \phi_i \rbrace_{i=1}^d
\bigr)
\label{eq:parameterArray2}
\end{align}
of scalars in $\mathbb F$ that satisfies 
{\rm (PA1)}, {\rm (PA2)}. Then the following are equivalent:
\begin{enumerate}
\item[\rm (i)] the 
sequence 
{\rm (\ref{eq:parameterArray2})} satisfies
{\rm (PA3)--(PA5)};
\item[\rm (ii)] 
there exists an invertible $G \in {\rm Mat}_{d+1}(\mathbb F)$
such that both
\begin{align*}
  G^{-1} \left(
	    \begin{array}{cccccc}
	      \theta_0 &  & & & & {\bf 0}  \\
	       1 & \theta_1 & &    & &   \\
		 & 1 & \theta_2  & &  &
		  \\
		  && \cdot & \cdot &&
		    \\
		    & & &  \cdot & \cdot & \\
		     {\bf 0}  & &  & & 1 & \theta_d
		      \end{array}
		       \right) G = 
	 \left(
	    \begin{array}{cccccc}
	      \theta_d &  & & & & {\bf 0}  \\
	       1 & \theta_{d-1} & &    & &   \\
		 & 1 & \theta_{d-2}  & &  &
		  \\
		  && \cdot & \cdot &&
		    \\
		    & & &  \cdot & \cdot & \\
		     {\bf 0}  & &  & & 1 & \theta_0
		      \end{array}
		       \right),
\end{align*}
\begin{align*}	
	G^{-1} \left(
	    \begin{array}{cccccc}
	      \theta^*_0 & \varphi_1  & & & & {\bf 0}  \\
	        & \theta^*_1 & \varphi_2 &    & &   \\
		 &  & \theta^*_2  & \cdot &  &
		  \\
		  && & \cdot & \cdot &
		    \\
		    & & &   &  \cdot  & \varphi_d \\
		     {\bf 0}  & &  & &  & \theta^*_d
		      \end{array}
		       \right) G=
	 \left(
	    \begin{array}{cccccc}
	      \theta^*_0 & \phi_1  & & & & {\bf 0}  \\
	        & \theta^*_1 & \phi_2 &    & &   \\
		 &  & \theta^*_2  & \cdot &  &
		  \\
		  && & \cdot & \cdot &
		    \\
		    & & &   &  \cdot  & \phi_d \\
		     {\bf 0}  & &  & &  & \theta^*_d
		      \end{array}
		       \right).
\end{align*}
\end{enumerate}
\end{theorem}
\begin{proof} 
${\rm (i)} \Rightarrow {\rm (ii)}$ 
By Theorem
\ref{thm:classification} there exists a Leonard
system on $V$ with parameter array
(\ref{eq:parameterArray2}).
The matrix $G$ is the transition matrix from
a basis for $V$ that satisfies
Theorem \ref{thm:thmChar}(i), to a basis
for $V$ that satisfies
Theorem \ref{thm:thmChar}(ii).
\\
\noindent 
${\rm (ii)} \Rightarrow {\rm (i)}$ 
Pick a basis $\lbrace u_i \rbrace_{i=0}^d$ for $V$.
Define $A, A^* \in {\rm End}(V)$ whose
 matrix representations 
with respect to $\lbrace u_i \rbrace_{i=0}^d$ are 
from Theorem \ref{thm:thmChar}(i).
Observe that $A$ (resp. $A^*$) is multiplicity-free with eigenvalues
$\lbrace \theta_i \rbrace_{i=0}^d$
(resp. $\lbrace \theta^*_i \rbrace_{i=0}^d$);
 for $0 \leq i \leq d$
let $E_i$ (resp. $E^*_i$) denote
the primitive idempotent of $A$ (resp. $A^*$) for
$\theta_i$
(resp. $\theta^*_i$).
The sequence 
$\Phi:=(A; \lbrace E_i\rbrace_{i=0}^d; A^*; \lbrace E^*_i\rbrace_{i=0}^d)
$
is a pre Leonard system on $V$.
We show that $\Phi$ is a 
Leonard system on $V$. By linear algebra there exists a basis
$\lbrace v_i \rbrace_{i=0}^d$ of $V$ such that
$G$ is the
transition matrix from 
$\lbrace u_i \rbrace_{i=0}^d$ to 
$\lbrace v_i \rbrace_{i=0}^d$.
The matrix representations of $A$ and $A^*$
with respect to 
$\lbrace v_i \rbrace_{i=0}^d$
are shown in
 Theorem \ref{thm:thmChar}(ii).
The pre Leonard system $\Phi$ satisfies
 Theorem \ref{thm:thmChar}(i) and
 Theorem \ref{thm:thmChar}(ii), so by that theorem $\Phi$
 is a Leonard system on $V$. Also by that theorem
(\ref{eq:parameterArray2}) is the parameter array of $\Phi$,
so 
(\ref{eq:parameterArray2})
satisfies (PA3)--(PA5).
\end{proof}

\section{The intersection numbers}

\noindent 
We now bring in the intersection numbers.
Throughout this section 
$\Phi=(A; \lbrace E_i\rbrace_{i=0}^d; A^*; \lbrace E^*_i\rbrace_{i=0}^d)$
denotes a Leonard system 
on $V$, with parameter array
$\bigl(
\lbrace \theta_i \rbrace_{i=0}^d;
\lbrace \theta^*_i \rbrace_{i=0}^d;
\lbrace \varphi_i \rbrace_{i=1}^d;
\lbrace \phi_i \rbrace_{i=1}^d
\bigr)$.
Applying Lemma
\ref{lem:normal}
to $\Phi^*$ we see that
$E_0$ is normalizing. For
$0 \not=\xi \in E_0V$ the vectors $\lbrace E^*_i \xi\rbrace_{i=0}^d$
form a basis for $V$, said to be {\it $\Phi$-standard}.
With respect to this basis
\begin{align*}
       A: \quad 
	 \left(
	    \begin{array}{cccccc}
	      a_0 & b_0 & & & & {\bf 0}  \\
	       c_1 & a_1 & b_1 &    & &   \\
		 & c_2 & \cdot  & \cdot &  &
		  \\
		  && \cdot & \cdot &\cdot &
		    \\
		    & & &  \cdot & \cdot & b_{d-1} \\
		     {\bf 0}  & &  & & c_d & a_d
		      \end{array}
		       \right),
\qquad \qquad
	A^*: \quad {\rm diag}(\theta^*_0,\theta^*_1,\ldots, \theta^*_d),
\end{align*}
where $\lbrace a_i\rbrace_{i=0}^d$ 
are from
Definition \ref{def:ai}
and
$\lbrace b_i\rbrace_{i=0}^{d-1}$,
$\lbrace c_i\rbrace_{i=1}^{d}$ are nonzero scalars in $\mathbb F$.
The vector $\xi$ is an eigenvector for $A$ with
eigenvalue $\theta_0$. Moreover $\xi=\sum_{i=0}^d E^*_i\xi$.
Consequently
\begin{align}
                  c_i + a_i + b_i = \theta_0 \qquad \qquad (0 \leq i \leq d),
\label{eq:crs}
\end{align}
where $c_0=0$ and $b_d=0$.
We call 
$\lbrace b_i\rbrace_{i=0}^{d-1}$,
$\lbrace c_i\rbrace_{i=1}^{d}$ the {\it intersection
numbers} of $\Phi$.
They are discussed in
\cite[Section~11]{ter2004}.
The intersection numbers
$\lbrace b^*_i\rbrace_{i=0}^{d-1}$,
$\lbrace c^*_i\rbrace_{i=1}^{d}$ of $\Phi^*$ are
called the {\it dual intersection
numbers} of $\Phi$. By construction
\begin{align}
                  c^*_i + a^*_i + b^*_i = \theta^*_0 \qquad \qquad (0 \leq i \leq d),
\label{eq:crsd}
\end{align}
where $c^*_0=0$ and $b^*_d=0$.
With respect to a $\Phi^*$-standard basis for $V$,
\begin{align}
	A: \quad {\rm diag}(\theta_0,\theta_1,\ldots, \theta_d),
\qquad \qquad 
A^*: \quad 
	 \left(
	    \begin{array}{cccccc}
	      a^*_0 & b^*_0 & & & & {\bf 0}  \\
	       c^*_1 & a^*_1 & b^*_1 &    & &   \\
		 & c^*_2 & \cdot  & \cdot &  &
		  \\
		  && \cdot & \cdot &\cdot &
		    \\
		    & & &  \cdot & \cdot & b^*_{d-1} \\
		     {\bf 0}  & &  & & c^*_d & a^*_d
		      \end{array}
		       \right).
\label{eq:abcs}
\end{align}
We mention a handy recurrence.
\begin{lemma}
\label{lem:rec3}
Assume $d\geq 1$. Then for $0 \leq i \leq d$ we have
\begin{align}
&c_i \theta^*_{i-1} + 
a_i \theta^*_{i} + 
b_i \theta^*_{i+1} = \theta_1 \theta^*_i + a^*_0(\theta_0-\theta_1),
\label{eq:rec3}
\\
&c^*_i \theta_{i-1} + 
a^*_i \theta_{i} + 
b^*_i \theta_{i+1} = \theta^*_1 \theta_i + a_0(\theta^*_0-\theta^*_1),
\label{eq:rec3s}
\end{align}
where
$\theta_{-1}$,
$\theta_{d+1}$,
$\theta^*_{-1}$,
$\theta^*_{d+1}$ denote indeterminates.
\end{lemma}
\begin{proof} The proof of
(\ref{eq:rec3}) is similar to the proof of 
(\ref{eq:crs}). To obtain
(\ref{eq:rec3}),
use the fact that
for $0 \not=\xi \in E_0V$ the vector
$(A^*-a^*_0 I) \xi$ is an eigenvector for $A$ with eigenvalue
$\theta_1$, and
$(A^*-a^*_0 I) \xi = \sum_{i=0}^d (\theta^*_i-a^*_0)E^*_i \xi$.
To get (\ref{eq:rec3s}), apply
(\ref{eq:rec3}) to $\Phi^*$.
\end{proof}

\begin{lemma} 
\label{lem:solvebc}
For $d\geq 1$ we have
\begin{align*}
 b_0 &= \theta_0-a_0,
\\
b_i &= \frac{(a_i - \theta_0)(\theta^*_i - \theta^*_{i-1})+ 
(\theta_0-\theta_1)(\theta^*_i-a^*_0)}{\theta^*_{i-1} - \theta^*_{i+1}}
\qquad (1 \leq i \leq d-1),
\\
c_i &= \frac{(a_i - \theta_0)(\theta^*_i - \theta^*_{i+1})+ 
(\theta_0-\theta_1)(\theta^*_i-a^*_0)}{\theta^*_{i+1} - \theta^*_{i-1}}
\qquad (1 \leq i \leq d-1),
\\
c_d &= \theta_0 - a_d.
\end{align*}
To get 
$\lbrace b^*_i\rbrace_{i=0}^{d-1}$
and $\lbrace c^*_i\rbrace_{i=1}^d$, exchange starred and nonstarred symbols
everywhere in the above equations.
\end{lemma}
\begin{proof} To obtain $b_0$ set
$i=0$ in
(\ref{eq:crs}). To obtain
$b_i$, $c_i$ for $1\leq i \leq d-1$,
solve the linear equations 
(\ref{eq:crs}),
(\ref{eq:rec3}).
To obtain $c_d$ set
$i=d$ in
(\ref{eq:crs}).
To obtain 
$\lbrace b^*_i\rbrace_{i=0}^{d-1}$
and $\lbrace c^*_i\rbrace_{i=1}^d$,
apply the above comments to $\Phi^*$.
\end{proof}

\noindent In Lemma
\ref{lem:rec3} we gave a recurrence. 
We now give a more general recurrence.

\begin{lemma} 
\label{lem:Tmat} For $0 \leq i\leq d$ and $1\leq j \leq d$,
\begin{align}
\label{eq:tijs}
c_i \tau^*_j(\theta^*_{i-1})+
a_i \tau^*_j(\theta^*_{i})+
b_i \tau^*_j(\theta^*_{i+1}) 
&= \theta_j \tau^*_j(\theta^*_i) + \varphi_j \tau^*_{j-1}(\theta^*_i),
\\
\label{eq:etaijs}
c_i \eta^*_j(\theta^*_{i-1})+
a_i \eta^*_j(\theta^*_{i})+
b_i \eta^*_j(\theta^*_{i+1}) 
&= \theta_j \eta^*_j(\theta^*_i) + \phi_{d-j+1} \eta^*_{j-1}(\theta^*_i),
\\
\label{eq:tij}
c^*_i \tau_j(\theta_{i-1})+
a^*_i \tau_j(\theta_{i})+
b^*_i \tau_j(\theta_{i+1}) 
&= \theta^*_j \tau_j(\theta_i) + \varphi_j \tau_{j-1}(\theta_i),
\\
\label{eq:etaij}
c^*_i \eta_j(\theta_{i-1})+
a^*_i \eta_j(\theta_{i})+
b^*_i \eta_j(\theta_{i+1}) 
&= \theta^*_j \eta_j(\theta_i) + \phi_j \eta_{j-1}(\theta_i),
\end{align}
where $\theta_{-1}$, $\theta_{d+1}$,
 $\theta^*_{-1}$, $\theta^*_{d+1}$ denote indeterminates.
\end{lemma}
\begin{proof} Concerning
(\ref{eq:tij}),
Let $T \in {\rm Mat}_{d+1}(\mathbb F)$
have $(i,j)$-entry $\tau_j(\theta_i)$ for $0 \leq i,j\leq d$.
For $0 \not=\xi \in E^*_0V$,
$T$ is the transition matrix from the
$\Phi^*$-standard basis
$\lbrace E_i \xi \rbrace_{i=0}^d$ to the
basis
$\lbrace \tau_i(A)\xi \rbrace_{i=0}^d$.
The matrix $B^*$ on the right in
(\ref{eq:abcs}) represents $A^*$ with respect to
$\lbrace E_i \xi \rbrace_{i=0}^d$.
The matrix $D^*$ on the right in
(\ref{eq:split}) represents $A^*$ 
with respect to 
$\lbrace \tau_i(A)\xi \rbrace_{i=0}^d$.
By linear algebra $B^*T = T D^*$.
The entries of this matrix give the equations 
(\ref{eq:tij}).
To obtain 
(\ref{eq:etaij}), in the above argument
replace the basis
$\lbrace \tau_i(A)\xi \rbrace_{i=0}^d$ by the basis
$\lbrace \eta_i(A)\xi \rbrace_{i=0}^d$,
and replace the matrix on the right in (\ref{eq:split}) 
by the matrix on the right in
(\ref{eq:phibasis}).
To obtain
(\ref{eq:tijs}) and
(\ref{eq:etaijs}),
apply 
(\ref{eq:tij}) and
(\ref{eq:etaij}) to $\Phi^*$.
\end{proof}

\begin{lemma}
\label{lem:bbcc}
{\rm(See \cite[Theorem~17.7]{ter2004}).}
We have
\begin{enumerate}
\item[\rm (i)] $b_i = \varphi_{i+1} 
\frac{\tau^*_i(\theta^*_i)}
{\tau^*_{i+1}(\theta^*_{i+1})} \qquad (0 \leq i \leq d-1)$;
\item[\rm (ii)] $c_i = \phi_{i} \frac{\eta^*_{d-i}(\theta^*_i)}
{\eta^*_{d-i+1}(\theta^*_{i-1})} \qquad (1 \leq i \leq d)$;
\item[\rm (iii)] $b^*_i = \varphi_{i+1} 
\frac{\tau_i(\theta_i)}
{\tau_{i+1}(\theta_{i+1})} \qquad (0 \leq i \leq d-1)$;
\item[\rm (iv)] $c^*_i = \phi_{d-i+1} \frac{\eta_{d-i}(\theta_i)}
{\eta_{d-i+1}(\theta_{i-1})} \qquad (1 \leq i \leq d)$.
\end{enumerate}
\end{lemma}
\begin{proof} (i) Set $j=i+1$ in 
(\ref{eq:tijs}).
\\
\noindent (ii) Set $j=d-i+1$ in 
(\ref{eq:etaijs}).
\\
\noindent (iii) Set $j=i+1$ in 
(\ref{eq:tij}).
\\
\noindent (ii) Set $j=d-i+1$ in 
(\ref{eq:etaij}).
\end{proof}

\noindent In Lemma
\ref{lem:bbcc} we see some fractions. In order to simplify
these fractions,  we consider the products
\begin{align}
\psi_i = \prod_{h=0}^{i-2} \frac{\theta_i - \theta_{h+1}}{\theta_{i+1}-\theta_h}
\qquad \qquad (1 \leq i \leq d-1).
\label{eq:psin}
\end{align}
Note that $\psi_1 = 1$. 
Using Definition
\ref{def:taui} we find that 
for $1 \leq i \leq d-1$,
\begin{align}
\frac{\tau_i(\theta_i)}
{\tau_{i+1}(\theta_{i+1})} = \frac{\psi_i (\theta_i - \theta_0)}
{(\theta_{i+1}-\theta_i)(\theta_{i+1}-\theta_{i-1})}.
\label{eq:tauVSetas}
\end{align}

\noindent We now describe the scalars
$\lbrace \psi_i \rbrace_{i=1}^{d-1}$ in
detail. To avoid trivialities we assume that $d\geq 3$.
Let $\beta+1$ denote the common value of
(\ref{eq:defbetaplusoneS99int}).
\begin{lemma}
\label{prop:prodformn}
For $d\geq 3$ 
and $1 \leq i \leq d-1$ we have the following.
\begin{enumerate}
\item[\rm (i)] Suppose $\beta \not=2$, $\beta\not=-2$. Then
\begin{align*}
\psi_i = q^{i-1}\frac{q-1}{q^i-1}\,\frac{q^{2}-1}{q^{i+1}-1},
\end{align*}
where  $q+q^{-1}=\beta $.
\item[\rm (ii)] Suppose $\beta = 2$ and 
${\rm char}(\mathbb F) \not=2$. Then
\begin{align*}
\psi_i = \frac{2}{i(i+1)}.
\end{align*}
\item[\rm (iii)] Suppose $\beta = -2$ and 
${\rm char}(\mathbb F) \not=2$.
 Then
\begin{align*}
\psi_i
=  \left\{ \begin{array}{ll}
                   -2/i,  & \mbox{if $i$ is even; } \\
				  2/(i+1), & \mbox{if $i$ is odd.}
				   \end{array}
				\right. 
\end{align*}                     
\item[\rm (iv)] Suppose $\beta = 0$, 
${\rm char}(\mathbb F) =2$, and 
  $d=3$. Then 
\begin{align*}
\psi_i = 1.
\end{align*}
\end{enumerate}
\end{lemma}
\begin{proof} 
Evaluate the right-hand side of
(\ref{eq:psin}) using
Lemma
\ref{lem:symeigformulaS99}, and simplify the result.
\end{proof}

\begin{note}\rm Under the assumptions of Lemma
\ref{prop:prodformn}(iii),
\begin{align*}
\psi_i =\frac{4 (-1)^i}{(-1)^i-1-2i} \qquad \qquad (1 \leq i \leq d-1).
\end{align*}
\end{note}

\begin{corollary}
\label{cor:psiDown}
For $1\leq i \leq d-1$,
\begin{align*}
\frac{\tau^*_i(\theta^*_i)}
{\tau^*_{i+1}(\theta^*_{i+1})} &= \frac{\psi_i (\theta^*_i - \theta^*_0)}
{(\theta^*_{i+1}-\theta^*_i)(\theta^*_{i+1}-\theta^*_{i-1})},
\\
\frac{\eta_{d-i}(\theta_i)}
{\eta_{d-i+1}(\theta_{i-1})} &= \frac{\psi_{d-i} (\theta_i - \theta_d)}
{(\theta_{i-1}-\theta_i)(\theta_{i-1}-\theta_{i+1})},
\\
\frac{\eta^*_{d-i}(\theta^*_i)}
{\eta^*_{d-i+1}(\theta^*_{i-1})} &= \frac{\psi_{d-i} (\theta^*_i - \theta^*_d)}
{(\theta^*_{i-1}-\theta^*_i)(\theta^*_{i-1}-\theta^*_{i+1})}.
\end{align*}
\end{corollary}
\begin{proof} The result is vacuous for $d\leq 1$
and trivial for $d=2$. Next assume that $d\geq 3$.
From the data in
Lemma \ref{prop:prodformn}, we see that for $1 \leq i \leq d-1$
the scalar $\psi_i$ depends only on $i$ and
$\beta$. Consequently $\psi_i$ is unchanged if we replace $\Phi$
by $\Phi^*$ or $\Phi^\Downarrow$. The result follows from
these comments and
(\ref{eq:tauVSetas}).
\end{proof}

\begin{proposition}
\label{prop:cibiform}
\rm
For $d\geq 1$ we have
\begin{align*}
b_0 &= \frac{\varphi_1}{\theta^*_1-\theta^*_0},
\\
b_i &= \frac{\varphi_{i+1} \psi_i (\theta^*_i-\theta^*_0)}{
(\theta^*_{i+1}-\theta^*_i) 
(\theta^*_{i+1}-\theta^*_{i-1}) 
}
\qquad  (1 \leq i \leq d-1),
\\
c_i &= \frac{\phi_{i} \psi_{d-i} (\theta^*_i-\theta^*_d)}{
(\theta^*_{i-1}-\theta^*_i) 
(\theta^*_{i-1}-\theta^*_{i+1}) 
}
\qquad  (1 \leq i \leq d-1),
\\
c_d &= \frac{\phi_d}{\theta^*_{d-1}-\theta^*_d}
\end{align*}
\noindent and
\begin{align*}
b^*_0 &= \frac{\varphi_1}{\theta_1-\theta_0},
\\
b^*_i &= \frac{\varphi_{i+1} \psi_i (\theta_i-\theta_0)}{
(\theta_{i+1}-\theta_i) 
(\theta_{i+1}-\theta_{i-1}) 
}
\qquad  (1 \leq i \leq d-1),
\\
c^*_i &= \frac{\phi_{d-i+1} \psi_{d-i} (\theta_i-\theta_d)}{
(\theta_{i-1}-\theta_i) 
(\theta_{i-1}-\theta_{i+1}) 
}
\qquad  (1 \leq i \leq d-1),
\\
c^*_d &= \frac{\phi_1}{\theta_{d-1}-\theta_d}.
\end{align*}
\end{proposition}
\begin{proof}
Evaluate the formulas in Lemma
\ref{lem:bbcc} using
(\ref{eq:tauVSetas})
and
Corollary
\ref{cor:psiDown}.
\end{proof}

\noindent We mention some attractive formulas involving
the intersection numbers; similar formulas apply to
the dual intersection numbers.

\begin{lemma}
\label{lem:frac}
The following hold:
\begin{enumerate}
\item[\rm (i)] for $0 \leq i \leq d-1$,
\begin{align}
\label{eq:hold1}
\frac{
\theta_0 + \theta_1 + \cdots + \theta_i - a_0 -a_1-\cdots - a_i}
{b_i} 
= \prod_{h=0}^{i-1}\frac{\theta^*_{i+1}-\theta^*_h}{\theta^*_i - \theta^*_h};
\end{align}
\item[\rm (ii)] for $1 \leq i \leq d$,
\begin{align}
\label{eq:hold2}
\frac{
a_0 +a_1+\cdots + a_{i-1}
-
\theta_d - \theta_{d-1} - \cdots - \theta_{d-i+1} 
}{c_i} 
= \prod_{h=i+1}^{d}\frac{\theta^*_{i-1}-\theta^*_h}
{\theta^*_i - \theta^*_h}.
\end{align}
\end{enumerate}
\end{lemma}
\begin{proof} (i) To verify
(\ref{eq:hold1}), eliminate $b_i$ using
Lemma
\ref{lem:bbcc}(i), and evaluate the result using
Lemma
\ref{lem:sumaieiequalsvarphiS99}.
\\
\noindent 
(ii) To verify
(\ref{eq:hold2}), eliminate $c_i$ using
Lemma
\ref{lem:bbcc}(ii), and evaluate the result using
Lemma
\ref{lem:sumaieiequalsphiS99}.
\end{proof}

\begin{corollary} 
\label{cor:c1}
For $d\geq 1$,
\begin{align}
\label{eq:bdm1}
&b_{d-1} = (a_d-\theta_d)\frac{
(\theta^*_{d-1} - \theta^*_0)
(\theta^*_{d-1} - \theta^*_1)
\cdots
(\theta^*_{d-1} - \theta^*_{d-2})}
{(\theta^*_d - \theta^*_0)
(\theta^*_d - \theta^*_1)
\cdots
(\theta^*_d - \theta^*_{d-2})},
\\
\label{eq:c1}
&c_1 = (a_0 - \theta_d)\frac{
(\theta^*_1 - \theta^*_2)
(\theta^*_1 - \theta^*_3)
\cdots
(\theta^*_1 - \theta^*_d)}
{(\theta^*_0 - \theta^*_2)
(\theta^*_0 - \theta^*_3)
\cdots
(\theta^*_0 - \theta^*_d)}.
\end{align}
\end{corollary}
\begin{proof} 
To get
(\ref{eq:bdm1})
set $i=d-1$ in
(\ref{eq:hold1}),
 and simplify the result using
(\ref{eq:aithi}).
To get 
(\ref{eq:c1})
set $i=1$ in
(\ref{eq:hold2}).
\end{proof}

\noindent 
Next we obtain $\lbrace \theta_i \rbrace_{i=0}^d$
in terms of the
intersection numbers and dual eigenvalues. We give two versions,
that resemble \cite[Corollary~8.3.3]{BCN}.
\begin{lemma}
\label{lem:vv1}
For $d\geq 1$,
\begin{align*}
\theta_0 &= a_0 + b_0,
\\
\theta_i &= a_i +
b_i \prod_{h=0}^{i-1}\frac{\theta^*_{i+1}-\theta^*_h}{\theta^*_i-\theta^*_h}
-
b_{i-1} \prod_{h=0}^{i-2}\frac{\theta^*_{i}-\theta^*_h}{\theta^*_{i-1}-\theta^*_h}
\qquad (1 \leq i \leq d-1),
\\
\theta_d &= a_d -
b_{d-1} \prod_{h=0}^{d-2}\frac{\theta^*_{d}-\theta^*_h}{\theta^*_{d-1}-\theta^*_h}.
\end{align*}
\end{lemma}
\begin{proof} To find $\theta_i$ for $0 \leq i \leq d-1$,
solve
(\ref{eq:hold1}) for $\theta_i$ and use induction on $i$.
The formula for $\theta_d$ comes from
(\ref{eq:bdm1}).
\end{proof}

\begin{lemma}
\label{lem:vv2}
For $d\geq 1$,
\begin{align*}
\theta_0 &= a_d + c_d,
\\
\theta_i &= a_{d-i} +
c_{d-i} \prod_{h=d-i+1}^{d}\frac{\theta^*_{d-i-1}-\theta^*_h}{\theta^*_{d-i}-\theta^*_h}
-
c_{d-i+1} \prod_{h=d-i+2}^{d}\frac{\theta^*_{d-i}-\theta^*_h}{\theta^*_{d-i+1}-\theta^*_h}
\qquad (1 \leq i \leq d-1),
\\
\theta_d &= a_0 -
c_1 \prod_{h=2}^d\frac{\theta^*_0-\theta^*_h}{\theta^*_1-\theta^*_h}.
\end{align*}
\end{lemma}
\begin{proof} To find $\theta_d, \theta_{d-1},\ldots, \theta_1$
use 
(\ref{eq:hold2}).
The formula for $\theta_0$ comes from
(\ref{eq:crs}) at $i=d$.
\end{proof}

\noindent Next we obtain some results about duality.

\begin{lemma}
\label{lem:thdual}
For $0 \leq i,j,r,s\leq d$ such that $i+j=r+s$ and $r\not=s$,
\begin{align*}
\frac{\theta_i - \theta_j}{\theta_r-\theta_s}
=
\frac{\theta^*_i - \theta^*_j}{\theta^*_r-\theta^*_s}.
\end{align*}
\end{lemma}
\begin{proof}
Use the data in Lemma
\ref{lem:symeigformulaS99}.
\end{proof}

\begin{lemma}
\label{lem:bdual}
For $0 \leq i \leq d-1$,
\begin{align*}
&b_i \frac{
(\theta^*_{i+1}-\theta^*_0)
(\theta^*_{i+1}-\theta^*_1)
\cdots (\theta^*_{i+1}-\theta^*_i)}
{(\theta^*_i-\theta^*_0)
(\theta^*_i-\theta^*_1)
\cdots (\theta^*_i-\theta^*_{i-1})}
\\
& \qquad \qquad =\quad b^*_i \frac{
(\theta_{i+1}-\theta_0)
(\theta_{i+1}-\theta_1)
\cdots (\theta_{i+1}-\theta_i)}
{(\theta_i-\theta_0)
(\theta_i-\theta_1)
\cdots (\theta_i-\theta_{i-1})}.
\end{align*}
\end{lemma}
\begin{proof}
Each side is equal to $\varphi_{i+1}$ by Lemma
\ref{lem:bbcc}(i),(iii).
\end{proof}

\begin{lemma} 
\label{lem:dd2}
For $d\geq 1$,
\begin{align*}
b_0(\theta^*_1-\theta^*_0) &= 
b^*_0(\theta_1-\theta_0),
\\
\frac{b_i 
(\theta^*_{i+1}-\theta^*_i)
(\theta^*_{i+1}-\theta^*_{i-1})
}{
\theta^*_i-\theta^*_0}
&=
\frac{b^*_i 
(\theta_{i+1}-\theta_i)
(\theta_{i+1}-\theta_{i-1})
}{
\theta_i-\theta_0}
\qquad (1 \leq i \leq d-1).
\end{align*}
\end{lemma}
\begin{proof} 
Compare the formulas for $b_i$, $b^*_i$ in
Proposition
\ref{prop:cibiform}.
\end{proof}

\begin{lemma}
\label{lem:cdual}
For $0 \leq i \leq d-1$,
\begin{align*}
&c_{i+1} \frac{
(\theta^*_{i}-\theta^*_d)
(\theta^*_{i}-\theta^*_{d-1})
\cdots (\theta^*_{i}-\theta^*_{i+1})}
{(\theta^*_{i+1}-\theta^*_d)
(\theta^*_{i+1}-\theta^*_{d-1})
\cdots (\theta^*_{i+1}-\theta^*_{i+2})}
\\
& \qquad \qquad =\quad c^*_{d-i} \frac{
(\theta_{d-i-1}-\theta_d)
(\theta_{d-i-1}-\theta_{d-1})
\cdots (\theta_{d-i-1}-\theta_{d-i})}
{(\theta_{d-i}-\theta_d)
(\theta_{d-i}-\theta_{d-1})
\cdots (\theta_{d-i}-\theta_{d-i+1})}.
\end{align*}
\end{lemma}
\begin{proof} Each side is equal to
$\phi_{i+1}$ by Lemma
\ref{lem:bbcc}(ii),(iv).
\end{proof}

\begin{lemma}
\label{lem:adual}
For $0 \leq i \leq d-1$,
\begin{align*}
\sum_{h=0}^i \frac{\theta_h - a_h}{\theta_i - \theta_{i+1}} &=
\sum_{h=0}^i \frac{\theta^*_h - a^*_h}{\theta^*_i - \theta^*_{i+1}},
\\
\sum_{h=0}^i \frac{\theta_{d-h} - a_h}{\theta_{d-i} - \theta_{d-i-1}} &=
\sum_{h=0}^i \frac{\theta^*_h - a^*_{d-h}}{\theta^*_i - \theta^*_{i+1}},
\\
\sum_{h=0}^i \frac{\theta_h - a_{d-h}}{\theta_i - \theta_{i+1}} &=
\sum_{h=0}^i \frac{\theta^*_{d-h} - a^*_h}{\theta^*_{d-i} - \theta^*_{d-i-1}},
\\
\sum_{h=0}^i \frac{\theta_{d-h} - a_{d-h}}{\theta_{d-i} - \theta_{d-i-1}} &=
\sum_{h=0}^i \frac{\theta^*_{d-h} - a^*_{d-h}}{\theta^*_{d-i} - \theta^*_{d-i-1}}.
\end{align*}
\end{lemma}
\begin{proof}
The first equation holds, because each side
is equal to
$-\varphi_{i+1}(\theta_{i+1}-\theta_i)^{-1}(\theta^*_{i+1}-\theta^*_i)^{-1}$
by Lemma
\ref{lem:sumaieiequalsvarphiS99}. The remaining equations
are similarly obtained.
\end{proof}

\noindent We emphasize a special case of Lemma
\ref{lem:adual}.

\begin{lemma}
\label{lem:a0}
For $d\geq 1$,
\begin{align*}
&\frac{\theta_0-a_0}{\theta_0-\theta_1}=
\frac{\theta^*_0-a^*_0}{\theta^*_0-\theta^*_1},
\qquad \qquad
\frac{\theta_d-a_0}{\theta_{d}-\theta_{d-1}}=
\frac{\theta^*_0-a^*_d}{\theta^*_0-\theta^*_1},
\\
&\frac{\theta_0-a_d}{\theta_0-\theta_1}=
\frac{\theta^*_d-a^*_0}{\theta^*_{d}-\theta^*_{d-1}},
\qquad \qquad
\frac{\theta_d-a_d}{\theta_{d}-\theta_{d-1}}=
\frac{\theta^*_d-a^*_d}{\theta^*_{d}-\theta^*_{d-1}}.
\end{align*}
\end{lemma}
\begin{proof} Set $i=0$ in Lemma
\ref{lem:adual}.
\end{proof}

\noindent  Motivated by Lemma
\ref{cor:iso1}, our next goal is to
explicitly write
the intersection numbers and dual intersection numbers
in terms of 
$\varphi_1$ and $\lbrace \theta_i \rbrace_{i=0}^{d}$,
 $\lbrace \theta^*_i \rbrace_{i=0}^{d}$.
To avoid trivialities we assume that $d\geq 2$.
We will use the following result,
which appears in \cite[Lemma~15.16]{cerzo} in the context of
distance-regular graphs.

\begin{lemma}
\label{lem:cerzo}
For $d\geq 2$, $\varphi_2$ is equal to each of
\begin{align}
&\varphi_1 \frac{\theta_0 + \theta_1-\theta_{d-1}-\theta_d}{\theta_0-\theta_d}
+ (\theta^*_0-\theta^*_1)(\theta_0+\theta_{1}-\theta_{d-1} - \theta_d)
+(\theta^*_2-\theta^*_0)(\theta_1-\theta_d),
\label{eq:var2}
\\
&\varphi_1 \frac{\theta^*_0 + \theta^*_1-\theta^*_{d-1}-\theta^*_d}{\theta^*_0-\theta^*_d}
+ (\theta_0-\theta_1)(\theta^*_0+\theta^*_{1}-\theta^*_{d-1} - \theta^*_d)
+(\theta_2-\theta_0)(\theta^*_1-\theta^*_d).
\label{eq:var2s}
\end{align}
\end{lemma}
\begin{proof}
To get (\ref{eq:var2}), 
evaluate (PA3) at $i=2$ and simplify the result using
(PA4) at $i=1$.
To get (\ref{eq:var2s}) from
(\ref{eq:var2}),
note that each of $\varphi_1, \varphi_2$
is unchanged if we replace $\Phi$ by $\Phi^*$.
\end{proof}

\noindent  
The following result appears in
\cite[Theorem~8.1.1]{BCN} and
\cite[Theorem~15.18]{cerzo}, in the context of distance-regular graphs.

\begin{proposition} 
\label{prop:bici}
For $d\geq 2$,
\begin{align*}
b_0 &= \frac{\varphi_1}{\theta^*_1-\theta^*_0},
\\
b_i &= \frac{(\theta^*_0 - \theta^*_i)(\varphi_1 f^-_i + g^-_i)}{
(\theta^*_{i+1}-\theta^*_i)(\theta^*_{i+1}-\theta^*_{i-1})}
\qquad (1 \leq i \leq d-1),
\\
c_i &= \frac{(\theta^*_0 - \theta^*_i)(\varphi_1 f^+_i + g^+_i)}{
(\theta^*_{i-1}-\theta^*_i)(\theta^*_{i-1}-\theta^*_{i+1})}
\qquad (1 \leq i \leq d-1),
\\
c_d &= \frac{\varphi_1+ (\theta_0 - \theta_1)
(\theta^*_{0}-\theta^*_d)}{\theta^*_{d-1}-\theta^*_d}
\end{align*}
where
\begin{align*}
f^{\pm}_i &= 
\frac{\theta^*_1-\theta^*_{i\pm 1}}{\theta^*_0-\theta^*_i}
-
\frac{\theta^*_1-\theta^*_{d-1}}{\theta^*_0-\theta^*_d},
\\
g^{\pm}_i &=
(\theta_1-\theta_2)(\theta^*_i-\theta^*_d)-
(\theta_0-\theta_1)(\theta^*_{i\pm 1}-\theta^*_{d-1}).
\end{align*}
To get 
$\lbrace b^*_i\rbrace_{i=0}^{d-1}$ and
$\lbrace c^*_i\rbrace_{i=1}^d$, exchange 
the symbols $\theta$, $\theta^*$ everywhere in the above equations.
\end{proposition}
\begin{proof} To get $b_0$, set $i=0$ in
Lemma
\ref{lem:bbcc}(i).
To get $c_d$, set $i=d$ in
Lemma
\ref{lem:bbcc}(ii) and eliminate $\phi_d$ using
(PA4).
We now compute $b_i, c_i$ for $1 \leq i \leq d-1$.
Let $i$ be given. 
For the equations
(\ref{eq:tijs}) at $j=1$ and $j=2$,
eliminate $a_i$ using
(\ref{eq:crs}) to obtain a system of 
two linear equations in the
unknowns $b_i$, $c_i$. Using linear algebra we routinely
solve this system for $b_i$, $c_i$, and 
in the solution eliminate $\varphi_2$ using
(\ref{eq:var2s}).
We have obtained
$\lbrace b_i\rbrace_{i=0}^{d-1}$ and
$\lbrace c_i\rbrace_{i=1}^d$.
To obtain 
$\lbrace b^*_i\rbrace_{i=0}^{d-1}$ and
$\lbrace c^*_i\rbrace_{i=1}^d$, apply the above arguments to
$\Phi^*$.
\end{proof}

\noindent Our next goal is to give a variation on Proposition
\ref{prop:bici} that involves
$a_0$, $a^*_0$  and 
$c_1$, $c^*_1$.

\begin{lemma}
\label{lem:v2two}
For $d\geq 2$,
$\varphi_2$ is equal to each of 
\begin{align}
(c_1-a_0+ \theta_1)(\theta^*_2-\theta^*_0),
\qquad \qquad 
(c^*_1-a^*_0+ \theta^*_1)(\theta_2-\theta_0).
\label{eq:vp2twover}
\end{align}
\end{lemma}
\begin{proof} Set $i=1$ in 
(\ref{eq:crs}),
Lemma
\ref{lem:bbcc}(i),
(\ref{eq:hold1}). The resulting equations show  that $\varphi_2$
is equal to the expression on the left in
(\ref{eq:vp2twover}). 
Note that $\varphi_2$ is unchanged if we replace $\Phi$
by $\Phi^*$. Therefore
 $\varphi_2$
is equal to the expression on the right in
(\ref{eq:vp2twover}). 
\end{proof}

\noindent We clarify how
$c_1, a_0$ are related and how
$c^*_1, a^*_0$ are related.

\begin{lemma}
\label{lem:c1}
For $d\geq 2$ we have
\begin{align*}
c_1 &= (\theta_d-a_0)\frac{
(\theta^*_0-\theta^*_1)(\theta_1-\theta_{d-1})
-
(\theta^*_1-\theta^*_2)(\theta_0-\theta_d)
}
{(\theta^*_0-\theta^*_2)(\theta_0-\theta_d)},
\\
c^*_1 &= (\theta^*_d-a^*_0)\frac{
(\theta_0-\theta_1)(\theta^*_1-\theta^*_{d-1})
-
(\theta_1-\theta_2)(\theta^*_0-\theta^*_d)
}
{(\theta_0-\theta_2)(\theta^*_0-\theta^*_d)}.
\end{align*}
\end{lemma}
\begin{proof} To verify the first equation,
express the left-hand side in terms of $\varphi_2$  using
Lemma
\ref{lem:v2two}, and in the resulting equation
express $a_0$ in terms of $\varphi_1$ 
using 
Lemma
\ref{lem:dg1}.
Evaluate the result using
(\ref{eq:var2}).
We have verified the first equation.
To get the second equation, apply the
first equation to $\Phi^*$.
\end{proof}

\begin{proposition}
\label{prop:bcform}
For $d\geq 2$,
\begin{align*}
b_0 &= \frac{(\theta^*_0-a^*_0)(\theta_1-\theta_0)}
{\theta^*_1-\theta^*_0},
\\
b_i &=\frac{
c^*_1(\theta^*_0-\theta^*_i)(\theta_0-\theta_2)+
(\theta^*_i-a^*_0)
\bigl(
(\theta_1-\theta_2)(\theta^*_0-\theta^*_i)-
(\theta_0-\theta_1)(\theta^*_1-\theta^*_{i-1})
\bigr)
}
{
(\theta^*_{i+1}-\theta^*_i)(\theta^*_{i+1}-\theta^*_{i-1})},
\\
c_i &=\frac{
c^*_1(\theta^*_0-\theta^*_i)(\theta_0-\theta_2)+
(\theta^*_i-a^*_0)
\bigl(
(\theta_1-\theta_2)(\theta^*_0-\theta^*_i)-
(\theta_0-\theta_1)(\theta^*_1-\theta^*_{i+1})
\bigr)
}
{
(\theta^*_{i-1}-\theta^*_i)(\theta^*_{i-1}-\theta^*_{i+1})},
\\
c_d &=\frac{(\theta^*_d-a^*_0)(\theta_1-\theta_0)}
{\theta^*_{d-1}-\theta^*_{d}}.
\end{align*}
To get 
$\lbrace b^*_i\rbrace_{i=0}^{d-1}$ and
$\lbrace c^*_i\rbrace_{i=1}^d$, 
exchange starred and nonstarred symbols everywhere in the
above equations.
\end{proposition}
\begin{proof} To get $b_0$,
set $i=0$ in Lemma \ref{lem:bbcc}(i) and evaluate
the result
using $\varphi_1=(\theta^*_0-a^*_0)(\theta_1-\theta_0)$.
To get $c_d$, set 
 $i=d$ in Lemma \ref{lem:bbcc}(ii) and evaluate
the result using
$\phi_d=(\theta^*_d-a^*_0)(\theta_1-\theta_0)$.
To obtain $b_i, c_i$ for $1 \leq i \leq d-1$,
we proceed as in the proof of Proposition
\ref{prop:bici}.
Let $i$ be given. 
For the equations
(\ref{eq:tijs}) at $j=1$ and $j=2$,
eliminate $a_i$ using
(\ref{eq:crs}) to obtain a system of 
two linear equations in the
unknowns $b_i$, $c_i$. Using linear algebra we routinely
solve this system for $b_i$, $c_i$, and 
in the solution eliminate $\varphi_1$, $\varphi_2$   
using $\varphi_1=(\theta^*_0-a^*_0)(\theta_1-\theta_0)$
and
$\varphi_2=(c^*_1-a^*_0+\theta^*_1)(\theta_2-\theta_0)$.
We have obtained
$\lbrace b_i\rbrace_{i=0}^{d-1}$ and
$\lbrace c_i\rbrace_{i=1}^d$.
To obtain 
$\lbrace b^*_i\rbrace_{i=0}^{d-1}$ and
$\lbrace c^*_i\rbrace_{i=1}^d$, apply the above arguments to
$\Phi^*$.
\end{proof}

\begin{note}\rm 
Lemma
\ref{lem:c1}
and Proposition 
\ref{prop:bcform} effectively give 
the intersection numbers (resp. dual intersection numbers)
in terms of
$a^*_0$ (resp. $a_0$)  and
$\lbrace \theta_i \rbrace_{i=0}^d$, 
$\lbrace \theta^*_i \rbrace_{i=0}^d$. The scalars
$a_0$ and $a^*_0$ are related by the first equation in
Lemma
\ref{lem:a0}.
\end{note}

\begin{note}\rm
There is a Leonard system attached to the Bose-Mesner
algebra 
of a $Q$-polynomial distance-regular graph;
see  
\cite[p.~260]{bannaiIto},
\cite{leonard}.
 For this
 Leonard system
$a_0=0$,
$a^*_0=0$
and 
$c_1=1$,
$c^*_1=1$.
\end{note}

\section{Appendix: Parameter arrays and intersection numbers}

The parameter arrays are listed in
\cite[Section~5]{ter2005}.
In this appendix we go through the list, and for each
parameter array we give the corresponding
intersection numbers and 
dual intersection numbers.

\begin{example} \rm ($q$-Racah). We have
\begin{align*}
\theta_i &= \theta_0 + h (1-q^i)(1-s q^{i+1}) q^{-i},
\\
\theta^*_i &= \theta^*_0 + h^* (1-q^i)(1-s^* q^{i+1}) q^{-i}
\end{align*}
\noindent for $0 \leq i \leq d$ and
\begin{align*}
\varphi_i &= h h^* q^{1-2i} (1-q^i)(1-q^{i-d-1})(1-r_1 q^i)(1-r_2 q^i),
\\
\phi_i &= h h^* q^{1-2i} (1-q^i)(1-q^{i-d-1})(r_1-s^* q^i)(r_2-s^* q^i)/s^*
\end{align*}
for $1 \leq i \leq d$, with 
$r_1 r_2 = s s^* q^{d+1}$.
We have
\begin{align*}
b_0 &= \frac{h(1-q^{-d})(1-r_1 q)(1-r_2q)}{1-s^*q^2},
\\
b_i &= \frac{h(1-q^{i-d})(1-s^*q^{i+1})(1-r_1 q^{i+1})(1-r_2q^{i+1})}
{(1-s^*q^{2i+1})(1-s^*q^{2i+2})}
\qquad \qquad (1 \leq i \leq d-1),
\\
c_i &= \frac{h(1-q^{i})(1-s^*q^{i+d+1})(r_1-s^* q^{i})(r_2-s^*q^{i})}
{s^*q^d(1-s^*q^{2i})(1-s^*q^{2i+1})}
\qquad \qquad (1 \leq i \leq d-1),
\\
c_d &= \frac{h (1-q^{d})(r_1-s^* q^d)(r_2-s^*q^d)}{s^*q^d(1-s^*q^{2d})}.
\end{align*}
To get
$\lbrace b^*_i\rbrace_{i=0}^{d-1}$
and $\lbrace c^*_i\rbrace_{i=1}^{d}$, in the above formulas
exchange 
$h \leftrightarrow h^*$,
$s \leftrightarrow s^*$
and preserve
$r_1$, $r_2$, $q$.
\end{example}

\begin{example} \rm ($q$-Hahn). We have 
\begin{align*}
\theta_i &= \theta_0 + h (1-q^i)q^{-i},
\\
\theta^*_i &= \theta^*_0 + h^* (1-q^i)(1-s^* q^{i+1}) q^{-i}
\end{align*}
\noindent for $0 \leq i \leq d$ and
\begin{align*}
\varphi_i &= h h^* q^{1-2i} (1-q^i)(1-q^{i-d-1})(1-r q^i),
\\
\phi_i &= -h h^* q^{1-i} (1-q^i)(1-q^{i-d-1})(r-s^* q^i)
\end{align*}
for $1 \leq i \leq d$.
We have
\begin{align*}
b_0 &=\frac{h(1-q^{-d})(1-rq)}{1-s^*q^2},
\\
b_i &= \frac{h(1-q^{i-d})(1-s^*q^{i+1})(1-r q^{i+1})}
{(1-s^*q^{2i+1})(1-s^*q^{2i+2})}
\qquad \qquad (0 \leq i \leq d-1),
\\
c_i &= \frac{-hq^{i-d}(1-q^{i})(1-s^*q^{i+d+1})(r-s^* q^{i})}
{(1-s^*q^{2i})(1-s^*q^{2i+1})}
\qquad \qquad (1 \leq i \leq d),
\\
c_d &= \frac{-h (1-q^{d})(r-s^* q^d)}{1-s^*q^{2d}}
\end{align*}
and
\begin{align*}
b^*_i &= h^*(1-q^{i-d})(1-r q^{i+1})
\qquad \qquad (0 \leq i \leq d-1),
\\
c^*_i &= h^*(1-q^{i})(qs^*-rq^{i-d})
\qquad \qquad (1 \leq i \leq d).
\end{align*}
\end{example}

\begin{example} \rm (Dual $q$-Hahn). We have 
\begin{align*}
\theta_i &= \theta_0 + h (1-q^i)(1-sq^{i+1})q^{-i},
\\
\theta^*_i &= \theta^*_0 + h^* (1-q^i) q^{-i}
\end{align*}
\noindent for $0 \leq i \leq d$ and
\begin{align*}
\varphi_i &= h h^* q^{1-2i} (1-q^i)(1-q^{i-d-1})(1-r q^i),
\\
\phi_i &= h h^* q^{d+2-2i} (1-q^i)(1-q^{i-d-1})(s-r q^{i-d-1})
\end{align*}
for $1 \leq i \leq d$.
We have
\begin{align*}
b_i &= h(1-q^{i-d})(1-r q^{i+1})
\qquad \qquad (0 \leq i \leq d-1),
\\
c_i &= h(1-q^{i})(qs-rq^{i-d})
\qquad \qquad (1 \leq i \leq d)
\end{align*}
\noindent
and
\begin{align*}
b^*_0 &= \frac{h^*(1-q^{-d})(1-r q)}{1-sq^2},
\\
b^*_i &= \frac{h^*(1-q^{i-d})(1-sq^{i+1})(1-r q^{i+1})}
{(1-sq^{2i+1})(1-sq^{2i+2})}
\qquad \qquad (1 \leq i \leq d-1),
\\
c^*_i &= \frac{-h^*q^{i-d}(1-q^{i})(1-sq^{i+d+1})(r-s q^{i})}
{(1-sq^{2i})(1-sq^{2i+1})}
\qquad \qquad (1 \leq i \leq d-1),
\\
c^*_d &= \frac{-h^* (1-q^{d})(r-s q^d)}{1-s q^{2d}}.
\end{align*}
\end{example}

\begin{example}
\rm
(quantum $q$-Krawtchouk). We have
\begin{align*}
\theta_i &= \theta_0 - sq (1-q^i),
\\
\theta^*_i &= \theta^*_0 + h^* (1-q^i) q^{-i}
\end{align*}
\noindent for $0 \leq i \leq d$ and
\begin{align*}
\varphi_i &= -r h^* q^{1-i} (1-q^i)(1-q^{i-d-1}),
\\
\phi_i &= h^* q^{d+2-2i} (1-q^i)(1-q^{i-d-1})(s-r q^{i-d-1})
\end{align*}
for $1 \leq i \leq d$.
We have
\begin{align*}
b_i &= -rq^{i+1}(1-q^{i-d})
\qquad \qquad (0 \leq i \leq d-1),
\\
c_i &= (1-q^{i})(qs-rq^{i-d})
\qquad \qquad (1 \leq i \leq d)
\end{align*}
\noindent
and
\begin{align*}
b^*_i &= \frac{h^*r (1-q^{i-d})}
{sq^{2i+1}}
\qquad \qquad (0 \leq i \leq d-1),
\\
c^*_i &= \frac{h^*(1-q^{i})(r-s q^{i})}
{sq^{2i}}
\qquad \qquad (1 \leq i \leq d).
\end{align*}
\end{example}

\begin{example}
\rm
($q$-Krawtchouk). We have
\begin{align*}
\theta_i &= \theta_0 + h (1-q^i)q^{-i},
\\
\theta^*_i &= \theta^*_0 + h^* (1-q^i)(1-s^*q^{i+1}) q^{-i}
\end{align*}
\noindent for $0 \leq i \leq d$ and
\begin{align*}
\varphi_i &= h h^* q^{1-2i} (1-q^i)(1-q^{i-d-1}),
\\
\phi_i &= h h^* s^* q (1-q^i)(1-q^{i-d-1})
\end{align*}
for $1 \leq i \leq d$.
We have
\begin{align*}
b_0 &= \frac{h(1-q^{-d})}{1-s^*q^2},
\\
b_i &= \frac{h (1-q^{i-d})(1-s^*q^{i+1})}{(1-s^*q^{2i+1})(1-s^*q^{2i+2})}
\qquad \qquad (1 \leq i \leq d-1),
\\
c_i &= \frac{h s^*q^{2i-d} (1-q^{i})(1-s^*q^{i+d+1})}{(1-s^*q^{2i})(1-s^*q^{2i+1})}
\qquad \qquad (1 \leq i \leq d-1),
\\
c_d &= \frac{h s^*q^{d} (1-q^{d})}{1-s^*q^{2d}}
\end{align*}
\noindent
and
\begin{align*}
b^*_i &= h^* (1-q^{i-d})
\qquad \qquad (0 \leq i \leq d-1),
\\
c^*_i &= h^*s^* q(1-q^{i})
\qquad \qquad (1 \leq i \leq d).
\end{align*}
\end{example}

\begin{example}
\rm
(affine $q$-Krawtchouk). We have
\begin{align*}
\theta_i &= \theta_0 + h (1-q^i)q^{-i},
\\
\theta^*_i &= \theta^*_0 + h^* (1-q^i) q^{-i}
\end{align*}
\noindent for $0 \leq i \leq d$ and
\begin{align*}
\varphi_i &= h h^* q^{1-2i} (1-q^i)(1-q^{i-d-1})(1-rq^i),
\\
\phi_i &= -h h^* r q^{1-i} (1-q^i)(1-q^{i-d-1})
\end{align*}
for $1 \leq i \leq d$.
We have
\begin{align*}
b_i &= h (1-q^{i-d})(1-rq^{i+1})
\qquad \qquad (0 \leq i \leq d-1),
\\
c_i &= -h r q^{i-d}(1-q^{i})
\qquad \qquad (1 \leq i \leq d).
\end{align*}
\noindent To get
$\lbrace b^*_i\rbrace_{i=0}^{d-1}$ and
$\lbrace c^*_i\rbrace_{i=1}^{d}$,
in the above formulas exchange 
$h \leftrightarrow h^*$
and preserve
$r$, $q$.
\end{example}

\begin{example}
\rm
(dual $q$-Krawtchouk). We have
\begin{align*}
\theta_i &= \theta_0 + h (1-q^i)(1-sq^{i+1}) q^{-i},
\\
\theta^*_i &= \theta^*_0 + h^* (1-q^i)q^{-i}
\end{align*}
\noindent for $0 \leq i \leq d$ and
\begin{align*}
\varphi_i &= h h^* q^{1-2i} (1-q^i)(1-q^{i-d-1}),
\\
\phi_i &= h h^* s q^{d+2-2i} (1-q^i)(1-q^{i-d-1})
\end{align*}
for $1 \leq i \leq d$.
We have
\begin{align*}
b_i &= h (1-q^{i-d})
\qquad \qquad (0 \leq i \leq d-1),
\\
c_i &= h s q(1-q^{i})
\qquad \qquad (1 \leq i \leq d)
\end{align*}
\noindent and
\begin{align*}
b^*_0 &= \frac{h^*(1-q^{-d})}{1-sq^2},
\\
b^*_i &= \frac{h^* (1-q^{i-d})(1-s q^{i+1})}{(1-sq^{2i+1})(1-sq^{2i+2})}
\qquad \qquad (1 \leq i \leq d-1),
\\
c^*_i &= \frac{h^* sq^{2i-d} (1-q^{i})(1-sq^{i+d+1})}{(1-sq^{2i})(1-sq^{2i+1})}
\qquad \qquad (1 \leq i \leq d-1),
\\
c^*_d &= \frac{h^* sq^{d} (1-q^{d})}{1-sq^{2d}}.
\end{align*}
\end{example}

\begin{example}
\rm
(Racah). We have
\begin{align*}
\theta_i &= \theta_0 + h i (i+1+s),
\\
\theta^*_i &= \theta^*_0+ h^* i (i+1+s^*)
\end{align*}
\noindent for $0 \leq i \leq d$ and
\begin{align*}
\varphi_i &= h h^* i (i-d-1) (i+r_1) (i+r_2),
\\
\phi_i &= h h^* i (i-d-1) (i+ s^*-r_1)(i+s^*-r_2)
\end{align*}
for $1 \leq i \leq d$, 
with $r_1+r_2= s+s^*+d+1$.
We have
\begin{align*}
b_0 &= \frac{
       -hd(1+r_1)(1+r_2)}
       {2+s^*},
\\
b_i &=\frac{
            h(i-d)(i+1+s^*)(i+1+r_1)(i+1+r_2)
	    }{
	    (2i+1+s^*)(2i+2+s^*)
	    }
\qquad \qquad (1 \leq i \leq d-1),
\\
c_i &= \frac{
                 hi(i+d+1+s^*)(i+s^*-r_1)(i+s^*-r_2)
		 }{
		 (2i+s^*)(2i+1+s^*)
		 }
\qquad \qquad (1 \leq i \leq d-1),
\\
c_d &=   \frac{hd(d+s^*-r_1)(d+s^*-r_2)}{2d+s^*}.
\end{align*}
\noindent To get
$\lbrace b^*_i\rbrace_{i=0}^{d-1}$ and
$\lbrace c^*_i\rbrace_{i=1}^{d}$,
in the above formulas exchange 
$h \leftrightarrow h^*$,
$s \leftrightarrow s^*$
and preserve
$r_1$, $r_2$.
\end{example}

\begin{example}
\rm
(Hahn). We have
\begin{align*}
\theta_i &= \theta_0 + s i,
\\
\theta^*_i &= \theta^*_0+ h^* i (i+1+s^*)
\end{align*}
\noindent for $0 \leq i \leq d$ and
\begin{align*}
\varphi_i &= h^* s i (i-d-1) (i+r),
\\
\phi_i &= -h^* s i (i-d-1) (i+ s^*-r)
\end{align*}
for $1 \leq i \leq d$. 
We have
\begin{align*}
b_0 &= \frac{
       -s d(1+r)}
       {2+s^*},
\\
b_i &=\frac{
            s(i-d)(i+1+s^*)(i+1+r)
	    }{
	    (2i+1+s^*)(2i+2+s^*)
	    }
\qquad \qquad (1 \leq i \leq d-1),
\\
c_i &= \frac{
                 -s i(i+d+1+s^*)(i+s^*-r)
		 }{
		 (2i+s^*)(2i+1+s^*)
		 }
\qquad \qquad (1 \leq i \leq d-1),
\\
c_d &=   \frac{-sd(d+s^*-r)}{2d+s^*}
\end{align*}
\noindent and
\begin{align*}
b^*_i &= h^* (i-d)(i+1+r)
\qquad \qquad (0 \leq i \leq d-1),
\\
c^*_i &= h^* i ( i-d-1-s^*+r)
\qquad \qquad (1 \leq i \leq d).
\end{align*}
\end{example}

\begin{example}
\rm
(dual Hahn). We have
\begin{align*}
\theta_i &= \theta_0+ h i (i+1+s),
\\
\theta^*_i &= \theta^*_0 + s^* i
\end{align*}
\noindent for $0 \leq i \leq d$ and
\begin{align*}
\varphi_i &= h s^* i (i-d-1) (i+r),
\\
\phi_i &= h s^* i (i-d-1) (i+ r-s-d-1)
\end{align*}
for $1 \leq i \leq d$. 
We have
\begin{align*}
b_i &= h (i-d)(i+1+r)
\qquad \qquad (0 \leq i \leq d-1),
\\
c_i &= h i ( i-d-1-s+r)
\qquad \qquad (1 \leq i \leq d)
\end{align*}
\noindent and
\begin{align*}
b^*_0 &= \frac{
       -s^* d(1+r)}
       {2+s},
\\
b^*_i &=\frac{
            s^*(i-d)(i+1+s)(i+1+r)
	    }{
	    (2i+1+s)(2i+2+s)
	    }
\qquad \qquad (1 \leq i \leq d-1),
\\
c^*_i &= \frac{
                 -s^* i(i+d+1+s)(i+s-r)
		 }{
		 (2i+s)(2i+1+s)
		 }
\qquad \qquad (1 \leq i \leq d-1),
\\
c^*_d &=   \frac{-s^*d(d+s-r)}{2d+s}.
\end{align*}
\end{example}

\begin{example}
\rm
(Krawtchouk). We have
\begin{align*}
\theta_i &= \theta_0+ s i,
\\
\theta^*_i &= \theta^*_0 + s^* i
\end{align*}
\noindent for $0 \leq i \leq d$ and
\begin{align*}
\varphi_i &= r i (i-d-1),
\\
\phi_i &= (r-s s^*) i (i-d-1)
\end{align*}
for $1 \leq i \leq d$. 
We have
\begin{align*}
b_i &= r(i-d)/s^*
\qquad \qquad (0 \leq i \leq d-1),
\\
c_i &= i ( r- s s^*)/s^*
\qquad \qquad (1 \leq i \leq d).
\end{align*}
To get
$\lbrace b^*_i\rbrace_{i=0}^{d-1}$
and $\lbrace c^*_i\rbrace_{i=1}^{d}$, in the above formulas
exchange 
$s \leftrightarrow s^*$
and preserve
$r$.
\end{example}

\begin{example} \rm (Bannai/Ito). We have
\begin{align*}
\theta_i &= \theta_0+h(s-1+(1-s+2i)(-1)^i),
\\
\theta^*_i &= \theta^*_0+h^*(s^*-1+(1-s^*+2i)(-1)^i)
\end{align*}
for $0 \leq i \leq d$ and
\begin{align*}
\varphi_i &= h h^* ((-1)^i r_2-2i-r_2)(2i+r_1-d-1+(-1)^{i+d}(r_1+d+1)),
\\
\phi_i &= h h^* (s^*+r_2+(-1)^i(2i-s^*-r_2))
\\ 
& \qquad \qquad \qquad \times\; (d+1-s^*-r_1+(-1)^{i+d}(2i-d-1-s^*-r_1))
\end{align*}
for $1 \leq i \leq d$, with
$r_1+r_2=-s-s^*+d+1$.
\noindent Note that
\begin{align*}
\theta_i &= \begin{cases}
\theta_0 + 2hi, 
     & {\mbox{\rm if $i$ even}}; \\
\theta_0 + 2h(s-i-1), 
     & {\mbox{\rm if $i$ odd}},
\end{cases}
\\
\theta^*_i &= \begin{cases}
\theta^*_0 + 2h^*i, 
     & {\mbox{\rm if $i$ even}}; \\
\theta_0 + 2h^*(s^*-i-1), 
     & {\mbox{\rm if $i$ odd}}
\end{cases}
\end{align*}
\noindent for $0 \leq i \leq d$ and
\begin{align*}
\varphi_i &= 
\begin{cases}
-4hh^*i(i+r_1),    & {\mbox{\rm if $i$ even, $d$ even}}; \\
-4hh^*(i-d-1)(i+r_2),    &{\mbox{\rm if $i$ odd, $d$ even}}; \\
-4hh^*i(i-d-1),    &{\mbox{\rm if $i$ even, $d$ odd}}; \\
-4hh^*(i+r_1)(i+r_2),    &{\mbox{\rm if $i$ odd, $d$ odd}}, 
\end{cases}
\\
\phi_i &= 
\begin{cases}
4hh^*i(i-s^*-r_1),    &{\mbox{\rm if $i$ even, $d$ even}}; \\
4hh^*(i-d-1)(i-s^*-r_2),    &{\mbox{\rm if $i$ odd, $d$ even}}; \\
-4hh^*i(i-d-1),    &{\mbox{\rm if $i$ even, $d$ odd}};\\
-4hh^*(i-s^*-r_1)(i-s^*-r_2) &{\mbox{\rm if $i$ odd, $d$ odd}} 
\end{cases}
\end{align*}
for $1 \leq i \leq d$.
We have
\begin{align*}
b_i &= h(2i+2+r_2-s^*+(-1)^i(r_2+s^*))
\\
& \qquad \times 
\frac{2i+r_1-d+1-(-1)^{i+d}(r_1+d+1)}
{2(2i+2-s^*)}
\qquad (0 \leq i \leq d-1),
\\
c_i &= -h(2i-r_2-s^* +(-1)^i(r_2+s^*))
\\
& \qquad \times 
\frac{2i-2s^*-r_1+d+1-(-1)^{i+d}(r_1+d+1)}
{2(2i-s^*)}
\qquad (1 \leq i \leq d).
\end{align*}
\noindent Note that
\begin{align*}
b_i &= 
\begin{cases}
\displaystyle{\frac{2h(i-d)(i+1+r_2)}{2i+2-s^*}},
& {\mbox{\rm if $i$ even, $d$ even}}; \\
\displaystyle{\frac{2h(i+1-s^*)(i+1+r_1)}{2i+2-s^*}},
&{\mbox{\rm if $i$ odd, $d$ even}}; \\
\displaystyle{\frac{2h(i+1+r_1)(i+1+r_2)}{2i+2-s^*}},
&{\mbox{\rm if $i$ even, $d$ odd}}; \\
\displaystyle{\frac{2h(i-d)(i+1-s^*)}{2i+2-s^*}},
&{\mbox{\rm if $i$ odd, $d$ odd}} 
\end{cases}
\end{align*}
for $0 \leq i \leq d-1$ and
\begin{align*}
c_i &= 
\begin{cases}
\displaystyle{\frac{-2hi(i-s^*-r_1)}{2i-s^*}},
& {\mbox{\rm if $i$ even, $d$ even}}; \\
\displaystyle{\frac{-2h(i+d+1-s^*)(i-s^*-r_2)}{2i-s^*}},
&{\mbox{\rm if $i$ odd, $d$ even}}; \\
\displaystyle{\frac{-2hi(i+d+1-s^*)}{2i-s^*}},
&{\mbox{\rm if $i$ even, $d$ odd}}; \\
\displaystyle{\frac{-2h(i-s^*-r_1)(i-s^*-r_2)}{2i-s^*}},
&{\mbox{\rm if $i$ odd, $d$ odd}} 
\end{cases}
\end{align*}
for $1 \leq i \leq d$.
\noindent To get
$\lbrace b^*_i\rbrace_{i=0}^{d-1}$ and
$\lbrace c^*_i\rbrace_{i=1}^{d}$,
in the above formulas exchange 
$h \leftrightarrow h^*$,
$s \leftrightarrow s^*$
and preserve
$r_1$, $r_2$, $q$.
\end{example}

\begin{example}
\rm (Orphan). Assume that ${\rm char}(\mathbb F)=2$ and $d=3$.
We have
\begin{align*}
&\theta_1 = \theta_0 + h(1+s), \quad \qquad 
\theta_2 = \theta_0 + h, \qquad \;
\theta_3 = \theta_0 + hs,
\\
&\theta^*_1 = \theta^*_0 + h^*(1+s^*),\qquad 
\theta^*_2 = \theta^*_0 + h^*,   \qquad 
\theta^*_3 = \theta^*_0 + h^*s^*
\end{align*}
\noindent and
\begin{align*}
&\varphi_1 = hh^* r, \quad \qquad \qquad \qquad 
\varphi_2 = hh^*, \qquad 
\varphi_3 = hh^* (r+s+s^*), 
\\
&\phi_1 = hh^* (r+s+ss^*), \qquad 
\phi_2 = hh^*, \qquad 
\phi_3 = hh^* (r+s^*+ss^*). 
\end{align*}
\noindent We have
\begin{align*}
&b_0 = \frac{hr}{1+s^*},
\qquad 
b_1 = \frac{h(1+s^*)}{s^*},
\qquad 
b_2 = \frac{h(r+s+ s^*)}{1+ s^*},
\\
&c_1 = \frac{h(r+s+s s^*)}{1+s^*},
\qquad 
c_2 = \frac{h(1+s^*)}{s^*},
\qquad 
c_3 = \frac{h(r+s^*+ ss^*)}{1+ s^*}.
\end{align*}
\noindent To get
$\lbrace b^*_i\rbrace_{i=0}^{2}$
and $\lbrace c^*_i\rbrace_{i=1}^3$, in the above formulas
exchange
$h\leftrightarrow h^*$,
$s\leftrightarrow s^*$ and preserve $r$.
\end{example}


\bigskip

\noindent Paul Terwilliger \hfil\break
\noindent Department of Mathematics \hfil\break
\noindent University of Wisconsin \hfil\break
\noindent 480 Lincoln Drive \hfil\break
\noindent Madison, WI 53706-1388 USA \hfil\break
\noindent email: {\tt terwilli@math.wisc.edu }\hfil\break

\end{document}